\documentclass[12pt,a4paper]{amsart}
\usepackage[marginratio=1:1]{geometry}

\usepackage[T1]{fontenc}
\usepackage[utf8]{inputenc}
\usepackage[american]{babel}
\usepackage[draft=false]{hyperref}
\usepackage{enumitem}
\usepackage{graphicx}
\usepackage{verbatim}

\usepackage{xcolor}
\hypersetup{
	colorlinks,
	linkcolor={red!50!black},
	citecolor={blue!50!black},
	urlcolor={blue!80!black}
}

\usepackage{amsfonts}
\usepackage{amssymb}
\usepackage{float}
\makeatletter
\providecommand*{\twoheadrightarrowfill@}{%
	\arrowfill@\relbar\relbar\twoheadrightarrow
}
\providecommand*{\twoheadleftarrowfill@}{%
	\arrowfill@\twoheadleftarrow\relbar\relbar
}
\providecommand*{\xtwoheadrightarrow}[2][]{%
	\ext@arrow 0579\twoheadrightarrowfill@{#1}{#2}%
}
\providecommand*{\xtwoheadleftarrow}[2][]{%
	\ext@arrow 5097\twoheadleftarrowfill@{#1}{#2}%
}
\makeatother

\newcommand{\C}{{\mathfrak C}}

\newcommand{\N}{{\mathbb{N}}}
\newcommand{\R}{{\mathbb{R}}}

\newcommand{\F}{{\mathbb{F}}}
\newcommand{\Z}{{\mathbb{Z}}}

\DeclareMathOperator{\tp}{{tp}}
\DeclareMathOperator{\Th}{{Th}}

\DeclareMathOperator{\cl}{{cl}}

\DeclareMathOperator{\id}{{id}}
\DeclareMathOperator{\aut}{{Aut}}

\DeclareMathOperator{\topo}{{top}}
\DeclareMathOperator{\defi}{{def}}

\DeclareMathOperator{\Sl}{{SL}}

\DeclareMathOperator{\Stab}{{Stab}}
\DeclareMathOperator{\St}{{St}}
\DeclareMathOperator{\Sym}{{Sym}}
\DeclareMathOperator{\sgn}{{sgn}}
\DeclareMathOperator{\graph}{{graph}}
\DeclareMathOperator{\im}{{Im}}
\DeclareMathOperator{\odd}{{odd}}
\DeclareMathOperator{\Core}{{Core}}

\def\lL{\ell}

\def\dD{\mathcal D}
\def\iI{\mathcal I}
\def\Nn{{\mathbb{N}}}
\def\tT{\mathcal{T}}
\def\eE{\mathcal{E}}
\def\fF{\mathcal{F}}

\newtheorem{thm}{Theorem}[section]
\newtheorem{conj}[thm]{Conjecture}
\newtheorem{ques}[thm]{Question}

\newtheorem{lem}[thm]{Lemma}
\newtheorem{fct}[thm]{Fact}
\newtheorem{cor}[thm]{Corollary}
\newtheorem{prop}[thm]{Proposition}

\theoremstyle{remark}
\newtheorem{rem}[thm]{Remark}
\theoremstyle{definition}
\newtheorem{dfn}[thm]{Definition}

\newtheorem{clm*}{Claim}
\newtheorem{ex}[thm]{Example}
\newcounter{claimcounter}[thm]
\newenvironment{clm}{\stepcounter{claimcounter}{\noindent {\textbf{Claim}} \theclaimcounter:}}{}
\newenvironment{clmproof}[1][\proofname]{\proof[#1]}{\endproof}

\title{Amenability, connected components, and definable actions}
\author{Ehud Hrushovski}
\email[E. Hrushovski]{Ehud.Hrushovski@maths.ox.ac.uk}
\address[E.\ Hrushovski]{Mathematical Institute, University of Oxford\\
Andrew Wiles Building\\
Radcliffe Observatory Quarter (550)\\
Woodstock Road\\
Oxford OX2 6GG, UK}

\author{Krzysztof Krupi\'nski}
\email[K.\ Krupi\'{n}ski]{kkrup@math.uni.wroc.pl}
\address[K.\ Krupi\'nski]{
	Instytut Matematyczny, Uniwersytet Wroc\l awski\\
	pl. Grunwaldzki 2/4\\
	50-384 Wroc\l aw, Poland}
\address{
	ORCID (K.\ Krupi\'{n}ski): \href{https://orcid.org/0000-0002-2243-4411}{0000-0002-2243-4411}}
\thanks{The second author is supported by National Science Center, Poland, grants 2015/19/B/ST1/01151, 2016/22/E/ST1/00450, and 2018/31/B/ST1/00357}

\author{Anand Pillay}
\email[A.\ Pillay]{apillay@nd.edu}
\address[A.\ Pillay]{Department of Mathematics, University of Notre Dame\\
	255 Hurley Hall\\
	Notre Dame, IN 46556, USA}
\thanks{The third author was supported by NSF grants DMS-136702,  DMS-1665035, and DMS-1760212}



\date{}

\begin{document}
\maketitle

\begin{abstract}
We study  amenability of definable groups and topological groups, and prove various results, briefly described below.

Among our main technical tools, of interest in its own right,  is  an elaboration on and  strengthening of the Massicot-Wagner version \cite{MaWa}  of  the stabilizer theorem \cite{Hrushovski-approximate}, and also some results about measures and measure-like functions (which we call means and pre-means).

As an application we show that if $G$ is an amenable topological group,  then the Bohr compactification of $G$ coincides with a certain ``weak Bohr compactification'' introduced in \cite{KrPi}. In other words, the conclusion says that certain connected components of $G$ coincide: $G^{00}_{\topo} = G^{000}_{\topo}$. We also prove wide generalizations of this result, implying in particular its extension to a ``definable-topological'' context,  confirming the main conjectures from \cite{KrPi}. 
%
We also introduce $\bigvee$-definable group topologies on a given $\emptyset$-definable group $G$ (including group topologies induced by type-definable subgroups as well as uniformly definable group topologies),  
and prove that the existence of a mean on the lattice of closed, type-definable subsets of $G$ implies (under some assumption) that $\cl(G^{00}_M) = \cl(G^{000}_M)$ for any model $M$.

Secondly, we study the relationship between (separate) definability of an action of a definable group on a compact space (in the sense of \cite{GPP}), weakly almost periodic (wap) actions of $G$ (in the sense of \cite{Ellis-Nerurkar}), and stability.  We conclude that any group $G$ definable in a sufficiently saturated structure is ``weakly definably amenable''  in the sense of \cite{KrPi}, namely  any definable action of $G$ on a compact space supports a $G$-invariant probability measure. This  gives negative solutions to some questions and conjectures raised in \cite{Kr} and \cite{KrPi}. 
Stability in continuous logic will play a role in some proofs in this part of the paper.

Thirdly, we give an example of a $\emptyset$-definable approximate subgroup $X$ in a saturated extension of the group $\F_2 \times \Z$ in a suitable language (where $\F_2$ is the free group in 2-generators) for which 
the $\bigvee$-definable group $H:=\langle X \rangle$ contains no  type-definable subgroup of bounded index.
This refutes a conjecture by Wagner and shows that the Massicot-Wagner approach to prove that a locally compact (and in consequence also  Lie) ``model'' exists for each approximate subgroup does not work in general (they proved in \cite{MaWa} that it works for definably amenable approximate subgroups).

\end{abstract}

\section{Introduction}

The general motivation standing behind this research is to understand relationships between dynamical and model-theoretic properties of definable [topological] groups. 
More specifically, similarly to \cite{KrPi}, in this paper our goal is to understand model-theoretic consequences of various notions of amenability.

The consequences that we consider in this paper are mainly the equalities between the appropriate versions of the connected components $G^{000}$ and $G^{00}$ of a definable group $G$ in various categories (e.g. in the category of topological groups). 

The notions of amenability are those considered in \cite{KrPi}, and they come from certain natural categories of flows which are described below.

Let us first give some motivation for studying model-theoretic connected components of a definable group $G$ (the definitions can be found in Section \ref{Section 1}, or below in topological contexts). One of the main objects of interest in model theory are definable, type-definable, and, more generally, invariant sets. These contexts lead to the naturally defined connected components $G^0$, $G^{00}$, and $G^{000}$, which yield interesting invariants $G/G^0$, $G/G^{00}$, $G/G^{000}$ of the definable group $G$. In an affine copy of $G$ added as a new sort, the orbits under the action of these components are exactly Shelah, Kim-Pillay, and Lascar strong types, respectively, so any example where $G^{00} \ne G^{000}$ yields an example of a non-$G$-compact theory. The component $G^0$ plays a fundamental role in stability theory, while $G^{00}$ takes over its role in generalizations to wider classes of theories (e.g. NIP theories). The groups $G/G^0$ and $G/G^{00}$ are topologically very nice (when equipped with the logic topology, where by definition the closed sets are those whose preimages under the relevant quotient map are type-definable): the former is a profinite and the latter a compact topological group (in fact, they are universal compactifications in appropriate categories). In this way, with a definable group $G$ we associate interesting classical mathematical objects. The rather influential Pillay's conjecture (now theorem) gives more precise information about $G/G^{00}$ in an o-minimal context \cite{HrPePi}. On the other hand, the group $G/G^{000}$ is a more mysterious, canonical object associated with the definable group $G$. It is a quasi-compact (so not necessarily Hausdorff) topological group, and it is Hausdorff if and only if $G^{00}=G^{000}$. In recent years, a new approach to look at $G/G^{000}$ (and more generally at Lascar (or even arbitrary) strong types) was developed, namely to measure the complexity of such quotients using Borel cardinalities in the sense of descriptive set theory (see e.g. \cite{KrPiSo,KaMiSi,KrPiRz}). This also involved very interesting connections with and applications of topological dynamics \cite{KrPiRz,KrRz}. Another remark is \cite[Proposition 2.18]{KrPi}, which describes the quotient $G/G^{000}$ as a universal object (say ``weak definable Bohr compactification'') in a certain category. If $G(M)$ (where $M$ is a model) is equipped with the full structure (i.e. all subsets of all finite Cartesian powers are $\emptyset$-definable), it is a new universal object in topological dynamics, say ``weak Bohr compactification'' of the discrete group $G(M)$. It may lead to an interesting extension of the class of almost periodic functions on $G(M)$, but it has not been studied yet.  These comments also apply to the connected components $G^{00}_{\topo}$ and $G^{000}_{\topo}$ for a topological group $G$. 
In Section \ref{Section 4}, we clarify the connections between the equality $G^{00}=G^{000}$ and definable approximate subgroups. In particular, we explain that examples with $G^{00} \ne G^{000}$ show some limitations of the Massicot-Wagner approach to find suitable ``Lie models'' for arbitrary approximate subgroups. 
We also give an example which shows that this approach does not work in general (which is explained later in this introduction).

Below we briefly recall and explain the contexts (or ``categories'') which we are interested in (including the relevant notions of amenability). More detailed discussions can be found in Section \ref{Section 1}; for the proofs the reader is referred to Section 2 of \cite{KrPi}. But before that, recall some issues concerning definability. A function $f$ from a set $D(M)$ definable in a structure $M$ to a compact (Hausdorff) space $X$ is said to be {\em definable} if the preimages of any two disjoint closed subsets of $X$ can be separated by a definable set; equivalently, $f$ is induced by a continuous map from the type space $S_{D}(M)$ to $X$. Let $G(M)$ be a definable group. In \cite{GPP}, a $G(M)$-flow (i.e. an action by homeomorphisms of $G(M)$ on a compact space $X$) is called {\em definable} if for every $x \in X$ the function $g \mapsto gx$ is a definable map from $G(M)$ to $X$ in the above sense. In fact, this should rather be called a {\em separately} (or {\em elementwise}) {\em definable flow}, and only for simplicity we will further  write ``definable flow''. Recall that in the classical topological situation, if $G$ is a topological group, then a $G$-flow is a jointly continuous action of $G$ on a compact space. 
In the model-theoretic context of a definable group G(M), it is natural to ask, what if anything is the right analogue of a jointly continuous action on a compact space.  One might want to call such an action ``jointly definable''.
Finally, recall that an ambit is a flow with a distinguished point with dense orbit.

\begin{enumerate}
\item {\em Definable context.} Here, $G(M)$ is a group definable in a structure $M$. By $G$ we mean the interpretation of $G(M)$ in the monster model. Then the quotient map $G(M) \to G/G^{00}_M$ turns out to be the {\em definable Bohr compactification} of $G(M)$ (i.e. the universal compactification among the definable group compactifications). The {\em definable amenability} of $G(M)$ means that there exists a left-invariant, Borel probability measure on the $G(M)$-ambit $S_G(M)$. (One has to be careful here: 
$(G(M),S_G(M),\tp(e/M))$ is not in general the universal definable $G(M)$-ambit; it is universal but in the category of $G(M)$-ambits $(G(M),X,x_0)$ such that the map from $G$ to $X$ given by $g \mapsto gx_0$ is definable (note that this holds for the distinguished $x_0$, but not necessarily for every $x \in X$).)

\item {\em Topological context.} Let $H$ be a topological group. Then there is a unique (up to isomorphism) Bohr compactification of $H$. The classical notion of {\em amenability} of $H$ is defined by saying that there exists a left-invariant, Borel probability measure on the universal (topological) $H$-ambit. To recover these notions model-theoretically, we treat $H$ as a group $G(M)$ definable in a structure $M$ in such a way that all open subsets of $H=G(M)$ are definable (e.g. $M=G(M)$ is the group $H$ expanded by predicates for all open subsets of $H$). Then, if we define $G^{00}_{\topo}$ as $\mu G^{00}_M$ (where $\mu$ is the group of infinitesimals), then the quotient map $G(M) \to G/G^{00}_{\topo}$ is exactly the Bohr compactification of $G(M)$. Analogously, $G^{000}_{\topo}$ is defined as $\langle \mu^G \rangle G^{000}_M$. The universal topological $G(M)$-ambit is described as  $S^{\mu}_G(M):= S_G(M)/\!\!\sim_\mu$, where $p \sim_\mu q \iff \mu \cdot p = \mu \cdot q$. 

\item {\em Definable topological context.} Here, $G(M)$ is a group definable in a structure $M$ which is also a topological group. But we do not assume any relationship between the topology on $G(M)$ and the structure $M$ (i.e. the topology need not be definable in any way). 
Considering a monster model of an expansion of $M$ in which all open subsets of $G(M)$ become $\emptyset$-definable, we define $G^{00}_{\defi, \topo}$ as $\mu G^{00}_M$ (where $G^{00}_M$ is computed in the original language); similarly,  $G^{000}_{\defi, \topo} := \langle \mu^G \rangle G^{000}_M$. Then the quotient map $G(M) \to G/G^{00}_{\defi, \topo}$ is the universal definable, continuous compactification of $G(M)$. In order to define the definable topological amenability of $G(M)$, we assume additionally that there is a basis of neighborhoods of the identity in $G(M)$ consisting of definable sets in the original structure $M$. Under this assumption, the $G(M)$-ambit $S^{\mu}_G(M)$ is defined as in the previous item. We say that $G(M)$ is {\em definably topologically amenable} if $S^{\mu}_G(M)$ supports a left-invariant, Borel probability measure. (One has to be careful as in the first item:  $(G(M),S^{\mu}_G(M),\tp(e/M)/\!\!\sim_\mu)$ is not in general the universal definable (jointly continuous) $G(M)$-ambit; it is universal but in the category of (jointly continuous) $G(M)$-ambits $(G(M),X,x_0)$ such that the map from $G$ to $X$ given by $g \mapsto gx_0$ is definable.)

\item {\em Weak definable [topological] amenability} of a definable [resp. topological] group $G(M)$ means that there is a left-invariant, Borel probability measure on the universal definable [resp. jointly continuous] $G(M)$-ambit (see Section \ref{Section 1}).
\end{enumerate}

Notice that the definable topological context is a common generalization of both the definable and the topological context.
We have defined above the notions of amenability in various categories by saying that the universal ambit in a given category supports an invariant, Borel probability measure. Equivalently, one can say that any ambit (or flow) in the given category supports such a measure.



The following statement is Conjecture 0.4 in \cite{KrPi}.

\begin{conj}\label{conjecture: Conjecture 0.4 of KrPi}
Let $G(M)$ be a topological group and assume that the members of a basis of neighborhoods of the identity are definable. If $G$ is definably topologically amenable, then $G^{00}_{\defi,\topo}=G^{000}_{\defi,\topo}$.
\end{conj}

In the definable context described above, this conjecture specializes to the theorem (easily deduced in \cite{KrPi} from \cite{MaWa}) saying that each definably amenable group $G(M)$ satisfies $G^{00}_M=G^{000}_M$. On the other hand, in the topological context, Conjecture \ref{conjecture: Conjecture 0.4 of KrPi} specializes to Conjecture 0.2 from \cite{KrPi} which predicts that whenever $G(M)$ is an amenable topological group, then $G^{00}_{\topo}=G^{000}_{\topo}$. One of the main results of \cite{KrPi} is Theorem 0.5 there saying that Conjecture \ref{conjecture: Conjecture 0.4 of KrPi} is true if $G(M)$ has a basis of open neighborhoods of the identity consisting of definable, open subgroups.

In Subsection \ref{section: main results on components}, we will prove Conjecture \ref{conjecture: Conjecture 0.4 of KrPi} in  full generality (see Corollary \ref{corollary: main conjecture from KrPi}).  
In fact, we obtain much more general results (namely, Theorems \ref{theorem: c4} and \ref{theorem: indeed c4}) than Conjecture \ref{conjecture: Conjecture 0.4 of KrPi}, which do not assume any topology on $G(M)$. The main content of these results can be stated as the following 

\begin{thm}\label{theorem: main theorems in the introduction}
Let $H$ be an $M$-type-definable subgroup of a $\emptyset$-definable group $G$, normalized by $G(M)$. Let $N$ be  the normal subgroup generated by $H$. If $S_{G/H}(M)$ or $S_{H\backslash G}(M)$ carries a $G(M)$-invariant, Borel probability measure, then   
$G^{00}_M  \leq  N G^{000}_M$.
\end{thm}

Conjecture \ref{conjecture: Conjecture 0.4 of KrPi} follows immediately from Theorem \ref{theorem: main theorems in the introduction} applied to $H:=\mu$. 

Similarly to \cite{KrPi}, the proof is based on the Massicot-Wagner argument from \cite{MaWa}, but here we use {\em means} on certain lattices instead of measures on Boolean algebras. Moreover, in Subsection \ref{Subsection 2.3}, we give a less numerical variant of the argument from \cite{MaWa}, using a general notion of largeness, discussed in Subsection \ref{Subsection 2.2}, which coincides with non-forking in stable theories and seems interesting also outside the stable context. 
It is new and essential in the proofs of our main results that we work here in the category of $\bigvee$-positively definable sets.
The proofs of the main results also require some extension results concerning pre-means, means, and measures --- established in Subsections  \ref{Subsection 2.4} and \ref{subsection: means and measures} --- which additionally yield several corollaries concerning model-theoretic ``absoluteness'' (e.g. the existence of a $G(M)$-invariant, Borel probability measure on $S_{G/H}(M)$ does not depend on the choice of the model $M$) and may prove to be useful also in other situations. 
In Subsection \ref{Subsection 2.5}, we apply these kind of arguments to topological groups equipped with the so-called {\em $\bigvee$-definable group topologies} (including group topologies induced by type-definable subgroups as well as uniformly definable group topologies). The key property of a $\bigvee$-definable group topology on a $\emptyset$-definable group $G$ is that for any model $M$ the group $G(M)$ is also a topological group. We prove (using our version of the Massicot-Wagner theorem) that the existence of a left-invariant mean on the lattice of closed, type-definable subsets of the group $G=G(M^*)$ (where $M^*\succ M$ is a monster model and $G$ is a $\emptyset$-definable group) equipped with a $\bigvee$-definable group topology, such that the projections of closed, type-definable sets are closed, implies that $\cl(G^{00}_M)=\cl(G^{000}_M)$, where $\cl$ denotes closure with respect to the $\bigvee$-definable topology; this is  Proposition \ref{proposition: normal2}. 

We already recalled the notion of {\em definable action} of a definable group $G(M)$ on a compact space as well as the notion of weak definable topological amenability.
The following generalization of Conjecture \ref{conjecture: Conjecture 0.4 of KrPi} is stated as Conjecture 0.3 of \cite{KrPi}.

\begin{conj}\label{con: the most general}
Let $G(M)$ be a topological group definable in an arbitrary structure $M$. If $G$ is weakly definably topologically amenable, then $G^{00}_{\defi,\topo}=G^{000}_{\defi,\topo}$.
\end{conj}

In Section \ref{Section 3}, we will refute this conjecture by showing that it is already false in the ``discrete case''. 
In fact, we show 

\begin{thm} 
Suppose $M$ is $\omega_{1}$-saturated. Then $G(M)$ is weakly definably amenable: for any definable action of $G(M)$ on a compact space $X$, $X$ supports a $G(M)$-invariant, Borel probability measure.
\end{thm}

Our methods are as interesting as the refutation of the conjecture: under the saturation assumption, definable actions are weakly almost periodic, so  support invariant measures. Our proofs involve stable group theory in a continuous logic setting.   This will also give us a negative answer to the question stated in \cite[Problem 4.11(1)]{Kr}, namely whether the assignment $S_G(M)/E \to G/G^{000}_M$ given by $\tp(a/M)/E \mapsto a/G^{000}_M$ is well-defined, where $E$ is the equivalence relation on $S_G(M)$ such that $S_G(M)/E$ is the universal definable $G(M)$-ambit. In \cite[Proposition 4.10]{Kr}, it was noted that an analogous assignment to $G/G^{00}_M$ is a well-defined continuous semigroup epimorphism (with the natural semigroup structure on $S_G(M)/E$ coming from the fact that this is the universal definable $G(M)$-ambit). We also provide a description of the universal definable $G(M)$-ambit as the ``Gelfand space''  of the algebra of stable continuous functions from $S_G(M)$ to $\mathbb{R}$, and describe the universal minimal definable $G(M)$-flow as $G/G_{M}^{00}$.  In this section, we also discuss definable actions when $M$ is not  necessarily saturated, and make the connection between weakly almost periodic actions and continuous logic stability {\em in a model}.

While in Sections \ref{Section 2} and \ref{Section 3} we focused on definable groups and their various connected components, in Section \ref{Section 4} we study definable approximate subgroups and connected components in this context.
Recall that a symmetric subset $X$ of a group $G$ is called an {\em approximate subgroup} if $X^2:=XX$ is contained in finitely many left translates of $X$ by elements of $G$; if this number of translates equals $K$, we say that $X$ is a {\em $K$-approximate subgroup}.\label{page: approximate subgroup} The problem of classifying all approximate subgroups has became one of the  central topics in multiplicative combinatorics in the last 15 years, with essential contributions coming from model theory. Breuillard, Green and Tao classified all finite approximate subgroups \cite{BrGrTa}; the problem of classification of all infinite approximate subgroups remains open. One of the key steps in the classification of finite approximate subgroups was finding a so-called ``Lie model'' and for that a ``locally compact model'',  and this was first done in \cite{Hrushovski-approximate} using model theory (in particular, a version of the stabilizer theorem). 
A bit more precisely, in order to prove the classification of finite approximate subgroups, it is enough to work  with an ultraproduct $X$ of some finite approximate subgroups, so with a {\em pseudofinite} approximate subgroup. Such $X$ is definably amenable, i.e. there is a left invariant, finitely additive measure $\mu$ on definable subsets of $H:=\langle X \rangle$ with $\mu(X)=1$. 
 We will not define what a ``locally compact model'' for $X$ is, 
 but describe a special case where it certainly exists.
Working in a sufficiently saturated model, if $H$ contains a  $\emptyset$-type-definable, bounded index  subgroup,
the smallest such subgroup is denoted $H^{00}_\emptyset$; see \cite{HrPePi}, \S 7.  
Then the quotient $H/H^{00}_\emptyset$ is such a model (where $H/H^{00}_\emptyset$ is equipped with the logic topology, i.e. a set is closed iff its preimage in $H$ intersected with any $X^n$ is type-definable). 
The existence of $H^{00}_\emptyset$ 
was proved in 
 \cite{Hrushovski-approximate} assuming the existence of a definable invariant measure,  and Massicot and Wagner \cite{MaWa}     generalized it to the amenable case \cite{MaWa}, removing the definability assumption on the measure.
In order for $H/H^{00}_\emptyset$ to be an appropriate ``locally compact model'', it was also important that $H^{00}_\emptyset \subseteq X^8$ (in fact, it would be satisfactory to have it for some number $m$ in place of $8$ which depends only on $K$ for $X$ ranging over $K$ approximate subgroups). Strictly speaking, the connected component obtained in \cite{MaWa} was even contained in $X^4$, but it was defined over some small set of parameters $A$ rather than over $\emptyset$. However, by \cite[Theorem 5.2 and Corollary 5.14]{Mas}, we know that $H^{00}_A$ exists if and only if $H^{00}_\emptyset$ exists, and if $H^{00}_A\subseteq X^m$, then $H^{00}_\emptyset \subseteq X^{2m}$. Keeping this in mind, without loss of generality, we can  work with $H^{00}_\emptyset$.

In 2013, Wagner conjectured that whenever $X$ is a $\emptyset$-definable $K$-approximate subgroup in a sufficiently saturated model (note that we can always pass to such a model, as being a $K$-approximate subgroup is preserved under elementary extensions), then $H^{00}_\emptyset$ exists. And, moreover, there is an integer $m$ depending only on $K$ (maybe even $m=8$) such that $H^{00}_\emptyset \subseteq X^m$. This appears in  \cite[Conjecture 0.15]{Mas} and it was the main problem motivating \cite{MaWa} and \cite{Mas}. As discussed above, the conjecture was proved in \cite{MaWa} in the definably amenable case.

In Section \ref{Section 4}, we refute this conjecture. This shows that in order to find a ``locally compact model'' in general (without any amenability assumptions on the approximate subgroup $X$), one would have to use a completely different approach from the one involving $H^{00}_\emptyset$.


Observe that by compactness, if $H^{00}_\emptyset$ exists, then it is contained in $X^m$ for some $m \geq 1$. In Section \ref{Section 4}, we describe a quasi-homomorphism $f \colon \F_2 \to \Z$ such that for $X$ being the graph of $f$ computed in a sufficiently saturated elementary extension of $((\F_2,\cdot), (\Z,+), f)$ the last condition fails, and so $H^{00}_\emptyset$ does not exist. In fact, we show more about this example. Let $H^{000}_\emptyset$ be defined as the smallest invariant, bounded index subgroup of $H$ (which always exists in a $\bigvee$-definable group $H$, in contrast with $H^{00}_\emptyset$). In our example, we show that for every $m \geq 1$, $H^{000}_\emptyset$ is not contained in $X^m$ (see Proposition \ref{proposition: H000 not contained in Xm}).

Recall that a subset $Y$ of $H$ is {\em generic} if finitely many left  translates of it covers $X$. By compactness, this is equivalent to saying that boundedly (or even countably) many translates of $Y$ covers $H$.
In Proposition \ref{proposition: three equivalent conditions}, we will also show that the existence of $m \in \N$ such that $H^{00}_\emptyset \subseteq Y^m$ [resp.  $H^{000}_\emptyset \subseteq Y^m$] for every $\emptyset$-definable, generic, symmetric  $Y \subseteq H$ is equivalent to the equality $H^{000}_\emptyset= H^{00}_\emptyset$. Thus, any example where $H^{000}_\emptyset \ne H^{00}_\emptyset$ (\cite{Gismatullin-Krupinski} yields many such examples with definable $H$, e.g. $X=H$ being a sufficiently saturated extension of the free group $\F_2$ considered with the full structure) yields a family of $\emptyset$-definable approximate subgroups $Y_m$, $m \in \N$, (which are generic subsets of the definable group in question, so generate $\emptyset$-definable, finite index subgroups in finitely many steps) such that $(H_m)^{00}_\emptyset \nsubseteq Y_m^m$, where $H_m: = \langle Y_m \rangle$.\\


This paper contains the material in Sections 2, 3, and 5 of our preprint ``Amenability and definability''. Following the advice of editors and referees we have divided that preprint into two papers, the current paper being the first.

\section{Some notions and definitions}\label{Section 1}

We recall here the model-theoretic definitions of certain components of groups in various categories, and also the relevant variants of the notion of amenability; for more details, see Section 2 of \cite{KrPi}. The new notions which we introduce in this paper will appear in the relevant sections.

As usual, by a monster model of a given theory we mean a $\kappa$-saturated and strongly $\kappa$-homogeneous model for a sufficiently large cardinal $\kappa$ (typically, $\kappa>|T|$ is a strong limit cardinal).  Where recall that the (standard) expression ``strongly $\kappa$-homogeneous'' means that any partial elementary map between subsets of the model of cardinality $<\kappa$ extends to an automorphism of the model. 
A set [tuple] is said to be small [short] if it is of bounded cardinality (i.e. $<\kappa$). When $G$ is a $\emptyset$-definable group (in the monster model) and $A$ a (small) set of parameters, then $G^{00}_A$ denotes the smallest $A$-definable subgroup of $G$ of bounded index; and $G^{000}_A$ the smallest $A$-invariant subgroup of $G$ of bounded index.

Let us now discuss in more details the ``topological context'' from item (2) in the introduction.
Let $G(M)$ be a topological group $\emptyset$-definable in a structure $M$. Assume (from now until Conjecture 1.3) that all open subsets of $G(M)$ are also $\emptyset$-definable. By $G$ we denote the interpretation of $G(M)$ in a monster model $M^*$. Define $\mu$ to be the intersection of all $U=U(M^*)$ with $U(M)$ ranging over all open neighborhoods of the identity. So $\mu$ is the subgroup of infinitesimals of $G$; it is not necessarily normal, but it is normalized by $G(M)$.

\begin{dfn}
1) $G^{00}_{\topo} := \mu \cdot G^{00}_M:=\{ab: a \in \mu, \; b \in G^{00}_M$\}; equivalently, this is the smallest $M$-type-definable subgroup of $G$ of bounded index which contains $\mu$.\\
2) $G^{000}_{\topo}:= \langle \mu^{G}\rangle \cdot G^{000}_M$ (where $\mu^G:=\{gag^{-1}: a \in \mu, \;g \in G\}$ and $\langle \mu^{G}\rangle$ is the subgroup generated by $\mu^G$); equivalently, this is the smallest normal, invariant over $M$ subgroup of $G$ of bounded index which contains $\mu$.
\end{dfn}

It turns out that $G^{00}_{\topo}$ is a normal subgroup of $G$ and the map $G(M) \to G/G^{00}_{\topo}$ is the classical Bohr compactification of $G(M)$ as a topological group (i.e. the universal group compactification). For a description of $G/G^{000}_{\topo}$ as the initial object in a certain category see \cite[Proposition 2.18]{KrPi}. In particular, one gets that both quotients $G/G^{00}_{\topo}$ and $G/G^{000}_{\topo}$ are independent as topological groups (equipped with the logic topology) of the choice of the language (provided that all open subsets of $G(M)$ are $\emptyset$-definable) and of the choice of the monster model in which they are computed. Moreover, the closure of the identity in $G/G^{000}_{\topo}$ is exactly $G^{00}_{\topo}/G^{000}_{\topo}$, so the property $G^{00}_{\topo}=G^{000}_{\topo}$ is also independent of the choice of the language and the monster model.

There is also a model-theoretic description of the universal (left) $G(M)$-ambit as the quotient $S_G^{\mu}(M):= S_G(M)/\!\!\sim_\mu$, where $p \sim_\mu q \iff  \mu \cdot p = \mu \cdot q$ with the distinguished point $\tp(1/M)/\!\!\sim_\mu$ and the action of $G(M)$ given by $g*(\mu \cdot p):=\mu \cdot (g\cdot p)=g\cdot(\mu \cdot p)$. It is clear that this ambit is isomorphic to $S_{\mu\backslash G}(M)$ -- the space of complete types over $M$ of hyperimaginary elements from $\mu\backslash G$.

Recall the classical definition of amenability.

\begin{dfn}
The topological group $G(M)$ is {\em amenable} if for every $G(M)$-flow (equivalently, $G(M)$-ambit) $X$ there is a $G(M)$-invariant, Borel probability measure on $X$; equivalently, there is a $G(M)$-invariant, Borel probability measure on the universal ambit $S_G^{\mu}(M)$.
\end{dfn}

The following is \cite[Conjecture 0.2]{KrPi}.

\begin{conj}\label{conjecture: Conjecture 0.2 of KrPi}
Let $G(M)$ be a topological group. If $G(M)$ is amenable, then $G^{00}_{\topo}=G^{000}_{\topo}$.
\end{conj}

Now, we discuss in more details the more general ``definable topological context'' from item (3) in the introduction, which was studied in Subsection 2.2 of  \cite{KrPi}.
It is a bit subtle, so we try to be precise about the notions and definitions (although a full account is given in \cite{KrPi}).   So we start with an $\mathcal{L}$-structure $M$, and a group $G(M)$ $\emptyset$-definable in $M$.   We assume that $G(M)$ is also a topological group, although this is not necessarily ``seen'' by the structure $M$.  
Let $M'$ be an expansion of $M$ in a language $\mathcal{L}'$ containing $\mathcal{L}$ such that we have predicates for all open subsets of the topological group $G(M)$. Let  $(M')^*\succ M'$ be a monster model  of $\Th(M')$ whose reduct $M^*\succ M $ to $\mathcal{L}$ is also a monster model.  So the dynamics of $G(M)$ as a {\em topological group}  is seen through the model theory of $M'$ and $(M')^{*}$ as discussed earlier in this section.  But we are more interested in what is definable in $M$. So as to avoid too much unnecessary notation, we will rather talk about $M, M^{*}$ and distinguish between definability in $\mathcal{L}$ (which we just call definable)  and definability in the richer language $\mathcal{L}'$. 
$G^{00}_{M}$ and $G^{000}_{M}$ are computed in ${\mathcal L}$, and $S_G(M)$ denotes the space of complete types in the sense of ${\mathcal L}$.

\begin{dfn}
1) $G^{00}_{\defi,\topo} := \mu \cdot G^{00}_M = G^{00}_{\topo} \cdot G^{00}_M$; equivalently, this is the smallest $M$-type-definable (in the sense of $\mathcal{L}$) subgroup of $G$ of bounded index which contains $\mu$.\\
2)  $G^{000}_{\defi, \topo}:= \langle \mu^G \rangle \cdot G^{000}_{M} = G^{000}_{\topo} \cdot G^{000}_{M}$;
equivalently, this is the smallest normal, invariant over $M$ (in the sense of $\mathcal{L}$) subgroup of $G$ of bounded index which contains $\mu$.
\end{dfn}

Note that we need the $\mathcal{L}'$-structure to make sense of $\mu$, and $G^{00}_{\topo}$, etc., although $G^{00}_{\defi,\topo}$ is nevertheless still type-definable over $M$ in $\mathcal{L}$. Observe also that all the connected components  $G^{00}_{\topo}$,  $G^{000}_{\topo}$,  $G^{00}_{\defi,\topo}$, and  $G^{000}_{\defi,\topo}$ depend on $M$ which is suppressed from the notion for simplicity (and because $M$ is in a sense hidden in the subscript $\topo$, as the topology is given only on $G(M)$).

It turns out that  $G^{00}_{\defi,\topo}$ is a normal subgroup of $G$ and the map $G(M) \to G/G^{00}_{\defi, \topo}$ is the (unique up to isomorphism) universal compactification of $G(M)$ among definable (in the sense of $\mathcal{L}$), continuous group compactifications of $G(M)$. 

Note that the definitions  of  $G^{00}_{\defi,\topo} := \mu \cdot G^{00}_M$ and $G^{000}_{\defi, \topo}:= \langle \mu^G \rangle \cdot G^{000}_{M}$ make sense even in the wider context when ${\mathcal L'}$ is any extension of ${\mathcal L}$ such that all members of some basis of neighborhoods of the identity in $G(M)$ are definable in ${\mathcal L'}$ with parameters from $M$ (where $\mu$ is defined as the intersection of all $U=U(M^*)$ with $U(M)$ ranging over the definable in ${\mathcal L'}$ neighborhoods of the identity); the difference is that now more monster models are allowed, because we do not require ${\mathcal L'}$ to contain predicates for all open subsets of $G$. By a standard argument, we get that the quotients $G/{G}^{00}_{\defi, \topo}$ and $G/{G}^{000}_{\defi, \topo}$ do not depend on the choice of both the language $\mathcal{L}'$ and the monster model in which they are computed. The property $G^{00}_{\defi,\topo} = G^{000}_{\defi,\topo}$ is also independent of the choice of $\mathcal{L}'$ and the monster model, which follows directly from definitions.

\begin{rem}\label{rem: equality of various components}
i) If $G(M)$ is discrete, then $G^{000}_{\defi, \topo}=G^{000}_{M}\geq G^{000}_{\topo}$ and $G^{00}_{\defi, \topo}=G^{00}_{M}\geq G^{00}_{\topo}$.\\
ii) If all open subsets of $G(M)$ are definable in $M$ (in the language $\mathcal{L}$), then $G^{000}_{\defi, \topo}=G^{000}_{\topo} \geq G^{000}_{M}$ and $G^{00}_{\defi, \topo}=G^{00}_{\topo} \geq G^{00}_{M}$.
\end{rem}

Recall that a group $G(M)$ definable in $M$ is {\em definably amenable} if and only if there is a left-invariant, Borel probability measure on $S_G(M)$. 
In order to give a suitable generalization of this notion in the ``definable topological context'', one needs to assume that all members of some basis of (not necessarily open) neighborhoods of the identity in $G(M)$ are definable in $M$ (in the original language $\mathcal{L}$). In \cite{KrPi}, we assumed more, namely, that there is such a basis consisting of open neighborhoods of the identity, but in the more general context everything works in the same way. 
In particular, $\mu$ is well-defined, and  $S_G^{\mu}(M)$ defined as above is still a $G(M)$-ambit. The following definition was proposed in \cite[Section 3]{KrPi}.

\begin{dfn} Assume that all members of some basis of neighborhoods of the identity in the topological group $G(M)$ are definable in $M$ (in $\mathcal{L}$). 
We say that $G(M)$ is {\em definably topologically amenable} if there exists a left-invariant, Borel probability measure on the $G(M)$-ambit $S_G^{\mu}(M)$.
\end{dfn}



Conjecture \ref{conjecture: Conjecture 0.4 of KrPi} recalled in the introduction is the main conjecture of \cite{KrPi}.
As was recalled in the introduction, one of the main results of \cite{KrPi} was \cite[Theorem 0.5]{KrPi} saying that Conjecture \ref{conjecture: Conjecture 0.4 of KrPi} is true if $G(M)$ has a basis of open neighborhoods of the identity consisting of definable, open subgroups. This implies Conjecture \ref{conjecture: Conjecture 0.2 of KrPi} for groups possessing a basis of open neighborhoods of the identity consisting open subgroups.
In Subsection \ref{section: main results on components} of this paper (see Corollary \ref{corollary: main conjecture from KrPi}), we prove Conjecture \ref{conjecture: Conjecture 0.4 of KrPi} (and so also Conjecture \ref{conjecture: Conjecture 0.2 of KrPi}) in its full generality. 

The definition of amenability of a topological group is by saying that there is a well-behaved measure on the universal topological ambit. The definitions of definable amenability or definable topological amenability are by saying that there is a well-behaved measure on the $G(M)$-ambits $S_G(M)$ or $S_G^{\mu}(M)$, respectively. But these ambits are not universal in any of the categories of ambits considered in \cite{KrPi} 
(they are universal ambits in some other categories described in parentheses in items (1) and (3) in the introduction). So based on \cite{Kr}, we proposed in \cite{KrPi}  more general notions of amenability, which we recall now.

As was pointed out in \cite{Kr}, there is a unique closed equivalence relation $E$ on $S_G(M)$ such that $S_G(M)/E$ is the universal definable $G(M)$-ambit; a description of $E$ can be found in Section 3 of \cite{Kr}.  In \cite[Subsection 2.2]{KrPi}, we described a closed equivalence relation $E_1$ on $S_G(M)$ such that $S_G(M)/E_1$ is the universal definable topological $G(M)$-ambit (where $G(M)$ is a topological group definable in $M$).

\begin{dfn}\label{def: definition of weak definable amenability}
1) We say that $G(M)$ is {\em weakly definably amenable} if there exists a left-invariant, Borel probability measure on the universal definable $G$-ambit, i.e. on $S_G(M)/E$.\\
2) We say that $G$ is {\em weakly definably topologically amenable} if there exists a left-invariant, Borel probability measure on the universal definable topological $G$-ambit, i.e. on $S_G(M)/E_1$.
\end{dfn}



Conjecture \ref{con: the most general} from the introduction is the most general conjecture of \cite{KrPi}. In Section \ref{Section 3}, we will show that it is false, even in the case when $G(M)$ is discrete (i.e. working in the definable category). In Subsection 3.4 of \cite{KrPi}, a weaker form of this conjecture was proposed. Namely, let ${G}^{000+}_{\defi, \topo}$ be the normal subgroup generated by all products $ab^{-1}$ for $(a,b) \in E_1'$, where $aE_1'b \iff \tp(a/M) E_1 \tp(b/M)$. It is $M$-invariant, and by Proposition 3.10 of \cite{Kr}, we easily get that ${G}^{000}_{\defi,\topo} \leq {G}^{000+}_{\defi,\topo}\leq {G}^{00}_{\defi,\topo}$. 

\begin{conj}
Let $G(M)$ be a topological group definable in an arbitrary structure $M$. If $G$ is weakly definably topologically amenable, then $G^{00}_{\defi,\topo}=G^{000+}_{\defi,\topo}$.
\end{conj}

At first glance, it seems that this conjecture should be reachable by the methods of Section \ref{Section 2}, but we do not quite see how to prove it.


\section{Means and connected components}\label{Section 2}

The main goal of this section is to prove the equality of various connected components under the existence of a suitable measure or mean. In particular, we will prove Conjecture \ref{conjecture: Conjecture 0.4 of KrPi}. As mentioned in the introduction, this conjecture was proved in \cite{KrPi} but under the stronger assumption that there is a basis of open neighborhoods of the identity consisting of definable open subgroups. Similarly to \cite{KrPi}, our proofs are based on the idea of the proof of the Massicot-Wagner version of the stabilizer lemma. Our key tricks to deal with the general case will be using means instead of measures (so something like measures but defined only on certain lattices of subsets), positively $\bigvee$-definable sets, and a notion of largeness. As to the Massicot-Wagner result, we will prove a variant of it (see Proposition \ref{proposition: variant of Massicot-Wagner} and Corollary \ref{corollary: good for applications}) which is applicable to various situations. The main results of this section are contained in Subsection \ref{section: main results on components}. In Subsection \ref{Subsection 2.5}, we study groups equipped with $\bigvee$-definable group topologies, also proving that existence of a mean on the appropriate lattice of subsets implies equality of the closures of the appropriate connected components.

\subsection{$\bigvee$-definable sets}\label{Subsection 2.1}

Let $T$ be any (complete) theory, $M \models T$, and $\C$ be a monster model of $T$. By a [type-]definable set we usually mean a set which is [type-]definable with parameters in $\C$. We can identify it with the corresponding formula [or set of formulas]. We will be often talking about sets which are $A$-type-definable, so using parameters from  a set $A$. One can often incorporate parameters into the language and work over $\emptyset$, e.g. in this and in the next subsection we work with $\emptyset$-definability, but sometimes parameters are essential (e.g. in Proposition \ref{proposition: variant of Massicot-Wagner} and the applications to Theorem \ref{theorem: c4} and Proposition \ref{proposition: normal2}).

We will now discuss the category of $\bigvee$-positively definable sets. Although all the fundamental observations that we make are valid for the category of $\bigvee$-definable sets, we focus on positive definability, as it is crucial in the main applications, i.e. in the proofs of Theorems \ref{theorem: c4} and \ref{theorem: indeed c4}.

By the category of {\em $\bigvee$-positively definable sets}, we mean the category whose objects are expressions of the form $\bigvee_{i \in \omega}  D_i$, where $D_0 \subseteq D_1 \subseteq \dots$ are positively definable sets, where two such expressions are considered to be equal if they agree in any model of $T$ (equivalently, in the monster model; so working in the monster model, any object can be identified with the corresponding subset of the model). A morphism $F \colon \bigvee_{i \in \omega}  D_i \to \bigvee_{i \in \omega}  E_i$ is a collection of definable functions $F_i \colon D_i \to E_{j_i}$, where $i$ ranges over $\omega$ and $j_i$ is some index in $\omega$, such that $\bigcup F_i$ is a well-defined function, and two such collections of functions are identified if they yield the same function from  $\bigvee_{i \in \omega}  D_i(M)$ to $\bigvee_{i \in \omega}  E_i(M)$ 
for every (equivalently, some) model $M$. 
We write $\bigvee_{i \in \omega}  D_i \subseteq \bigvee_{i \in \omega}  E_i$ if this holds in every model (equivalently, in $\C$); this is equivalent to saying that for every $i$ there is $j_i$ such that $D_i \subseteq E_{j_i}$. Whenever $\bigvee_{i \in \omega}  D_i(\bar x)$ is $\bigvee$-positively definable and $\bar a$ is a tuple of parameters, we say that $\bigvee_{i \in \omega}  D_i(\bar a)$ {\em holds} if there is $i$ such that $M \models D_i(\bar a)$ for some (any) model 
$M \prec \C$ containing $\bar a$.

In fact, we can consider any $\bigvee_{i \in I} D_i$ for a countable set $I$ and positively definable sets $D_i$, as then one can replace $I$ by $\omega$ and the $D_i$'s by the unions of initial sets $D_i$, $i <n$. We will be doing this freely without mentioning it. Also, one could extend the context to uncountable sets $I$, but countable families are sufficient for the purpose of our main theorems.  


Recall that a subset $D$ of a group $G$ is said to be {\em (left) generic} if finitely many left translates of it cover $G$; following \cite[Definition 3.1]{Gis}, $D$ is said to be {\em thick} if there is $n$ such that for every $g_1,\dots,g_n \in G$ there is $i<j$ such that $g_j^{-1}g_i \in D$. It is clear that each thick subset of $G$ is generic. As to the converse, if $D\subseteq G$ is generic, then $D^{-1}D$ is thick.


Let $G$ be a group definable in $T$. For a positively definable set $D(x,\bar y) \vdash G(x)$, by $(\exists^{\textrm{gen}} x) D(x,\bar y)$ we mean the $\bigvee$-positively definable set $\bigvee _{l \in \omega} (\exists^{l\textrm{-gen}}x) D(x,\bar y)$, where  
$$(\exists^{l\textrm{-gen}}x)D(x,\bar y) \;:= \; (\exists x_1,\dots,x_l)(\forall z) \bigvee_{i=1}^l D(x_iz,\bar y).$$
(Formally, the quantifiers in the last formula are restricted to $G$; if $(G,\cdot)$ is a sort, then this formula is clearly positive, so $(\exists^{\textrm{gen}} x) D(x,\bar y)$ is $\bigvee$-positively definable. Abusing terminology by allowing in positive formulas both the group operation on $G$ and quantification over $G$, we can say that $(\exists^{\textrm{gen}} x) D(x,\bar y)$ is $\bigvee$-positively definable also for any definable $G$.)

In particular, for any parameters $\bar b$, $(\exists^{\textrm{gen}} x) D(x,\bar b)$ holds iff $D(M,\bar b): = \{ a \in G(M) : M \models D(a,\bar b) \}$ is generic in $G(M)$ for some [any] model $M$ containing $\bar b$. For a $\bigvee$-positively definable set $D(x,\bar y)=\bigvee_{i \in \omega}  D_i(x,\bar y)$ such that $D(x,\bar y) \vdash G(x)$ by $(\exists^{\textrm{gen}} x) D(x,\bar y)$ we mean the $\bigvee$-positively definable set $\bigvee _{i \in \omega} (\exists^{\textrm{gen}}x) D_i(x,\bar y)$.
In particular, for any parameters $\bar b$, $(\exists^{\textrm{gen}} x) \bigvee_{i \in \omega}  D_i(x,\bar b)$ holds iff for some $i$ the set $D_i(M,\bar b):= \{ a \in G(M) : M\models D_i(a,\bar b) \}$ is generic in $G(M)$ for some [any] model $M$ containing $\bar b$. 
Working in the monster model, this is equivalent to saying that $D(\C,\bar b):= \{ a \in G(\C): \C \models D(a,\bar b)\} = \bigcup_{i \in \omega} D_i(\C,\bar b)$ is generic in $G(\C)$ (but this equivalence need not be true for a non $\aleph_0$-saturated model).

Analogous definitions apply when we replace ``generic'' by ``thick''. The only difference is, of course, that the displayed formula above  is now the following
$$(\exists^{l\textrm{-thick}}x)D(x,\bar y) \; := \; (\forall x_1,\dots,x_l) \bigvee_{i<j} D(x_j^{-1}x_i,\bar y).$$

\subsection{A largeness notion}\label{Subsection 2.2}

Throughout, $G$ is a group acting on $X$. We work in the language of group actions, $(G,\cdot,X,\cdot,\dots)$.   ($\cdot$ refers both to the group operation and the action, and $\dots$ to possible additional structure.) In the particular case when $G$ acts on itself via left translations, the results which we will obtain for $(G,\cdot,X,\cdot,\dots)$ transfer  automatically to the corresponding statements in the language of groups $(G,\cdot,\dots)$ (i.e. without the extra sort for $X$), just identifying $X$ with $G$.

We define a  largeness notion $\lL_k$ for subsets of  $X$, resembling ``rank $\geq k$'' for certain model-theoretic ranks. In fact, we define two largeness notions $\lL^{\textrm{gen}}_k$ and $\lL^{\textrm{thick}}_k$. The stronger notion $\lL^{\textrm{thick}}_k$ corresponds to non-forking in stable theories (see Remark \ref{remark: largeness = nonforking}).  For our purposes, 
the two notions work in the same way, so later we will just write $\lL_k$. It would be interesting to further investigate  $\lL^{\textrm{gen}}_k$ and $\lL^{\textrm{thick}}_k$ (and  variants)  for unstable theories.

 
In what follows, we deal with $\lL^{\textrm{gen}}_k$, but  everything works also for the  analogously defined $\lL^{\textrm{thick}}_k$.

\begin{dfn}   Let $Y(x,\bar y) \subseteq X(x)$ be a $\bigvee$-positively definable set $\bigvee_i Y_i(x,\bar y)$.   
\begin{enumerate} 
\item   $\lL^{\textrm{gen}}_0(Y(x,\bar y))$ is the $\bigvee$-positively definable set $\bigvee_i (\exists x)Y_i(x,\bar y)$ .     
\item For $k>0$, $\lL^{\text{gen}}_k(Y(x,\bar y))$ is the $\bigvee$-positively definable set in variables  $\bar y$
$$(\exists^{\textrm{gen}} z) \lL^{\textrm{gen}}_{k-1}(Y(x,\bar y) \cap Y(z^{-1}x,\bar y)).$$
\end{enumerate}
\end{dfn} 

In particular, using terminology from Subsection \ref{Subsection 2.1}, for a $\bigvee$-positively definable set $Y=Y(x) \subseteq X(x)$ we have a well-defined meaning of ``$\lL^{\text{gen}}_k(Y)$ holds'' (working in a given theory). 
Namely,  $\lL^{\textrm{gen}}_0(Y)$ holds iff $Y \neq \emptyset$, and $\lL^{\text{gen}}_k(Y)$ holds iff $\{g \in G :  \lL^{\textrm{gen}}_{k-1}(Y \cap gY) \;\, \mbox{holds}\}$ is generic as a $\bigvee$-definable set, i.e. writing it as a countable increasing union of definable sets, one of them must be generic.
The word ``hold'' will be often skipped from now on.


\begin{rem}\label{remark: lL is vee definable}   $\lL^{\textrm{gen}}_k(Y(x,\bar y))$ can be expressed by a disjunction of positive, translation invariant formulas $\psi_j(\bar y)$ of the language $(G,\cdot, X,\cdot,Y_i)_i$, where $Y= \bigvee Y_i$. (Here, by a translation invariant formula we mean a formula $\psi_{(Y_i: i<\omega)}(\bar y)$ depending on the  $Y_i$'s (and with variables $\bar y$ appearing only in the $Y_i(x,\bar y)$'s) such that  $\psi_{(Y_i: i<\omega)}(\bar y)$ is equivalent to $\psi_{(gY_i: i<\omega)}(\bar y)$ for any $g \in G$.)
\end{rem}
\begin{proof} 
The proof is by induction on $k$. Clearly $\lL^{\textrm{gen}}_0(Y(x,\bar y))$ can by expressed as $\bigvee_i (\exists x ) Y_i(x,\bar y)$ which does the job. Now, suppose that $\lL^{\textrm{gen}}_k(Y(x,\bar y))$ can be expressed as $\bigvee_{j \in \omega} \psi_{j,(Y_i:i<\omega)}(\bar y)$, where each $\psi_{j,(Y_i:i<\omega)}(\bar y)$ is positive and translation invariant. Then $\lL^{\textrm{gen}}_{k+1}(Y(x,\bar y))$ can be expressed as  $\bigvee_{j,l \in \omega} (\exists^{l \textrm{-gen}} z) \psi_{j,(Y_i \cap zY_i:i<\omega)}(\bar y, z)$. By induction hypothesis, it is clear that each $(\exists^{l \textrm{-gen}} z) \psi_{j,(Y_i \cap zY_i:i<\omega)}(\bar y,z)$ is positive and translation invariant.
\end{proof}





It is also easy to express the $\lL^{\textrm{gen}}_k$ directly,  e.g. $ \lL^{\textrm{gen}}_2(Y) \equiv \bigvee_{i,l,l'}  \lL^{\textrm{gen},l,l'}_2(Y_i)$, where
\[\lL^{\textrm{gen},l,l'}_2(Y_i) \equiv (\exists^{l\textrm{-gen}} z)(\exists^{l'\textrm{-gen}} z') \lL^{\textrm{gen}}_0( Y_i \cap zY_i \cap z'Y_i \cap z'z Y_i). \]

\begin{rem}\label{independence of largness of the language}
Let $L$ and $L'$ be two languages on a given universe, expanding the language of the action of $G$ on $X$. Suppose $Y =Y(x) \subseteq X(x)$ and $Y= \bigvee Y_i$ with $Y_i$ definable in both $L$ and $L'$. Then $\lL^{\textrm{gen}}_k(Y)$ holds with respect to $\Th^L(M)$ if and only if it holds with respect to $\Th^{L'}(M)$.
\end{rem}

Let $Y =Y(x) \subseteq X(x)$ and $Y= \bigvee Y_i$.
By Remark \ref{remark: lL is vee definable}, $\lL^{\textrm{gen}}_k$ is translation invariant, i.e. $\lL^{\textrm{gen}}_k(Y) \iff \lL^{\textrm{gen}}_k(hY)$.
 
 Define 
$$\St_{\lL^{\textrm{gen}}_k} (Y) := \{g: \lL^{\textrm{gen}}_k(gY \cap Y) \}.$$   
This is an operator from the class of $\bigvee$-positively definable sets to itself.   
Note that 
  $\lL^{\textrm{gen}}_{k+1}(Y)$ holds iff $\St_{\lL^{\textrm{gen}}_{k}}(Y)$ is generic as a $\bigvee$-definable set (which remember means that writing $\St_{\lL^{\textrm{gen}}_{k}}(Y)$  as a countable increasing union of definable sets $U_{n}$ say, one of the $U_{n}$'s is generic).  By Remark \ref{remark: lL is vee definable}, we get

\begin{rem}\label{remark: St is a union f pos. def. symmetric sets}
$S:=\St_{\lL^{\textrm{gen}}_k} (Y)$ satisfies $S=S^{-1}$. If additionally $\lL^{\textrm{gen}}_k (Y)$, then $1 \in S$, so $S$ is symmetric. 
Even more: $S$ can be expressed by a disjunction of positive formulas (with parameters over which $Y$ is defined) which are closed under inversion; if additionally  $\lL^{\textrm{gen}}_k (Y)$, then these formulas can be chosen to contain 1, so they are symmetric.
\end{rem}

The next basic remark shows in particular that (working in a given theory) the $\bigvee$-definable set $\lL^{\textrm{gen}}_k (Y(x,\bar y))$ does not depend on the choice of the presentation of $Y$ as a union of positively definable sets.

\begin{rem}
(1) Let $Y(x,\bar y) \subseteq Y'(x,\bar y) \subseteq X(x)$ be $\bigvee$-positively definable sets. Then $\lL_k(Y(x,\bar y)) \subseteq \lL_k(Y'(x,\bar y))$ (as $\bigvee$-definable sets in variables $\bar y$). In particular, if the tuple $\bar y$ is empty, then ``$\lL_k(Y)$ holds'' implies ``$\lL_k(Y')$ holds''.\\
(2)  Let $Y(x,\bar y) \subseteq X(x)$ be a $\bigvee$-positively definable set. Then for every $k \in \omega$, $\lL^{\textrm{thick}}_k(Y(x,\bar y)) \subseteq \lL^{\textrm{gen}}_k(Y(x,\bar y))$. In particular, if the tuple $\bar y$ is empty, then ``$\lL^{\textrm{thick}}_k(Y)$ holds'' implies ``$\lL^{\textrm{gen}}_k(Y)$ holds''.
\end{rem}

The whole discussion in Subsection \ref{Subsection 2.1} and above goes through working with $\bigvee$-definable sets instead of $\bigvee$-positively definable sets. However, the above observation that the operator $\St_{\lL^{\textrm{gen}}_k}$ preserves $\bigvee$-positive definability will be crucial in our applications.

As already mentioned, the above definitions and facts have obvious counterparts with ``generic'' replaced by ``thick''. In the rest of the paper, we can work with any of these two versions, so we will be writing $\lL$ in place of $\lL^{\textrm{gen}}$ or $\lL^{\textrm{thick}}$. An exception is the next remark which holds for $\lL^{\textrm{thick}}$.

\begin{rem}\label{remark: largeness = nonforking} 
When $G=\aut(\C)$ is the automorphism group of a monster model of a stable theory $T$, and $Y$ is definable (over $\C$), then $\lL^{\textrm{thick}}_k(Y)$ holds for all $k \in \omega$ if and only if $Y$ does not fork over $\emptyset$. (Here, $\lL^{\textrm{thick}}_k(Y)$ is computed in $(G,\cdot, \C, \cdot)$ with $\C$ equipped with its original stable structure.)
\end{rem}

\begin{proof}
Let $T = \Th(\C)$. 
The structure in which we will be working is $(G,\cdot, \C, \cdot)$, with $\C$ equipped with its original stable structure.\\[1mm]
($\leftarrow$).  It is enough to show this implication working in a monster model $(G^*,\cdot, \C^*, \cdot) \succ (G,\cdot, \C, \cdot)$.
We argue by induction on $k$.  

If $Y$ does not fork over $\emptyset$, then $Y \ne \emptyset$, so $\lL^{\textrm{thick}}_0(Y)$. For the induction step, consider any $Y$ which does not fork over $\emptyset$. By inductive hypothesis, it is enough to show that 
$$S: = \{g \in G^* : gY \cap Y \;\, \textrm{does not fork over}\;\, \emptyset\}$$
is thick. 
Take $p^* \in S(\C^*)$ which does not fork over $\emptyset$ and contains $Y$. By stability, we know that the orbit $G^* \cdot p^*$ is bounded (of cardinality at most $2^{|T|}$), so $\Stab_{G^*}(p^*)$ is a bounded index subgroup of $G^*$. Write explicitly $Y(x) = \varphi(x,\bar a)$. Then $\Stab_{G^*}(p^*)$ is contained in
$$S':= \{ g \in G^* : gY \in p^*\} = \{ g \in G^* : \varphi(x,g\bar a)  \in p^*\} =\{ g \in G^* : \C^*\models d_{p^*}\varphi(g \bar a)\}.$$
By stability, $S'$ is a definable subset of $G^*$ (in the sense of the structure $(G^*,\cdot, \C^*, \cdot)$). 
All of this implies that $S'$ is thick, as otherwise, by the sufficient saturation of $(G^*,\cdot, \C^*, \cdot)$, we would get a sequence $(g_i)_{i < (2^{|T|})^+}$ of elements of $G^*$ such that $g_j^{-1}g_i \notin S'$ for all $i<j< (2^{|T|})^+$, which contradicts the fact that $[G^*: \Stab_{G^*}(p^*)] < (2^{|T|})^+$. On the other hand, $S'$ is clearly contained in $S$.\\[1mm]
($\rightarrow$). Suppose  $Y$ forks over $\emptyset$. Then, by stability, $Y$  $k$-divides over $\emptyset$ for some $k$. Then one can easily check that $\lL^{\textrm{thick}}_{k-1}(Y)$ does not hold. We will check it for $k=2$ and $k=3$, leaving the general case for the reader. 

Suppose $Y$ 2-divides over $\emptyset$. Then, by the strong $\aleph_0$-homogeneity of $\C$, there are $g_0,g_1,\dots \in G$ such that for all $i<j$, $g_iY \cap g_jY = \emptyset$. If $\lL^{\textrm{thick}}_1(Y)$ holds, then $\{ g : gY \cap Y \ne \emptyset\}$ is thick, so there are $i<j$ such that $g_j^{-1}g_i Y \cap Y \ne \emptyset$, a contradiction.  (Note that this argument does not work for ``generic'' in place of ``thick''.)

Suppose $Y=Y(x,\bar a)$ 3-divides over $\emptyset$, i.e. there is an indiscernible sequence $\bar a_0=\bar a,\bar a_1,\bar a_2, \dots$ such that $(Y(x,\bar a_i))_{i<\omega}$ is 3-inconsistent. Take $g_0:=\id \in G$ and $g_1 \in G$ with $g_1(\bar a_i)=\bar a_{i+1}$ for all $i<\omega$. Put $g_i :=g_1^i$, $i<\omega$. Then for all $i<j<k$, $g_iY \cap g_jY \cap g_kY= \emptyset$, and for all $i$ and $j$, $g_i g_j =g_{i+j}$. Suppose for a contradiction that $\lL^{\textrm{thick}}_2(Y)$ holds. Then there are $i<j$ such that $\lL^{\textrm{thick}}_1(g_i^{-1}g_jY \cap Y)$ holds. Hence, we can find $k<l$ such that $(g_i^{-1}g_jY \cap Y) \cap (g_k^{-1}g_lg_i^{-1}g_j Y \cap g_k^{-1}g_lY) \ne \emptyset$. In particular, $g_jY \cap g_iY \cap g_{l-k+j} Y \ne \emptyset$, a contradiction as $i<j<l-k+j$.
\end{proof}




\subsection{Means and stabilizers}\label{Subsection 2.3}

Let $X$ be a $G$-set. By a {\em $G$-lattice} we mean a family of subsets of $X$ including $\emptyset$ and $X$, which is closed under $G$-translations,   and intersections and unions of pairs.

 \begin{dfn}  Let $G$ be a group acting on $X$, $\dD$ a $G$-lattice of subsets of $X$.
  A {\em   mean} is a monotone, (non-negative),  translation-invariant function $m\colon \dD \to \R $
satisfying $m(\emptyset)=0$, and  for $Y,Z \in \dD$
\[ m (Y \cup Z) = m (Y)+m (Z) - m(Y \cap Z).  \]
 The mean $m$ is {\em normalized}, if $m(X)=1$.

 Given a mean $m$ and $\epsilon \in \mathbb{R}$, the {\em $\epsilon$-stabilizer} of a set $Y \subseteq X$ is defined to be  
 \[ \St_\epsilon(Y) : = \{g \in G: m(gY \cap Y)> (1-\epsilon) m(Y) \}. \]
 \end{dfn}

\begin{lem}\label{lemma: key0} 
Let $X$ be a $G$-set and $\dD$ a $G$-lattice. Let $m$ be a mean on $\dD$ (so $m(X)<\infty$), and  let $W \in \dD$ satisfy $m(W)>0$.  Then: 
\begin{enumerate} 
\item $\St_1(W) = \{g \in G: m(gW \cap W) >0 \}$ is thick (so generic).   
\item  We have $\lL_k(W)$ for all $k$ (working in $(G,\cdot, X, \cdot, W,\dots)$).  
\end{enumerate} 
\end{lem}

\begin{proof}  
(1) For some $n\in \mathbb{N}$ we have $n \cdot m(W) > m(X)$.  Suppose $\St_1(W)$ is not $n$-thick. Then one can find $g_i \in G$, $i=1,\ldots,{n}$, satisfying
$g_j^{-1} g_i \notin \St_1(W)$ for all $i<j$. Therefore, $m(g_i W \cap g_j W) = m(g_j^{-1} g_i W \cap W ) = 0$ for all $i<j$.  Hence, $n \cdot m(W)\leq m(X)$, a contradiction.\\[1mm]
(2) Let us work with $\lL = \lL^{\textrm{thick}}$ which clearly implies the version with $\lL = \lL^{\textrm{gen}}$.
Without loss, we can work in a monster model  $(G^*,\cdot, X^*, \cdot, W^*,\dots) \succ (G,\cdot, X, \cdot, W,\dots)$. To see this, apply the standard construction by incorporating $m$ into the language (as the collection of functions $m_{\varphi(x,\bar y)}$, where $m_{\varphi(x,\bar y)}(\bar b) := m(\varphi(x,\bar b))$ when it is defined, and say symbol $\infty$ otherwise), extending to the monster model, and taking the standard part; this yields a mean (which we still denote by $m$) on a certain $G^*$-lattice of subsets of $X^*$, including $W^*$, and such that $m(X^*)=m(X)<\infty$ and $m(W^*) = m(W)>0$. So without loss $(G,\cdot, X, \cdot, W,\dots)$ is a monster model.

We argue by induction on $k$. For $k=0$, $m(W)>0$ ensures $\lL_0(W)$.  For higher $k$, we know by induction that $\lL_{k-1}(gW \cap W)$ holds whenever $m(gW \cap W) >0$. Thus, $\{g \in G: \lL_{k-1}(gW \cap W)\}$ is thick by (1), so $\lL_k(W)$ holds by the sufficient saturation of the model and the definition of $\lL_k$. (Note that $\lL_{k-1}(gW \cap W)$ is a $\bigvee$-positively definable set $\bigvee_i D_i(g)$, so saturation is needed to deduce that $\{ g \in G: D_i(g)\}$ is thick for some $i$.) 
\end{proof}

 \begin{rem} In fact, the ideal $\iI_m = \{Y: m(Y)=0\}$ is an S1-ideal, 
i.e. $\iI_m$ is a $G$-invariant ideal on the lattice $\dD$ such that whenever $W \in \dD$ and there are arbitrary long finite sequences $(g_i)$ of elements of $G$ such that $g_iW \cap g_jW \in I$, then $W \in I$.    
 The stabilizer $\St_1$ can be defined for any S1-ideal
 $\iI$ as $\{g: gW \cap W  \notin \iI\}$, and Lemma \ref{lemma: key0} continues to hold for $W \notin \iI$.   
 The assumption on $m'$ in Proposition \ref{proposition: variant of Massicot-Wagner} below can be replaced by:  $\dD'$ carries an S1-ideal.  
  \end{rem}

\begin{lem}\label{lemma: product of stabilizers}
Let $X$ be a $G$-set and $\dD$ a $G$-lattice. Let $m$ be a mean on $\dD$. Then,
for any $Z \in \dD$ and $\epsilon_1,\epsilon_2 \in \mathbb{R}$, $\St_{\epsilon_1}(Z) \St_{\epsilon_2}(Z) \subseteq \St_{\epsilon_1 + \epsilon_2}(Z)$.
\end{lem}

\begin{proof}
The natural argument uses symmetric differences of sets, but here our lattice is not closed under set-theoretic difference, so we will mimic  means of symmetric differences.  (In fact, using Proposition \ref{proposition: extension of mean to a content}, we could work with the Boolean algebra generated by $\dD$ and use symmetric differences, but we do not do it here to keep this argument self-contained and completely elementary.)

Note that, by the invariance of $m$,  for any $\epsilon$ we have 
$$(\dag)\;\;\; g \in \St_{\epsilon}(Z) \iff m(gZ)+m(Z) -2m(gZ \cap Z) <2\epsilon m(Z).$$
Consider any $g_i \in  \St_{\epsilon_i}(Z)$ for $i=1,2$. Then, $m(g_iZ)+m(Z) -2m(g_iZ \cap Z) <2\epsilon_i m(Z)$ for $i=1,2$. Hence, by invariance, we easily get
$$m(g_1g_2Z) + 2m(g_1Z) + m(Z) - 2m(g_1g_2Z \cap g_1Z) -2m(g_1Z \cap Z)<2(\epsilon_1 + \epsilon_2)m(Z).$$ 
 By $(\dag)$, it is enough to show that the left hand side of the above inequality is greater than or equal to 
$m(g_1g_2Z) + m(Z) - 2m(g_1g_2Z \cap Z)$. By the modularity of $m$, this is easily seen to be equivalent to $m(g_1Z \cup(g_1g_2Z \cap Z)) \geq m(g_1Z \cap(g_1g_2Z \cup Z))$ which is true by the monotonicity of $m$.  
\end{proof}

 The following proposition is our strong version of the Massicot-Wagner elaboration of the stabilizer theorem of the first author.  It will be the engine for
 most of our main results.  We will actually need it only in case $X=G$, but the more general statement clarifies some aspects of the proof.   Note that when $X=G$, the conclusion ${{Y}}^{N} \subseteq \St_1(B A)$ implies that $Y^N \subseteq BAA^{-1} B^{-1}$. A suitable version also holds for approximate groups (yielding information on amenable approximate groups as in Massicot-Wagner), but we will stick with the global assumptions. 

\begin{prop} \label{proposition: variant of Massicot-Wagner} Let  $A \subseteq X$, $B \subseteq G$,  $N \in \Nn$. Let $\dD'$ be the set of finite intersections of translates $gB$. Let $\dD$ be a $G$-lattice including   $A$ and $B'A$ for $B' \in \dD'$.  Let    $m$  be an invariant mean on $\dD$, $m(A)>0$, and $m'$ an invariant mean on the lattice generated by $\dD'$, with $m'(B)>0$. Then   there exists a generic, symmetric set $Y \subseteq G$ that  is positively definable in  $(G,\cdot, B)$ over parameters from $G$, and such that $Y^{N} \subseteq \St_1(B A)$. 
\end{prop} 

\begin{proof}

In this proof, both largeness and $\bigvee$-definability are considered with respect to the theory of the structure $(G,\cdot, B)$.
We use the mean $m'$ only for the largeness of $B$.  Namely, by Lemma \ref{lemma: key0}, we have $\lL_k(B)$ for all $k\in \omega$.   
We will show:\\[-3mm]

    (**)   for some $k$ and $B' \in \dD'$ with $B' \subseteq B$ and $\lL_{k+1}(B')$, the set ${Y} : = \St_{\lL_{k}}(B')$ is generic as a $\bigvee$-definable set, and ${{Y}}^{N} \subseteq \St_1(B'A)$.\\[-3mm]

This means that if we present  (using Remark \ref{remark: St is a union f pos. def. symmetric sets}) $Y$ as  $\bigvee_n Y_n$ with the $Y_n$'s increasing, symmetric and positively definable over $G$, then some $Y_n$ is generic, and, of course, $Y_n^N \subseteq Y^N$.  So (**) will suffice.

Let $\epsilon = 1/N$.  
   Let $f(k) $ be the infimum of $m(B'A)$ over all $B' \in \lL_k \cap \dD'$ with  $B' \subseteq B$. So $0<m(A) \leq f(k) \leq {m}(X)$.    
  Thus, we cannot have $f(l) \geq \sqrt{1+\epsilon}  f(l-1)$ for all $l>0$.  Fix $l>0$ with   
  $f(l) < \sqrt{1+\epsilon} f(l-1)$.   Let $\lambda = f(l) \sqrt{1+\epsilon}$.   Let $B' \in \dD'$ with $B' \subseteq B$ satisfy\\[-3mm]
  
  (***) $\lL_{l}(B')$ and  $  m(B'A) < \lambda$.\\[-3mm] 
  
  We will show that any such $B'$ satisfies (**) (with $k=l-1$.)   Let ${Y}=\St_{\lL_{l-1}}(B')$.
 Since $B' \in \lL_l$, ${Y}$ is generic as a $\bigvee$-definable set.  
For $g \in Y$ we have $\lL_{l-1}(gB' \cap B')$, so 
\[ m(gB' A\cap B'A)\geq m((gB' \cap B') A) \geq f(l-1) > f(l) /\sqrt{1+\epsilon} >   m(B'A) / (1+\epsilon). \]
   Hence,  $g \in \St_\epsilon (B'A)$.  So $Y \subseteq \St_\epsilon(B'A)$.  
%
By Lemma \ref{lemma: product of stabilizers}, for any $Z$, $\St_{\epsilon}(Z)^N \subseteq \St_{N \epsilon} (Z)$. Thus, we conclude that $Y^N \subseteq \St_{1}(B'A)$.
This proves (**). 
\end{proof}

We will also need the following corollary of the proof of Proposition \ref{proposition: variant of Massicot-Wagner}.

\begin{cor}\label{corollary: good for applications}
Let  $A \subseteq X=G$, $\mathcal{B} \subseteq \mathcal{P}(G)$,  $N \in \Nn$. Put $\dD'=\{g_1B \cap \dots \cap g_nB: B \in \mathcal{B}, n \in \mathbb{N}_{>0}, g_1,\dots,g_n \in G\}$. Let $\dD$ be a $G$-lattice containing $\dD'$ and including $A$ and $B'A$ for $B' \in \dD'$.  Let $m$  be an invariant mean on $\dD$ with $m(A)>0$ and $m(B)>0$ for $B \in \mathcal{B}$. 
Then there exist $l \in \mathbb{N}_{>0}$, $\lambda \in \mathbb{R}$, $B \in \mathcal{B}$, $n \in \mathbb{N}_{>0}$, and $g_1,\dots,g_n \in G$ such that for $B':= B \cap g_1B \cap \dots \cap g_nB$ we have
%
$$\lL_{l}(B')\; \textrm{(working in $(G,\cdot, B)$)} \;\; \textrm{and}\;\;  m(B'A) < \lambda,$$
and whenever $E \in \mathcal{B}$ and $h_1,\dots,h_{n'} \in G$ (for some $n' \in \mathbb{N}_{>0}$) are chosen so that for $E':= E \cap h_1E \cap \dots \cap h_{n'}E$ one has  $\lL_{l}(E')$ (working in $(G,\cdot, E)$) and  $m(E'A) < \lambda$, then $S:=\St_{\lL_{l-1}}(E')$ is generic (as a set $\bigvee$-definable in $(G,\cdot,E)$), symmetric, and 
$S^N\subseteq E'A(E'A)^{-1}\subseteq EA(EA)^{-1}$.
\end{cor}
The above corollary will be used later for $N=8$; in \cite{HKP2}, we will use it for $N=16$.


\subsection{From pre-mean to mean}\label{Subsection 2.4}

We show how to extend a pre-mean to a mean canonically; if the pre-mean is $G$-invariant, the resulting mean
will therefore be $G$-invariant, too. This will be essential in the proofs of the main results of Section \ref{Section 2}.

By a {\em lattice} of subsets of a set $X$ we mean a family of subsets of $X$ including $\emptyset$ and $X$, which is closed under intersections and unions of pairs.

\begin{dfn}
  A {\em normalized mean} on a lattice $(L,\cup,\cap)$ of subsets of a set $X$ is a monotone function $\rho: L \to   [0,1]$, satisfying:
\[ \rho (Y \cup Y') = \rho (Y)+\rho (Y')- \rho(Y \cap Y'), \]
and $\rho(\emptyset)=0$, $\rho(X)=1$.
\end{dfn}


Whenever we present a type-definable set $Z$ as an intersection $\bigcap_{i}  Z_i$, we mean that the $Z_i$'s are definable, $i$ ranges over a directed set $(I,<)$, and $Z_j \subseteq Z_i$ for $i<j$.

Let $E= \bigcap_{i \in I} R_i$ be a type-definable equivalence relation on a definable set $X$, where without loss each $R_i$ is reflexive and symmetric.

Working in the monster model, we write $Y/E$ for the image of $Y \subseteq X$ in $X/E$, and $YE$ for the pullback of $Y/E$ in $X$. 
For a binary relation $R$ on $X$, and $Y \subseteq X$, by $R \circ Y$ we mean $\{ x \in X: (\exists y \in Y) R(y,x) \}$. In particular, $YE=E \circ Y$.

The following definition and lemma can be read over any base set of parameters.

\begin{dfn} \label{definition: premean} A {\em pre-mean} for $X/E$ is a  monotone  function $m$ from definable subsets of $X$
into $[0,1]$,  with $m(\emptyset)=0$, $m(X)=1$,  
and  $m (Y \cup Y') \leq m (Y)+m (Y') $, such that equality holds whenever $(R_i \circ Y) \cap Y' = \emptyset$ for some $i$.
\end{dfn}

By compactness, the condition ``$(R_i \circ Y) \cap Y' = \emptyset$ for some $i$'' is equivalent to ``$(E \circ Y) \cap Y' =\emptyset$''.

\begin{lem} \label{lemma: premeantomean}
Let $m$ be a pre-mean for $X/E$.  Then $m$ induces a normalized mean $\nu$ on the lattice of  sets $Y/E$, with $Y$ a type-definable subset of $X$, or equivalently on the lattice  of type-definable subsets $Y$ of $X$ with $YE=Y$, in the following way 
$${\nu}(Y) := \inf \{ m(D): D \hbox{ definable, } Y \subseteq D  \}.$$
\end{lem}

\begin{proof}  
 Let $L$ be the lattice of all   $\bigwedge$-definable  sets $Y$ with $YE=Y$.
 For $Y \in L$, define   
 $${\nu}(Y) = \inf \{ m(D): D \hbox{ definable, } Y \subseteq D  \}.$$
Clearly  ${\nu}(\emptyset)=0$, ${\nu}(X)=1$, ${\nu}$ is monotone, and ${\nu}(Y \cup Y') \leq {\nu}(Y) + {\nu}(Y')$.
If $Y,Y' \in L$ are disjoint, then $\bigwedge_{i} R_i(y,y') \wedge y \in Y \wedge y' \in Y'$ is inconsistent.
By compactness, for some $i$ and some definable $D \supseteq Y$ and $D' \supseteq Y'$, we have
$(R_i \circ D) \cap D' = \emptyset$.  As $m$ is a pre-mean, we have $m(D \cup D') = m(D) +m(D')$,
and likewise for any definable subsets of $D,D'$.  Hence, in this case, ${\nu}(Y \cup Y') = {\nu}(Y)+{\nu}(Y')$.  

Now, $L$ is not complemented, but we do have:\\[-2mm]

\begin{clm}  
Let $Y \subseteq Z$ be both in $L$.  For any $\epsilon>0$ there exists $Y' \subseteq Z$, $Y' \in L$,
$Y'$ disjoint from $Y$, and  with ${\nu}(Y)+{\nu}(Y') \geq {\nu}(Z)-\epsilon$.  
\end{clm}

\begin{clmproof}
Write $Y=\bigcap Y_k$ and $Z = \bigcap_l Z_l$ with definable $Y_k$ and $Z_l$.
Find $k$ such that $\nu(Y) \geq m(Y_k)-\epsilon$. 
Let \[  Y' = E\circ (Z \setminus Y_k) =E \circ \left(\bigcap_{l} Z_l \setminus Y_k \right)  =  \bigcap_{i,l} R_i \circ (Z_l \setminus Y_k). \]
Then  $Y' \in L$, and $Y' \subseteq E \circ Z=Z$. 
Also, $Y' \subseteq E \circ (X \setminus Y_k) \subseteq E \circ (X \setminus Y) =X \setminus Y$ (as $Y$ and $X \setminus Y$ are unions of $E$-classes), so $Y \cap Y' = \emptyset$.
 Finally, $\nu(Y') = \inf_{i,l} m(R_i \circ (Z_l \setminus Y_k)) \geq \inf_{l} m( Z_l \setminus Y_k) $,  so  
 \[\nu(Y') +  m(Y_k)   \geq \inf_l m(Z_l \setminus Y_k)  + m(Y_k)  \geq \inf_l m(Z_l) = \nu(Z) \]
As $m(Y_k) \leq \nu(Y)+\epsilon$, we obtain $\nu(Y')+\nu(Y)+\epsilon \geq \nu(Z)$ as required.  
\end{clmproof}

From this, the  equality $\nu(Y \cup Z) = \nu(Y)+\nu(Z) - \nu(Y \cap Z) $ can be shown as follows. Take any $\epsilon >0$. Find $Y' \in L$ such that $Y' \subseteq Y$, $Y'$ disjoint from $Y \cap Z$, and $\delta_1 := \nu(Y) - \nu(Y') - \nu(Y \cap Z) \leq \frac{1}{2}\epsilon$. Similarly, find $Z' \in L$ such that $Z' \subseteq Z$, $Z'$ disjoint from $Y \cap Z$, and $\delta_2 := \nu(Z) - \nu(Z') - \nu(Y \cap Z) \leq \frac{1}{2}\epsilon$.  Then $Y \cap Z, Y',Z'$ are pairwise disjoint subsets of $Y \cup Z$. Finally, find $T \in L$ such that $T \subseteq Y  \cup Z$, $T$ disjoint from $Y \cap Z$, and $\delta:= \nu(Y \cup Z) - \nu(T) - \nu(Y \cap Z) \leq \epsilon$. Put $Y'' = Y' \cup(T \cap Y) \in L$ and $Z'' =Z' \cup(T \cap Z)\in L$. Then $Y''$, $Z''$ and $Y \cap Z$ are pairwise disjoint subsets of $Y \cup Z$. Put $T'' = Y'' \cup Z''\in L$. We see that $\delta'':=\nu(Y \cup Z) - \nu(T'') - \nu(Y \cap Z) \leq \delta \leq \epsilon$ and $T''$ is disjoint from $Y \cap Z$. Moreover, $\delta_1'' := \nu(Y) - \nu(Y'') - \nu(Y \cap Z) \leq \delta_1\leq \frac{1}{2}\epsilon$ and $\delta_2'' := \nu(Z) - \nu(Z'') - \nu(Y \cap Z) \leq \delta_2\leq \frac{1}{2}\epsilon$. Note also that $\delta_1'',\delta_2'',\delta'' \geq 0$.

We get $|\nu(Y \cup Z) -\nu(Y)-\nu(Z) +\nu(Y \cap Z)| = |\delta'' + \nu(T'') - \nu(Y) - \nu(Z) +2\nu(Y \cap Z)|= |\delta'' -(\nu(Y)-\nu(Y'')-\nu(Y\cap Z)) - (\nu(Z) - \nu(Z'') - \nu(Y \cap Z))|=|\delta'' - \delta_1'' - \delta_2''|$. Since $\delta'' \in [0,\epsilon]$ and $\delta_1'',\delta_2'' \in [0,\frac{1}{2}\epsilon]$, we see that $|\delta'' - \delta_1'' - \delta_2''|\leq \epsilon$. Letting $\epsilon \to 0$, we obtain the desired equality. 
\end{proof}

Lemma \ref{lemma: premeantomean} will be sufficient to deal with Case 1 in Subsection \ref{section: main results on components}, i.e. to prove Theorem \ref{theorem: c4}. In order to deal with Case 2 and prove Theorem  \ref{theorem: indeed c4}, we will need some variant of this lemma. Namely,  suppose that the type-definable equivalence relation $E$ is on a definable group $G$.

\begin{dfn} \label{definition: G-premean} A {\em $G$-pre-mean} for $G/E$ is a  pre-mean for $G/E$ such that  
$m (Y \cup Y') = m (Y)+m (Y') $ whenever $((g_1R_i \cap \dots \cap g_n R_i) \circ Y) \cap Y' = \emptyset$ for some $g_1,\dots,g_n \in G$ and some $i \in I$ (where $gR_i:=\{(ga,gb): (a,b) \in R_i\}$).
\end{dfn}

The following variant of Lemma \ref{lemma: premeantomean} follows from Lemma \ref{lemma: premeantomean}.

\begin{cor} \label{corollary: variant of premeantomean}
Let $m$ be a $G$-pre-mean for $G/E$.  Then $m$ induces a normalized mean $\nu$ on the lattice of type-definable subsets $Y$ of $G$ with $(g_1E\cap \dots \cap g_nE) \circ Y=Y$ for some $g_1,\dots,g_n \in G$, in the following way 
$${\nu}(Y) := \inf \{ m(D): D \hbox{ definable, } Y \subseteq D  \}.$$
\end{cor}

\subsection{Means and measures}\label{subsection: means and measures}

In this subsection, we will prove that, in a certain general context, the existence of an invariant mean is equivalent to the existence of an invariant measure on an appropriate space. This is interesting in its own right, but also yields model-theoretic absoluteness of various notions of ``amenability'', i.e. the existence of invariant measures on appropriate spaces computed for a given model $M$ does not depend on the choice of $M$.

Let us recall some definitions from measure theory.

\begin{dfn}
Let $R$ be a ring of subsets of a given set $X$, namely closed under finite unions and differences; an example is a Boolean algebra of subsets of $X$.\\ 
1) A {\em content} on $R$ is a function $m \colon R \to [0,+\infty]$ which is finitely additive and satisfies $m(\emptyset)=0$.\\ 
2) A {\em pre-measure} on $R$ is a content which is $\sigma$-additive, namely  if $(A_n)_{n <\omega}$ is a sequence of pairwise disjoint members of $R$ whose union $A$ is also in $R$, then $m(A) = \sum_n m(A_n)$.\\
3) A {\em measure} is a pre-measure on a $\sigma$-algebra of subsets of a given set. 
\end{dfn}

A content $m$ on a ring $R$ of subsets of $X$ is called {\em $\sigma$-finite} if $X$
is the union of an increasing sequence $(X_n)_{n < \omega}$ of elements of $R$ with $m(X_n)<\infty$.

\begin{fct}[Carath\'{e}odory extension theorem]
Let $\nu$ be a $\sigma$-finite pre-measure on a ring $R$ of subsets of $X$. Then
there is a unique extension of $\nu$ to a measure on the $\sigma$-algebra $\sigma(R)$ generated by $R$. 
\end{fct}

From the proof, or from a more precise statement which says that the extended measure (restricted to $\sigma(R)$) is just the outer measure induced by $\nu$, it follows that if $R$ is a $G$-ring (for an action of a group $G$ on $X$) and $\nu$ is $G$-invariant, then so is the extended measure.
It is clear that the converse of the above theorem is also true, i.e. if a content $\nu$ on $R$  extends to a measure on $\sigma(R)$, then $\nu$ is a pre-measure.

When $(X_n)_{n <\omega}$ is a descending sequence of sets whose intersection is empty, we will write $X_n \downarrow \emptyset$; when  $(X_n)_{n <\omega}$ is an ascending sequence of sets whose union is $X$, we will write $X_n \uparrow X$. 

\begin{rem}\label{remark: characterizations of sigma additivity}
Let $\nu$ be a content on a ring $R$ of subsets of $X$ taking only finite values. Then $\nu$ is a pre-measure if and only if for every sequence $(X_n)_{n < \omega}$ of sets from $R$ such that $X_n \downarrow \emptyset$ one has $\lim_n \nu(X_n) =0$ (in this case we say that $\nu$ is continuous at 0). If $R$ is a Boolean algebra, these conditions are also equivalent to the condition that for every sequence $(X_n)_{n < \omega}$ of sets from $R$ such that $X_n \uparrow X$ one has $\lim_n \nu(X_n) = \nu(X)$.
\end{rem}

\begin{prop}\label{proposition: extension of mean to a content}
If $\rho$ is a normalized mean on a lattice $(L,\cap,\cup)$ of subsets of a set $X$, then it extends uniquely to a content $\nu$ on the Boolean algebra $B(L)$ generated by $L$. If $L$ is a $G$-lattice and $\rho$ is $G$-invariant, then so is $\nu$.
\end{prop}

\begin{proof}
{\bf Case 1:} $L$ is finite, say equal to $\{A_0,\dots,A_n\}$.

It is clear that there is a unique possible candidate for $\nu$, namely $\nu$ is determined by the formulas
$$\nu\left(A_0^{\epsilon(0)} \cap \dots \cap A_n^{\epsilon(n)}\right) = \rho\left(\bigcap_{i \in \Delta_{\epsilon}^+} A_i\right) - \rho\left((\bigcap_{i \in \Delta_{\epsilon}^+} A_i) \cap (\bigcup_{i \in \Delta_{\epsilon}^-} A_i)\right),$$
for any $\epsilon \in \{0,1\}^{n+1}$, where  $\Delta_{\epsilon}^+:= \{ i\leq n: \epsilon(i)= 1\}$, $\Delta_{\epsilon}^-:= \{ i\leq n: \epsilon(i)= 0\}$, and $A_i^{0}:=X \setminus A_i$, $A_i^{1}:= A_i$. This follows by finite additivity of $\rho$ and the fact that each element of $B(L)$ can be (uniquely) written as a (disjoint) union of sets of the form $A_0^{\epsilon(0)} \cap \dots \cap A_n^{\epsilon(n)}$. 

Conversely, it is clear that when we define $\nu$ be the above formulas on the atoms of $B(L)$ and then extend additively, then we get a content. It is also clear that if $\rho$ is $G$-invariant, so is $\nu$. The remaining thing to check is that $\nu$ extends $\rho$, i.e. $\nu(A_k) =\rho(A_k)$ for all $k \leq n$. 

We argue by induction on $n$, where the base induction step for $n=0$ is clear. Assume the conclusion holds for numbers less then a given $n>0$. It is enough to show that $\nu(A_n)=\rho(A_n)$. 

$$
\nu(A_n)= \sum_{\epsilon \in 2^{n}}\rho\left(A_n \cap \bigcap_{i \in  \Delta_{\epsilon}^+} A_i\right) - \rho\left(A_n \cap (\bigcap_{i \in \Delta_{\epsilon}^+} A_i) \cap (\bigcup_{i \in \Delta_{\epsilon}^-} A_i)\right)=S_1+S_2,$$

where 
$$S_1= \sum_{\epsilon \in 2^{n-1}} \rho\left(A_n \cap A_{n-1} \cap  \bigcap_{i \in  \Delta_{\epsilon}^+} A_i\right) -  \rho\left(A_n \cap A_{n-1} \cap (\bigcap_{i \in \Delta_{\epsilon}^+} A_i) \cap (\bigcup_{i \in \Delta_{\epsilon}^-} A_i)\right),$$
$$S_2 =  \sum_{\epsilon \in 2^{n-1}}\rho\left(A_n \cap \bigcap_{i \in  \Delta_{\epsilon}^+} A_i\right) - \rho\left(A_n \cap (\bigcap_{i \in \Delta_{\epsilon}^+} A_i) \cap (A_{n-1} \cup \bigcup_{i \in \Delta_{\epsilon}^-} A_i)\right).
$$
%
%
By the modularity of $\rho$, 
$$\begin{array}{lll}
S_2 & = & \sum_{\epsilon \in 2^{n-1}}\rho\left(A_n \cap \bigcap_{i \in  \Delta_{\epsilon}^+} A_i\right) - \rho\left(A_n \cap (\bigcap_{i \in \Delta_{\epsilon}^+} A_i) \cap (\bigcup_{i \in \Delta_{\epsilon}^-} A_i)\right)\\ 
&& -  \rho\left(A_n \cap A_{n-1} \cap \bigcap_{i \in \Delta_{\epsilon}^+} A_i\right) + \rho\left(A_n \cap A_{n-1} \cap (\bigcap_{i \in \Delta_{\epsilon}^+} A_i) \cap (\bigcup_{i \in \Delta_{\epsilon}^-} A_i)\right).
\end{array}
$$
Thus, $S_1+S_2 = \sum_{\epsilon \in 2^{n-1}}\rho\left(A_n \cap \bigcap_{i \in  \Delta_{\epsilon}^+} A_i\right) - \rho\left(A_n \cap (\bigcap_{i \in \Delta_{\epsilon}^+} A_i) \cap (\bigcup_{i \in \Delta_{\epsilon}^-} A_i)\right)$, which is equal to $\rho(A_n)$ by induction hypothesis. Thus, the induction step has been completed.\\[2mm]
{\bf Case 2:} $L$ is arbitrary.

For uniqueness notice that any content on $B(L)$ extending $\rho$ is determined by its restrictions to all Boolean algebras generated by finite sublattices of $L$ and that these restrictions are unique by Case 1. To show existence, for any finite sublattice $L_0 \subseteq L$ let $\nu_{L_0}$ be the unique content on $B(L_0)$ extending $\rho|_{L_0}$, which exists by Case 1. Then note that by uniqueness in Case 1, $\bigcup_{L_0} \nu_{L_0}$ is the desired content.
 %
%
\end{proof}

The next easy example shows that it may happen that a mean $\rho$ is continuous at $0$, but the unique extension $\nu$ to a content on the generated Boolean algebra is not continuous at $0$, i.e. $\nu$ is not a pre-measure and so it cannot be extended to a measure.

\begin{ex} Take any infinite set $X$ and present it as an increasing union of sets $X_n$. Let the lattice $L$ consist of $\emptyset, X_0,X_1,\dots,X$. Define a mean on $L$ by: $\rho(\emptyset)=0$, $\rho(X)=1$, $\rho(X_n)=\frac{1}{2}$. Then, if $A_n \downarrow \emptyset$, where $A_n \in L$, then eventually $A_n=\emptyset$, so $\rho$ is continuous at $0$. Let $\nu$ be the unique extension of $\rho$ to a content on $B(L)$. Then $X \setminus X_n \downarrow \emptyset$, but $\lim \nu(X \setminus X_n) = \frac{1}{2} \ne 0$, so $\nu$ is not a pre-measure.
\end{ex}

The means we are interested in come from pre-means, and we will see that this rules out obstacles as in the above example.

From now on, we work in models of a given theory $T$.  As is well-known, a  definable family of definable sets  is given by a formula $\varphi(\bar x, \bar y)$, in the sense that the family is precisely the collection of sets defined by the formulas  $\varphi(\bar x, \bar b)$ as $\bar b$ varies (over a given model, or over the monster model, or over all models).  We generalize this to the 
notion of a {\em $\bigvee$-definable family} (of definable sets), given now by  a collection $\{\varphi_{i}(\bar x, \bar z_i):i\in I\}$ of formulas.  
Namely, the family is $\{\varphi_{i}(\bar x, \bar z_i):i\in I, \bar z_i \textrm{ belongs to any model}\}$. We call it a $\bigvee$-definable family, as it is a union of definable families (so it can also be seen as an inductive limit of definable families). 

\begin{dfn}
We will say that a $\bigvee$-definable family $\eE:=\{\varphi_i(x,y,\bar z_i): i \in I, \bar z_i \; \mbox{belongs to any model}\}$ {\em defines} an equivalence relation if for every model $M$, $E_M:= \bigcap \eE_M$ is an ($M$-type-definable) equivalence relation, where $\eE_M := \{\varphi_i(x,y,\bar b_i) : i \in I, \bar b_i \subseteq M\}$. 
\end{dfn}

By a standard trick, we can and do assume that $I$ is a directed set and for every $i<j$, $(\forall \bar z_i)(\exists \bar z_j) (\varphi_j(x,y,\bar z_j) \rightarrow \varphi_i(x,y,\bar z_i))$.

The above definition is introduced in order to capture for example the following situations. A $\emptyset$-type-definable equivalence relation $E = \bigcap_{i \in I} R_i(x,y)$ is defined by the $\bigvee$-definable family $\{R_i(x,y): i \in I\}$ (so here there are no parameter variables $\bar z_i$). In particular, the relation of lying in the same left [resp. right] coset of a $\emptyset$-type-definable subgroup $H$ of a $\emptyset$-definable group $G$ is defined by a $\bigvee$-definable family (without parameter variables). To get another important example, consider any $\emptyset$-type-definable subgroup $H=\bigcap_{i \in I} X_i$ (where $I$ is directed and $X_j \subseteq X_i$ whenever $i <j$). 
Put $\eE = \{ G(x) \wedge G(y) \wedge z^{-1}(yx^{-1})z \in X_i : i \in I, z\}$ (where for $z \notin G$ or $a \notin G$ we put $z^{-1}az:=a$). Then, for any model $M$, $E_M$ is the $M$-type-definable equivalence relation of lying in the same right coset of $\bigcap_{g \in G(M)} H^g$. More generally, when $G$ is equipped with a $\bigvee$-definable group topology as in Subsection \ref{Subsection 2.5}, then the relation of lying in the same left [resp. right] coset of the infinitesimals (i.e. $\mu_M$ from Definition \ref{definition: infinitesimals over M}) is also naturally defined by a $\bigvee$-definable family.

From now on, let $G$ be a $\emptyset$-definable group and let $E$ be an equivalence relation on $G$ defined by a $\bigvee$-definable family $\eE:=\{\varphi_i(x,y,\bar z_i): i \in I, \bar z_i\}$; we assume that each $\varphi_i(x,y,\bar z_i)$ implies that $x,y \in G$. Work in a monster model $M^*$; so $G=G(M^*)$.

\begin{dfn}\label{definition: G-premean for vee-definable equivalence relations}
 By a {\em $G$-pre-mean} for $\eE_M$ we mean a $G$-pre-mean for $G/E_M$ (see Definition \ref{definition: G-premean}), i.e.  a  monotone  function $m$ on definable (with parameters) subsets of $G$  
into $[0,1]$,  with $m(\emptyset)=0, m(G)=1$,  
and  $m (Y \cup Y') \leq m (Y)+m (Y') $, such that equality holds
whenever $((g_1\varphi_i(x,y,\bar b_1) \cap \dots \cap g_n \varphi_i(x,y,\bar b_n)) \circ Y) \cap Y' = \emptyset$ for some $g_1,\dots,g_n \in G$, $i \in I$, and $\bar b_1,\dots,\bar b_n$ from $M$.
\end{dfn}

By the standard construction (by incorporating the mean into the language, as was recalled in the proof of Lemma \ref{lemma: key0}(2)), we have the following remark.
\begin{rem}\label{remark: extensions of premeans to premeans}
A $G$-pre-mean for $\eE_M$, but defined only on $M$-definable sets and satisfying the ``equality criterion'' only for $g_1,\dots,g_n \in G(M)$, extends to a $G$-pre-mean for $\eE_M$ (defined on all definable sets). In fact, it extends to a $G$-pre-mean for $\eE_{M^*}$, which is clearly also a $G$-pre-mean for $\eE_N$ for any $N\prec M^*$. If the initial $G$-pre-mean is $G(M)$-invariant, then the $G$-pre-mean for $\eE_{M^*}$ is $G(M^*)$-invariant, so it is also a $G$-invariant $G$-pre-mean for $\eE_N$ for any $N \prec M^*$.
\end{rem}

For a model $M$, let $\dD_{\eE_M}$ be the $G$-lattice of type-definable (with parameters) subsets $D$ of $G(M^*)$ such that $(g_1E_M \cap \dots \cap g_nE_M) \circ D =D$ for some $g_1,\dots,g_n \in G$.

From Corollary \ref{corollary: variant of premeantomean} and Remark \ref{remark: extensions of premeans to premeans}, we get:

\begin{cor}\label{corollary: extension of premean to mean}
Let $m$ be a $G$-pre-mean for $\eE_M$.  Then $m$ induces a normalized mean $\rho$ on the lattice $\dD_{\eE_M}$ via ${\rho}(Y) := \inf \{ m(D): D \hbox{ definable, } Y \subseteq D  \}$. If $m$ is $G(M)$-invariant, then we can replace it by a $G$-invariant pre-mean, and then the induced $\rho$ is $G$-invariant as well.  
\end{cor}

The converse is easy is check.

\begin{rem}
A normalized mean $\rho$ on the lattice $\dD_{\eE_M}$ induces a $G$-pre-mean $m$ for $\eE_M$ via $m(Y):=\inf \{\rho((g_1E_M \cap \dots \cap g_nE_M) \circ Y): g_1,\dots,g_n \in G\}$. If $\rho$ is $G(M)$-invariant [resp. $G$-invariant], so is $m$. 
\end{rem}

\begin{cor}\label{corollary: mean always induced by premean}
If $\dD_{\eE_M}$ carries a [$G(M)$-invariant], normalized mean, then it carries such a mean $\rho$ which is induced from a [$G$-invariant] $G$-pre-mean $m$ for $\eE_M$ via  ${\rho}(Y) := \inf \{ m(D): D \hbox{ definable, } Y \subseteq D  \}$.
\end{cor}

\begin{cor}\label{corollary: absoluteness of the existence of mean}
The existence of a $G$-invariant normalized mean on  $\dD_{\eE_M}$ does not depend either on the choice of $M$ or the monster model $M^*$ in which the lattice is computed.
\end{cor}

The next proposition is the main observation of this subsection.

\begin{prop}\label{proposition: main proposition}
Assume $E_M$ is $G(M)$-invariant. The following conditions are equivalent.
\begin{enumerate}
\item $S_{G/E_M}(M)$ carries a $G(M)$-invariant, Borel probability measure.
\item There is a $G(M)$-invariant [$G$-invariant] $G$-pre-mean for $\eE_M$.
\item The lattice $\dD_{\eE_M}$ carries a $G(M)$-invariant [$G$-invariant], normalized mean.
\end{enumerate}
\end{prop}

\begin{proof} 
First note that $E_M$ being $G(M)$-invariant guarantees that $G(M)$ acts naturally on $G/E_M$, which induces an action of $G(M)$ on $S_{G/E_M}(M)$.\\ 
(1) $\rightarrow$ (2). Let $\mu$ witnesses (1). For an $M$-definable subset $D$ of $G$ define $m(D):=\mu(\bar D)$, where $\bar D$ is the set of complete types (over $M$) of elements of $D/E_M$. Since $E_M$ is $G(M)$-invariant, we easily see that $m$ is a $G(M)$-invariant $G$-pre-mean for $\eE_M$, but defined only on $M$-definable sets and satisfying the ``equality criterion'' only for $g_1,\dots,g_n \in G(M)$. By Remark \ref{remark: extensions of premeans to premeans}, it extends to an actual $G$-invariant $G$-pre-mean (defined on all definable sets) for $\eE_M$.\\
(2) $\rightarrow$ (3). This follows from Corollary \ref{corollary: extension of premean to mean}.\\
(3) $\rightarrow$ (1). 
Take a $G(M)$-invariant, normalized mean on $\dD_{\eE_M}$. By Corollary \ref{corollary: mean always induced by premean}, there exists a $G$-invariant, normalized mean $\rho$ induced from a $G$-invariant $G$-pre-mean $m$ for $\eE_M$ via  ${\rho}(Y) := \inf \{ m(D): D \hbox{ definable, } Y \subseteq D  \}$. By Proposition \ref{proposition: extension of mean to a content}, let $\nu$ be the unique extension of $\rho$ to a $G$-invariant content on the Boolean algebra $B(\dD_{\eE_M})$. We will show that $\nu$ is a pre-measure, which by Carath\'{e}dory theorem can be further extended to a $G$-invariant measure $\bar\nu$ on the generated $\sigma$-algebra $\sigma(B(\dD_{\eE_M}))$. 
Then $\bar \nu$ induces a $G(M)$-invariant, Borel probability measure on $S_{G/E_M}(M)$ via $\mu(P):= \bar \nu (\{a \in M^*: \tp((a/E_M)/M) \in P\})$ for any Borel subset $P$ of $S_{G/E_M}(M)$, and the proof will be complete. So it remains to show
\\[2mm]
\begin{clm}
$\nu$ is a pre-measure.
\end{clm}

\begin{clmproof}
Put $R: =B(\dD_{\eE_M})$.
By Remark \ref{remark: characterizations of sigma additivity}, it is enough to show that for every sequence $(X_n)_{n < \omega}$ of sets from $R$ such that $X_n \uparrow G$ one has $\lim_n \nu(X_n) = 1$. Take any $\epsilon >0$. We need to show that $\nu(X_n) > 1 - \epsilon$ for some $n$. 

One can find sets $Z_k \subseteq Y_k$ (for $k \in \omega$) from $\dD_{\eE_M}$ and natural numbers $n_0<n_1<\dots$ such that 
$$X_i = (Y_0 \setminus Z_0) \sqcup \dots \sqcup (Y_{n_i} \setminus Z_{n_i})$$
for every $i<\omega$ (where $\sqcup$ stands for disjoint union).  Then 
$$\nu (X_i)=\sum_{k=0}^{n_i} \rho(Y_k) - \rho(Z_k).$$

For each $k$ we can choose a definable set $D_k \supseteq Y_k$ such that 
$$\sum_{k=0}^\infty m(D_k) - \rho(Y_k) < \epsilon.$$

For each $k$ let $\fF_k$ be the family of all sets $F$ definable over the set of parameters over which $Z_k$ is defined and such that $Z_k \subseteq F \subseteq D_k$. Then $\bigcap \fF_k = Z_k$ for every $k$. Therefore,
$$G = \bigcup_i X_i = \bigcup_k Y_k \setminus Z_k \subseteq \bigcup_k D_k \setminus Z_k = \bigcup_k \left(D_k \setminus \bigcap \fF_k\right) = \bigcup_k \bigcup_{F \in \fF_k} D_k \setminus F,$$
so by the saturation of $M^*$, there are $k_1<\dots<k_n <\omega$ and $F_{k_1} \in \fF_{k_1}, \dots, F_{k_n} \in \fF_{k_n}$ such that 
$$G=(D_{k_1} \setminus F_{k_1}) \cup \dots \cup (D_{k_n} \setminus F_{k_n}).$$
(Note that this is not necessarily a disjoint union.) 

We also have $Z_{k_j} \subseteq F_{k_j}$. 
Since $Z_{k_j}\in \dD_{\eE_M}$, by compactness, it is easy to see that there are definable sets $F_{k_j}' \in \fF_{k_j}$ contained in $F_{k_j}$ such that $((g_1\varphi_i(x,y,\bar b_1) \cap \dots \cap g_n \varphi_i(x,y,\bar b_n)) \circ F_{k_j}') \cap (D_{k_j} \setminus F_{k_j}) = \emptyset$ for some $g_1,\dots,g_n \in G$, $i \in I$, and $\bar b_1,\dots,\bar b_n$ from $M$  (all depending on $j$ of course).
Hence, $m((D_{k_j} \setminus F_{k_j}) \cup F_{k_j}') = m(D_{k_j} \setminus F_{k_j})+m(F_{k_j}')$, which implies that $m(D_{k_j} \setminus F_{k_j}) \leq m(D_{k_j}) -m(F_{k_j}')$.

From all these observations, we get
$$
\begin{array}{l}
1 = m(G)=m(\bigcup_{j=1}^n D_{k_j} \setminus F_{k_j}) \leq \sum_{j=1}^n m(D_{k_j} \setminus F_{k_j}) \leq\\
\sum_{j=1}^n m(D_{k_j}) -m(F_{k_j}')< (\sum_{j=1}^n \rho(Y_{k_j}) - \rho (Z_{k_j})) + \epsilon.
\end{array}
$$ 
Hence, $\nu(X_{k_n}) \geq \nu((Y_{k_1} \setminus Z_{k_1}) \sqcup \dots \sqcup (Y_{k_n} \setminus Z_{k_n})) = \sum_{j=1}^n \rho(Y_{k_j}) - \rho (Z_{k_j}) > 1- \epsilon$.
\end{clmproof}
The proof of the proposition is complete
.\end{proof}

\begin{rem}
In Proposition \ref{proposition: main proposition}, one can add one more equivalent condition:
\begin{enumerate}
\item[(4)] The $G(M)$-lattice $\dD_{E_M}$ of type-definable subsets $D$ of $G(M^*)$ such that $E_M \circ D = D$ carries a $G(M)$-invariant, normalized mean. 
\end{enumerate}
\end{rem}

\begin{proof}
The implication (3) $\rightarrow$ (4) is trivial, while (4) $\rightarrow$ (1) follows similarly to (3) $\rightarrow$ (1) (note that, in the proof of $(3) \to (1)$, it is enough to work with $G(M)$-invariant pre-means, means and contents in order to get that $\mu$ is $G(M)$-invariant).
\end{proof}

The reason why we work with the more complicated lattice $\dD_{\eE_M}$ instead of $\dD_{E_M}$ is that the former is a $G$-lattice which is needed in Case 2 in Subsection \ref{section: main results on components}.

From Corollary \ref{corollary: absoluteness of the existence of mean} and Proposition \ref{proposition: main proposition}, we get

\begin{cor}\label{corollary: absoluteness of the existence of measure}
Assume $E_M$ is $G(M)$-invariant for every model $M$. Then, the existence of a $G(M)$-invariant, Borel probability measure on $S_{G/E_M}(M)$ does not depend on the choice of $M$.
\end{cor}

For the type-definable equivalence relation $E(x,y)$ given by $x^{-1}y \in H$, where $H$ is a $\emptyset$-type-definable subgroup of $G$, Corollary \ref{corollary: absoluteness of the existence of measure} specializes to

\begin{cor}
The existence of a $G(M)$-invariant, Borel probability measure on $S_{G/H}(M)$ does not depend on the choice of $M$.
\end{cor} 

Corollary \ref{corollary: absoluteness of the existence of measure} specializes to more absolutness results in the context of $\bigvee$-definable group topologies, which will be discussed in Subsection \ref{Subsection 2.5} (see Corollary \ref{corollary: absoluteness of definable topological amenability}).

\begin{rem}
If $E_M$ is not $G(M)$-invariant, then there is no natural (left) action of $G(M)$ on $S_{G/E_M}(M)$. But we can always replace the family $\eE$, by $\eE':=\{\varphi_i(t_ix,t_iy,\bar z_i): i \in I, t_i,\bar z_i\}$ (where for $a \notin G$ or $b \notin G$ we put $ab:=b$). Then, for any model $M$, the induced equivalence relation $E'_M$ will be the intersection of all $gE_M$ for g ranging over $G(M)$. And, by Corollary  \ref{corollary: absoluteness of the existence of measure}, the existence of a $G(M)$-invariant, Borel probability measure on $S_{G/E_M'}(M)$ does not depend on the choice of $M$.
\end{rem}

\subsection{Measures, means, and connected components}\label{section: main results on components}

Now, consider a structure $M$, a $\emptyset$-definable group $G$, and an $M$-type-definable subgroup $H$ of $G$ (naming parameters, we can assume that $H$ is $\emptyset$-type-definable). Usually $G$ will stand for the interpretation of $G$ in a monster model $M^*$ (i.e. $G=G^*=G(M^*)$); by $G(M)$ we denote the interpretation of $G$ in $M$. 

We will be interested in the following two cases.\\

{\bf Case 1:}  The type space $S_{G/H}(M)$ (i.e. the space of complete types over $M$ of left cosets modulo $H$) carries a $G(M)$-invariant, Borel probability measure. 

The discussion below repeats some arguments from the previous subsection in a special case, but since this will be the context of the main results of Section \ref{Section 2}, we prefer to write it explicitly.

Let $\bar{m}$ be a $G(M)$-invariant, Borel probability measure on $S_{G/H}(M)$. 
We define a $G(M)$-invariant pre-mean  $m'$ (see Definition \ref{definition: premean}, where the equivalence relation is $xH=yH$)  on $M$-definable subsets of $G$, by $m'(Y) := \bar{m}(\bar{Y})$, where $\bar{Y}$ is the set of complete types over $M$ of elements of $Y/H$. 

As in the proof of Lemma \ref{lemma: key0}, the standard construction allows us to extend $m'$ to a $G$-invariant pre-mean on $M^*$-definable subsets of $G=G(M^*)$. Note that this extended pre-mean is definable over $\emptyset$ in some expansion of the language (meaning that for any closed interval $I$ and for any formula $\varphi(x,\bar y)$ of the original language the set $\{\bar b : m'(\varphi(x,\bar b)) \in I\}$ is $\emptyset$-type-definable in this expansion of the language), and $M^*$ can be chosen so that $M \prec M^*$ in the expanded language. Next, using Lemma \ref{lemma: premeantomean}, we obtain a normalized, $G$-invariant mean $m$ on $\dD_H$ -- the $G$-lattice of $M^*$-type-definable subsets $Y$ of $G$ satisfying $YH=Y$ -- which satisfies 
$m(Y) = \inf \{ m'(D): D \hbox{ definable, } Y \subseteq D  \}$.\\

{\bf Case 2:} $H$ is normalized by $G(M)$, and $S_{H \backslash G}(M)$ carries a $G(M)$-invariant, Borel probability measure, where the action of $G(M)$ on $S_{H \backslash G}(M)$ is induced by the action on $H \backslash G$ given by $g*(Ha):=gHa=H(ga)$.

Let $\bar{m}$ be a $G(M)$-invariant, Borel probability measure on $S_{H \backslash G}(M)$. As in Case 1, we obtain a $G(M)$-invariant pre-mean  (for the equivalence relation $Hx=Hy$) $m'$ on $M$-definable subsets of $G$. The standard construction allows us to extend it to a definable over $\emptyset$ (in some expansion of the language), $G$-invariant pre-mean $m'$ on $M^*$-definable subsets of $G=G(M^*)$, for some monster model $M^*$ such that $M^* \succ M$ also in the expanded language. Moreover, since $H$ is normalized by $G(M)$, the standard construction gives us the following additional property of $m'$: For any $Y$ and $Z$ definable subsets of $G$, $M$-definable superset $D$ of $H$, and $g_1,\dots,g_n \in G$, if $(D^{g_1}\cap \dots \cap D^{g_n})Y \cap Z =\emptyset$, then $m'(Y \cup Z)=m'(Y)+m'(Z)$, i.e. $m'$ is a $G$-pre-mean for $H \backslash G$, using the terminology from Definition  \ref{definition: G-premean} (which follows from the fact that the left translate by $g$ of the relation $Dx=Dy$ is precisely the relation $D^gx=D^gy$). By Corollary \ref{corollary: variant of premeantomean}, we obtain a normalized, $G$-invariant mean $m$ on $\dD_H'$ --- the $G$-lattice of $M^*$-type-definable subsets $Y$ of $G$ satisfying $(H^{g_1}\cap \dots\cap H^{g_n})Y=Y$ for some $g_1,\dots,g_n \in G$ --- which satisfies $m(Y) = \inf \{ m'(D): D \hbox{ definable, } Y \subseteq D  \}$.\\
 
 

 


We are ready to prove the main results of this section. They concern situations from the above Cases 1 and 2, respectively. We will give a detailed proof of the first theorem and only explain how to modify it to get the second one.

In the rest of this section, we will write $Z^4$ to mean $ZZ Z^{-1} Z^{-1} $; $Z^8$ denotes $Z^4 Z^4$.

 \begin{thm}\label{theorem: c4}  Let $H$ be a $\emptyset$-type-definable subgroup of $G$, normalized by $G(M)$.
 Let $N$ be  the normal subgroup generated by $H$.  
Then $(1)\leftrightarrow (2)\leftrightarrow(3)\leftrightarrow(4) \to (5)$:
\begin{enumerate}
\item $S_{G/H}(M)$ carries a $G(M)$-invariant, Borel probability measure.  
\item There is a $G(M)$-invariant pre-mean for $G/H$ on $M$-definable subsets of $G$.
\item There is a $G$-invariant pre-mean for $G/H$ which is definable over $\emptyset$ in some expansion of the language in which $M \prec M^*$ (enlarging $M^*$ if necessary).
\item  The lattice $\dD_H$  carries a normalized, $G$-invariant mean.
\item  $G^{00}_M    \leq  N G^{000}_M$.
\end{enumerate}
   \end{thm}
 
 \begin{proof} 
The equivalence of conditions (1)-(4) essentially follows from Proposition \ref{proposition: main proposition} applied to $\eE: = \{G(x) \wedge G(y) \wedge x^{-1}y \in X_i: i \in I\}$, where $H = \bigcap_{i \in I} X_i$  (with $I$ directed and $X_j \subseteq X_i$ whenever $i <j$). For that notice that the relation $E_M$ in this special case is just lying in the same left coset of $H$, so it is $G$-invariant, and the lattice $\dD_{\eE_M}$ coincides with $\dD_H$. 

However, for the reader's convenience, we explain some of these equivalences more explicitly.
By the above discussion of Case 1, any $G(M)$-invariant, Borel probability measure on $S_{G/H}(M)$ induces a $G(M)$-invariant pre-mean on $M$-definable subsets of $G$, which then can be extended to a $G$-invariant pre-mean $m'$ on definable subsets of $G$ which is definable over $\emptyset$ in some expansion of the language in which $M \prec M^*$ (enlarging $M^*$ if necessary).  This in turn induces a normalized, $G$-invariant mean $m$ on the lattice $\dD_H$, which satisfies $m(Y) = \inf \{ m'(D): D \hbox{ definable}, Y \subseteq D  \}$. So $(1) \to (2) \leftrightarrow (3) \to (4)$. 
The implication $(4) \to (1)$ follows from the implication $(3) \to (1)$ in Proposition \ref{proposition: main proposition}.

It remains to prove $(4) \to (5)$. So assume $(4)$. By $(4) \to (2)$ and the above discussion, we have a  $G$-invariant mean $m$ on $\dD_H$ given by $m(Y) = \inf \{ m'(D): D \hbox{ definable, } Y \subseteq D  \}$ for some pre-mean $m'$ satisfying (3).

 Let $p \in S_G(M)$ be a wide type of $G$, in the sense that $m(DH) >0$ for any $D \in p$. In order to finish the proof, it is enough to show that  $(HpH)^4$ contains $G^{00}_M$. Indeed, then, since $p p^{-1} \subseteq G^{000}_M$ implies $ppp^{-1} p^{-1}\subseteq G^{000}_M$, and so $(HpH)^{4} \subseteq NG^{000}_M$, we get  $G^{00}_M \leq NG^{000}_M$ which is the desired conclusion.

 As $HpH$ is an intersection of partial types $P$ over $M$  satisfying $HPH=P$ and $m(P)>0$
 (namely the appropriate $HDH$ with $D$ $M$-definable),  it suffices to show that for each such $P$, $P^4$  contains $G^{00}_M$.
For this, it suffices to find for any $M$-definable set $P'$ containing $P$ a generic, $M$-type-definable set $Q=HQH$ with $Q^8 \subseteq P'^4$, for then $m(Q)>0$ and we can find an $M$-definable set $Q'$ containing $Q$ such that $Q'^8 \subseteq P'^4$, and we can iterate: find a generic, $M$-type-definable $R=HRH$ with $R^8 \subseteq Q'^4$ and an $M$-definable $R'$ containing $R$ and satisfying $R'^8 \subseteq Q'^4$, etc., and at the limit take the intersection $P'^4 \cap Q'^4 \cap R'^4 \cap \dots$ -- an $M$-type-definable, bounded index, subgroup contained in $P'^4$, which clearly contains $G^{00}_M$. Since this is true for any $M$-definable $P'$ containing $P$, we get $G^{00}_M \subseteq P^4$.

So consider a partial type $P$ over $M$ satisfying $HPH=P$ and $m(P)>0$. Consider any $M$-definable $P'$ containing $P$. We will apply 
Corollary \ref{corollary: good for applications} to: $X:=G$, $A:=P$, and the family 
$$\mathcal{B}: = \{HQH : P \subseteq HQH \subseteq P'\; \textrm{and $Q$ is $M$-definable}\}$$ 
of subsets of $G$. 
Recall that $\dD'$ is the collection of all intersections $g_1B \cap \dots \cap g_nB$, where $B \in \mathcal{B}$ and $g_1,\dots,g_n \in G$, and as $\dD$ take the $G$-lattice generated by: $\dD'$, the set $A$, and all sets $B'A$ for $B' \in \dD'$.
Note that $\dD \subseteq \dD_H$, so our mean $m$ is defined on $\dD$.
%
%
By Corollary \ref{corollary: good for applications}, we find $l \in \mathbb{N}$, $\lambda \in \mathbb{R}$, $B \in \mathcal{B}$, $n \in \mathbb{N}_{>0}$, and $g_1,\dots,g_n \in G$ such that for $B':= B \cap g_1B \cap \dots \cap g_nB$ we have\\[-2mm] 

  (***) $\lL_{l}(B')$ (working in $(G,\cdot,B)$) and  $m(B'A) < \lambda$,\\[2mm]
and whenever $E \in \mathcal{B}$ and $h_1,\dots,h_{n'} \in G$ (for some $n' \in \mathbb{N}_{>0}$) are chosen so that for $E':= E \cap h_1E \cap \dots \cap h_{n'}E$ one has  $\lL_{l}(E')$ (working in $(G,\cdot,E)$) and  $  m(E'A) < \lambda$, then $\St_{\lL_{l-1}}(E')$ is generic (as a set $\bigvee$-definable in $(G,\cdot,E)$), symmetric and has 8th power contained in $E'P(E'P)^{-1} \subseteq E^4$. 

We can find $M$-definable sets $C$ and $D$ such that $B \subseteq C \subseteq P'$, $P \subseteq D$ and  for $C':= C \cap g_1C \cap \dots \cap g_nC$, we have $m'(C'D) < \lambda$ (where $m'$ is the pre-mean on definable subsets of $G$ chosen at the beginning of the proof of $(4) \to (5)$). Now, choose any $M$-definable set $Q$ such that $B \subseteq Q \subseteq HQH \subseteq C$. Let $Q'=Q \cap g_1Q \cap \dots \cap g_nQ$. 
Then $B' \subseteq Q'$, so, by (***) and Remark \ref{independence of largness of the language}, we get $\lL_{l}(Q')$ (working in $(G,\cdot,Q)$). Since $\lL_{l}(Q')$ is a $\bigvee$-definable (over $\emptyset$) condition on $g_1,\dots,g_n$ in the structure $(G,\cdot,Q)$ and $Q$ is $M$-definable in the original theory, we see that $\lL_{l}(Q')$ is an $M$-$\bigvee$-definable condition on $g_1,\dots,g_n$ in the original theory. On the other hand, $m'(C'D) < \lambda$ is a $\bigvee$-definable (over $\emptyset$) condition on $g_1,\dots,g_n$ in the expanded language (in which $m'$ is definable over $\emptyset$). Since $M \prec M^*$ also in this expanded language, we can find $g_1,\dots,g_n \in G(M)$ such that  $\lL_{l}(Q')$ and $m'(C'D) < \lambda$ still holds  for the corresponding $Q'$ and $C'$. Finally, take $E:=HQH$ and $E':= E \cap g_1E \cap \dots \cap g_nE$. 
We see that $E \in \mathcal{B}$,  $\lL_{l}(E')$ (working in $(G,\cdot,E)$) and $m(E'A)<\lambda$.

Define $Y:=\St_{\lL_{l-1}}(E')$. By the the choice of $l$ and $\lambda$, we have that $Y = \bigvee_\nu Y_\nu$ is generic and $Y^8 \subseteq E^4 \subseteq P'^4$.

As $H$ is normalized by $g_1,\dots,g_n$, we have $HE'H=E'$. Since $HE' =E'$, we have $HYH= Y$, and moreover $Y$ is a disjunction of sets $Y_\nu$ positively $M$-definable in $(G,\cdot,E)$ and satisfying $HY_\nu H=Y_\nu$. Indeed,  let $R(x,\bar y)$ be a new predicate. By the approximations to $\lL_l$ mentioned in and after Remark \ref{remark: lL is vee definable}, we have that for any $s \in \omega$ there are increasing sets $P_{\nu,s}(R)(\bar y)$, $\nu \in \omega$, positively $\emptyset$-definable in $(G,\cdot,X,\cdot, R)$ such that $P_{\nu,s}(R(gx,\bar y))(\bar y) \iff  P_{\nu,s}(R(x,\bar y))(\bar y)$ for all $g \in G$,  and $\lL_s (R(x,\bar y))$ can be presented as the $\bigvee$-positively definable set $\bigvee_\nu  P_{\nu,s}(R(x,\bar y))(\bar y)$. In particular, if $R(x,\bar y)$ is positively definable in $(G,\cdot,X, \cdot, E)$, then the $P_{\nu,s}(R)(\bar y)$ are positively definable in $(G,\cdot,X, \cdot, E)$ over the same parameters over which $R(x,\bar y)$ is defined. Applying this to our situation for $s:=l-1$ and $R(x,y):= (x \in (yE' \cap E'))$, we get that $Y = \{y : \lL_s(yE' \cap E')\}$ can be presented as $\bigvee_\nu P_{\nu,s}(R(x,y))(y)$. Putting $Y_\nu(y)=P_{\nu,s}(R)(y)$, we have that $Y_\nu$ is positively $M$-definable in $(G,\cdot, E)$, and, since $HE'=E'$, we get that for any $h_1,h_2 \in H$:  $Y_{\nu}(h_1yh_2) \iff P_{\nu,s}(R(x,h_1yh_2))(y) \iff  P_{\nu,s}(x \in h_1yh_2E' \cap E')(y) \iff  P_{\nu,s}(x \in h_1yE' \cap E')(y) \iff P_{\nu,s}(x \in yE' \cap h_1^{-1}E')(y) \iff P_{\nu,s}(x \in yE' \cap E')(y) \iff Y_\nu(y)$. So $HY_\nu H=Y_\nu$ as was claimed at the beginning of this paragraph.

Since the $Y_\nu \subseteq G$ are positively $M$-definable in $(G,\cdot, E)$ and $E$ is $M$-type-definable in the original theory, we easily get that the $Y_\nu$ are $M$-type-definable in the original theory. Moreover, some $Y_\nu$ will be generic, and $HY_\nu H=Y_\nu$, and $Y_\nu^8 \subseteq P'^4$.   
\end{proof}

In the situation of Case 2, we have

 \begin{thm}\label{theorem: indeed c4}  Let $H$ be a $\emptyset$-type-definable subgroup of $G$, normalized by $G(M)$.
 Let $N$ be  the normal subgroup generated by $H$.  
Then $(1)\leftrightarrow (2)\leftrightarrow(3)\leftrightarrow(4) \leftrightarrow (5) \to (6)$:
\begin{enumerate}
\item $S_{H\backslash G}(M)$ carries a $G(M)$-invariant, Borel probability measure.  
\item There is a $G(M)$-invariant pre-mean for $H\backslash G$ on $M$-definable subsets of $G$.
\item There is a $G$-invariant $G$-pre-mean $m'$ for $H\backslash G$ (i.e. a  $G$-invariant pre-mean $m'$ for $H\backslash G$ such that $m'(Z \cup Z')=m'(Z) + m'(Z')$ whenever $D'Z\cap Z' = \emptyset$ for some  $M$-definable superset $D$ of $H$ and $D'=D^{g_1}\cap \dots \cap D^{g_n}$ for some $g_1,\dots, g_n \in G$).
\item There is a $G$-invariant $G$-pre-mean for $H\backslash G$ which is definable over $\emptyset$ in some expansion of the language in which $M \prec M^*$ (enlarging $M^*$ if necessary).
\item  The lattice $\dD_H'$  carries a normalized, $G$-invariant mean.
\item  $G^{00}_M    \leq  N G^{000}_M$.
\end{enumerate}
   \end{thm}

\begin{proof}
The equivalence of conditions (1)-(5) essentially follows from Proposition \ref{proposition: main proposition} applied to $\eE: = \{G(x) \wedge G(y) \wedge yx^{-1} \in X_i: i \in I\}$, where $H = \bigcap_{i \in I} X_i$  (with $I$ directed and $X_j \subseteq X_i$ whenever $i <j$). For that notice that the relation $E_M$ in this special case is just lying in the same right coset of $H$, so it is $G(M)$-invariant by the assumption that $H$ is normalized by $G(M)$, and the lattice $\dD_{\eE_M}$ coincides with $\dD_H'$. One should also use the above discussion of Case 2.
%

It remains to justify $(5) \to (6)$. 

So assume $(5)$. By $(5) \to (2)$ and the discussion of Case 2, we have a $G$-invariant mean $m$ on $\dD_H'$ given by $m(Y) = \inf \{ m'(D): D \hbox{ definable, } Y \subseteq D  \}$ for some pre-mean $m'$ satisfying (4). We follow the lines of the proof of $(4) \to (5)$ in Theorem \ref{theorem: c4}, but now it is enough to work with right cosets modulo $H^{g_1} \cap \dots \cap H^{g_n}$ for some $g_1,\dots,g_n \in G(M)$ (in place of two-sided cosets of $H$), e.g. $P$ is a partial type over $M$ satisfying $(H^{g_1} \cap \dots \cap H^{g_n})P=P$ (for some $g_1,\dots,g_n \in G(M)$) and $m(P)>0$. 
The way how $\dD_H'$ was defined is essential to ensure that $\dD \subseteq \dD_H'$ (and so $m$ is defined on $\dD$).
\end{proof}

Conjecture \ref{conjecture: Conjecture 0.4 of KrPi} follows immediately from Theorem \ref{theorem: indeed c4}, taking $H := \mu$:

\begin{cor}\label{corollary: main conjecture from KrPi}
1) Let $G(M)$ be a topological group and assume that the members of a basis of neighborhoods of the identity are definable in $M$. If $G$ is definably topologically amenable, then $G^{00}_{\defi,\topo}=G^{000}_{\defi,\topo}$. \\
2) Let $G(M)$ be a topological group. If $G(M)$ is amenable, then $G^{00}_{\topo}=G^{000}_{\topo}$.
\end{cor}

\subsection{$\bigvee$-definable group topologies}\label{Subsection 2.5}


In Section \ref{Section 1}, we recalled two contexts to deal with topological groups model-theoretically: one with all open subsets being definable, and a more general one with a basis of open neighborhoods at the identity consisting of definable sets.  Notice, however, that in each of these contexts we do not get a natural group topology when passing to elementary extensions. In order to get a group topology in an arbitrary elementary extension, one usually considers a more special context with a uniformly definable basis of open neighborhoods at the identity (in other words, when a basis of open sets at the identity is a definable family).

As usual, let $G$ be a $\emptyset$-definable group, and $M$ or $N$ denotes a model.
Here, we extend the last context, for example to cover topologies induced on $G(M)$ by type-definable subgroups of $G$ normalized by $G(M)$.
Note that any $\emptyset$-type-definable subgroup $H= \bigcap_{i \in I} X_i$ 
(with the definable sets $X_i$, where without loss $I$ is a directed set such that $X_j \subseteq X_i$ for $i<j$), normalized by $G(M)$, can be viewed as topologizing $G(M)$ in the sense that the family $\{X_i(M) : i \in I\}$ is a basis of 
(not necessarily open!) neighborhoods at the identity; but on a bigger model it will not in general give a topology.  It is thus natural to consider a slightly stronger condition.

We first elaborate on some terminology introduced briefly in Subsection \ref{subsection: means and measures}. By a {\em $\bigvee$-definable family of definable subsets of $G$}, we mean a class $\tT=\{\varphi_i(x,\bar y_i): i \in I,\;\bar y_i \textrm{ belongs to any model}\}$, where $\varphi_{i}(x,\bar y_i)$ are some formulas implying $G(x)$. In any model $M$, 
$$\tT(M) := \{\varphi_i(M,\bar y_i): i \in I, \bar y_i \in M\}$$ 
is an actual collection of subsets of $G(M)$; also, put 
$$\tT_M := \{ \varphi_i(x,\bar y_i): i \in I, \bar y_i \in M\}.$$ 
By a standard trick, we can, and will from now on, assume that $I$ is a directed set, and for every $i<j$ we have $(\forall \bar y_i)(\exists \bar y_j)(\varphi_j(x,\bar y_j) \rightarrow \varphi_i(x,\bar y_i))$; the last condition is equivalent to the property that for every model $M$ and $i<j$, each member of the definable family  $\{\varphi_i(M,\bar y_i): \bar y_i \in M\}$ contains a member of the family $\{\varphi_j(M,\bar y_j): \bar y_j \in M\}$. (In fact, by the aforementioned standard trick, we could even replace the word ``contains'' by ``equals'', but we will not need it.)

\begin{dfn}  A {\em $\bigvee$-definable group topology} on $G$ is a $\bigvee$-definable family $\tT=\{\varphi_i(x,\bar y_i): i \in I, \bar y_i\}$ of definable subsets of $G$ containing $1$ such that in any model $M$, $\tT(M)$ forms a basis of (not necessarily open) neighborhoods of the identity for a topology on $G(M)$, making the group operations continuous.  
%
Equivalently, for any model $M$, $\tT(M)$ consists of subsets of $G(M)$ containing $1$, such that each of the following sets 
\begin{enumerate}
\item the intersection of any two members of $\tT(M)$, 
\item the inversion of any member of $\tT(M)$,
\item the conjugate of any member of $\tT(M)$ by an element of $G(M)$
\end{enumerate}
contains a member of $\tT(M)$, and, additionally, if $A \in \tT(M)$, then there exists $B \in \tT(M)$ with $B^2 \subseteq A$.   
\end{dfn}

It easy to check that in the above definition, it is enough to take a sufficiently saturated model $M$.

By compactness,  a  $\bigvee$-definable group topology on $G$ is a $\bigvee$-definable family $\tT=\{\varphi_i(x,\bar y_i): i \in I, \bar y_i\}$ of definable subsets of $G$ containing $1$ such that:

\begin{enumerate}
\item For every $i,j \in I$ there is $k \in I$ such that $(\forall \bar y_i  )(\forall \bar y_j)(\exists \bar y_k)(\varphi_k(x,\bar y_k) \rightarrow (\varphi_i(x,\bar y_i) \wedge \varphi_j(x,\bar y_j)))$. 
\item For every $i \in I$ there is $j \in I$ such that $(\forall \bar y_i )(\exists \bar y_j)(\varphi_j(x^{-1},\bar y_j) \rightarrow \varphi_i(x,\bar y_i))$.
\item For every $i \in I$ there is $j \in I$ such that $(\forall \bar y_i )(\forall z)(\exists \bar y_j) (\varphi_j(z^{-1}xz,\bar y_j) \rightarrow \varphi_i(x,\bar y_i))$.
\item For every $i \in I$ there is $j \in I$ such that $(\forall \bar y_i )(\exists \bar y_j)((\exists x_1,x_2)(\varphi_j(x_1,\bar y_j) \wedge \varphi_j(x_2,\bar y_j) \wedge x=x_1\cdot x_2) \rightarrow \varphi_i(x,\bar y_i))$.
\end{enumerate}


Let $\tT =\{ \varphi(x,\bar z_i): i \in I, \bar z_i\}$ be a $\bigvee$-definable group topology on $G$. Work in a fixed monster model $M^*$ (so $M \prec M^*$ by convention).

\begin{dfn}\label{definition: infinitesimals over M}
We let $\mu^\tT_M$ be the $M$-type-definable subgroup $\bigcap_{D \in \tT_M} D$.  When the identity of $\tT$ is
clear, we write $\mu_M$. 
\end{dfn}

It is clear that $\mu^\tT_M$ is normalized by $G(M)$.

 
\begin{rem}\label{remark: definition of cl}  For any $A$-definable set $D$, there exists an $A$-type-definable set $\cl(D)$ such that for any model $M$, 
the closure of $D(M)$ is $\cl(D)(M)$.   
Namely, $a \in \cl(D)$ iff  $\bigwedge_i (\forall \bar y_i)(a\cdot \varphi_i(x,\bar y_i) \cap D(x) \ne \emptyset$), more formally: $\bigwedge_i (\forall \bar y_i)(\exists x)(\varphi_i(a^{-1}x,\bar y_i) \wedge D(x)$).

For a type-definable set $D = \bigcap_i D_i$ (where $D_j \subseteq D_i$ for $i<j$),  let $\cl(D) = \bigcap_i \cl(D_i)$.  For any sufficiently saturated  model $M$, the closure of $D(M)$ is $\cl(D)(M)$. 
\end{rem}

\begin{rem} \label{remark: closed}  
For $P$ $M$-type-definable, 
$\cl(P)   \subseteq P  \mu^\tT_M $.  
Indeed, $\cl(P) = \bigcap \{PH:  H \in \tT\}$, which formally means that $\cl(P)(z)$ is the type $\bigwedge_i (\forall \bar y_i)(\exists x_1,x_2)(P(x_1) \wedge \varphi_i(x_2,\bar y_i) \wedge z=x_1 \cdot x_2)$. In particular,  $\mu^\tT_M$ is closed. Similarly, $\cl(P)   \subseteq  \mu^\tT_M P$. In fact, $\cl(P)$ is contained in both $P  ((\mu^\tT_M)^{g_1} \cap \dots  \cap (\mu^\tT_M)^{g_n})$ and $((\mu^\tT_M)^{g_1} \cap \dots  \cap (\mu^\tT_M)^{g_n})P$ for any $g_1,\dots,g_n \in G$.
\end{rem}

By $\mathcal{C}^\tT$ (or just $\mathcal{C}$) we will denote $\cl(1)$. Then $\mathcal{C} = \bigcap \tT$, so it is $\emptyset$-type-definable, and it is a normal subgroup of $G$. It is clear that $\mathcal {C} \leq \mu_M$ for any $M$.  Note that formally $\mathcal C$ coincides with  $\mu_{M^{*}}$ which happens to be $\emptyset$-type-definable in the monster model $M^{*}$. 

We say that $\tT$ is {\em strongly Hausdorff} if $G(M)$ is Hausdorff in every model $M$; equivalently $\mathcal {C} = \{ 1\}$; equivalently,  $\bigcap \mathcal{F}=\{1\}$ for some definable family $\mathcal{F} \subseteq \tT$.
Note that, in contrast with definable families, $\tT(M)$ may be Hausdorff for one $M$, without $\tT$ being strongly Hausdorff.  This occurs when $\mu^\tT_M(M)=\{1\}$, see Example \ref{example: explicit example}.  Note also that given a Hausdorff topological group, we can expand it to a first order structure in which there is a definable  $\tT$ which is strongly Hausdorff and induces the given topology on the group we started from.

Define the following $G$-lattices of subsets of $G$ (a $\emptyset$-definable group equipped with a $\bigvee$-definable group topology $\tT$). 
\begin{enumerate}
\item $\dD^{\mu_M}_{\tT}$ -- the $G$-lattice of sets $D$ type-definable in $M^*$ over arbitrary parameters, such that $D\mu_{M}^\tT=D$.
\item  ${\dD^{\mu_M}_{\tT}}'$ -- the $G$-lattice of sets $D$ type-definable in $M^*$ over arbitrary parameters, such that $((\mu^\tT_M)^{g_1} \cap \dots  \cap (\mu^\tT_M)^{g_n})D=D$ for some $g_1,\dots,g_n \in G$.
\item $\dD_{\tT}$ -- the $G$-lattice of closed sets $D$ type-definable in $M^*$ over arbitrary parameters.
\item $\dD_\mathcal{C}$ -- the $G$-lattice of sets $D$ type-definable in $M^*$ over arbitrary parameters, such that $D\mathcal{C}=D$. 
\end{enumerate}
By Remark \ref{remark: closed}, it is clear that  $\dD^{\mu_M}_{\tT}$ and  ${\dD^{\mu_M}_{\tT}}'$ are both contained in $\dD_{\tT} \subseteq \dD_\mathcal{C}$.

Now, we give an example showing that type-definable subgroups lead, in a natural way, to $\bigvee$-definable group topologies.

\begin{ex}\label{Example: Vee-topology given by a type-definable subgroup} 
Let $H=\bigcap_{i \in I} X_i$ be any $\emptyset$-type-definable subgroup of $G$  
(and without loss $I$ is directed, and $X_j \subseteq X_i$ whenever $i<j$).  Let $\tT$ be the union of all families $\tT_{i,m}$, where $\tT_{i,m}$ is the class of $m$-fold intersections of conjugates of $X_i$, for instance $\tT_{i,1} = \{ g^{-1} X_i g: g \in G\}$. It is clear that with the order $(i,m) <(j,n) \iff i<j \wedge m<n$, $\tT$ can be treated as a $\bigvee$-definable family of definable subsets of $G$ containing $1$. 
Clearly, for any model $M$, $\mu_{M}^\tT=\bigcap \tT_M$  is a type-definable subgroup of $G$ normalized by $G(M)$; 
it follows that $\tT(M)$ defines a group topology on $G(M)$.

In case when $H$ is invariant under conjugation by elements of $G(M)$, 
we can recover $H$ as the intersection of all $M$-definable
neighborhoods of the identity. 

All of this works also for $H$ type-definable over $M$ (allowing formulas with parameters from $M$ in the definition of $\bigvee$-definable group topology).

In case $H$ is a normal subgroup of $G$, the family  $\tT$ yields the same topology as the family $\{X_i : i \in I\}$ (where $X_i(x)$ are definable sets which do not depend on any parameters $\bar y_i$), $\mu_M =\mathcal{C}= H$ does not depend on $M$, and $\cl(P)=PH$ for any type-definable set $P$. In particular, 
$\dD^{\mu_M}_{\tT} = {\dD^{\mu_M}_{\tT}}' = \dD_{\tT} = \dD_\mathcal{C}$ for every $M$.

For an arbitrary $M$-type-definable subgroup $H$, the above $\bigvee$-definable group topology $\tT$ may be strictly weaker than the one given by the normal core of $H$, i.e. the type-definable normal subgroup $\Core(H)$ defined by the type $\{(\forall z)(x \in X_i^z): i \in I\}$. It may even happen that $H$ is normalized by $G(M)$, the topology on $G(M)$ induced by $H$ is non-discrete, whereas $\Core(H)$ is trivial and so induces the discrete topology. To give an example, one can start from the free abelian group $\mathbb{Z}$, and then use suitable HNN-extensions on even steps and direct products with $\Z$ on odd steps so that the resulting group has a single non-trivial conjugacy class and the centralizer of any finite subset is non-trivial. Let $M =G(M)$ be such a group, and $G \succ M$ a monster model. Let $H$ be the intersection of the centralizers of all elements of $G(M)$. It is an $M$-type-definable subgroup of $G$ normalized by $G(M)$. Since the intersection of any finitely many centralizers is non-trivial, the induced topology on $G(M)$ is non-discrete. On the other hand, since in $G(M)$, and so also in $G$, there is a single non-trivial conjugacy class, $\Core(H)$ coincides with the center of $G$ which is trivial. 
\end{ex}




\begin{rem} Example \ref{Example: Vee-topology given by a type-definable subgroup} shows a connection between 
 $\bigvee$-definable group topologies and $G(M)$-normal, $M$-type-definable subgroups:
\begin{itemize}
\item each $\bigvee$-definable group topology $\tT$ yields the $G(M)$-normal, $M$-type-definable subgroup $\mu_M^{\tT}$;
\item each $G(M)$-normal, type-definable over $\emptyset$ [or over $M$] subgroup $H$ yields the $\bigvee$-definable group topology $\tT$ on $G$ [defined over $M$, if $H$ is defined over $M$]  described in Example  \ref{Example: Vee-topology given by a type-definable subgroup}, such that $\mu_M^{\tT} =H$.
\end{itemize}
However, the former notion, namely that of a $\bigvee$-definable group topology is more precise, as it is a priori given without reference to the particular small model $M$.  Also, the map from $\bigvee$-definable group topologies to
$G(M)$-normal, $M$-type-definable subgroups (or topologies on $G(M)$), given by $\tT \mapsto \mu_M^{\tT}$, is not injective;  Example  \ref{Example: Vee-topology given by a type-definable subgroup} provides the smallest
$\bigvee$-definable group topology specializing to a given topology on $G(M)$, but there can certainly be others, e.g. in the Abelian case, non-discrete, strongly Hausdorff topologies are never deduced from a single model in this way (see also Example \ref{example: explicit example}).
\end{rem}

\begin{rem}\label{remark: smallest extension of topology}
Let us change the notation only for the purpose of this remark. Let $G$ be an arbitrary topological group. Choose a basis $\{X_i : i \in I\}$ of open sets at the identity, with $X_j \subseteq X_i$ whenever $i<j$. Expand the pure group language with predicates for all $X_i$'s, and denote the resulting structure by $M$ and the resulting language by $\mathcal{L}$. Let $M^*$ be a monster model, $G^*=G(M^*)$ and $X_i^*= X_i(M^*)$. Then $H:=\bigcap X_i^*$ is a $\emptyset$-type-definable group which is normalized by $G=G(M)$.
So Example \ref{Example: Vee-topology given by a type-definable subgroup} yields a $\bigvee$-definable group topology $\tT$ which specializes to the original topology on $G$. This is the smallest (in a strong sense) $\bigvee$-definable group topology which specializes to the original topology on $G$, namely, for any other such topology $\tT'$ which is $\bigvee$-definable in an expansion of the pure group structure on $G$ whose language is denoted by $\mathcal{L'}$, and for any model $N\succ M$ in the sense of 
$\mathcal{L}\cup \mathcal{L}'$, the topology on $G(N)$ given by $\tT$ is weaker than the one given by $\tT'$. This shows that Example \ref{Example: Vee-topology given by a type-definable subgroup} allows us to extend the given group topology on $G$ to the canonical (i.e. smallest among topologies $\bigvee$-definable in arbitrary expansions) group topology on elementary extensions.
\end{rem}

Let us look at a concrete example illustrating some of the above discussions.

\begin{ex}\label{example: explicit example}
Take $M:= (\mathbb{Z},+,\cdot)$ and $G(M) := (\mathbb{Z},+)$. Take the $\emptyset$-type-definable subgroup $H:=\bigcap_{n \in \mathbb{N}} n!G$.
The family $\tT$ from Example \ref{Example: Vee-topology given by a type-definable subgroup} coincides with the $\bigvee$-definable family $\{ n!G : n \in \mathbb{Z}\}$. So $\tT(M)$ is Hausdorff, but $\tT$ is not strongly Hausdorff. 
Now, consider the definable family $\mathcal {F}=\{g\cdot G: g \in G \setminus \{0\}\}$ of definable subsets of $G$ containing $0$. It is clear that the family $\tT'$ of finite intersections of members of $\mathcal{F}$ is a  strongly Hausdorff $\bigvee$-definable group topology on $G$, and $\tT(M)$ and $\tT'(M)$ induce on $G(M)$ the same topology. But for every $\aleph_0$-saturated model $M$, the topologies  $\tT(M)$ and $\tT'(M)$ on $G(M)$ are different (as the later is Hausdorff, but the former is not).
\end{ex}

We return to the general context where  $\tT =\{ \varphi(x,\bar z_i): i \in I, \bar z_i\}$ is a $\bigvee$-definable group topology on $G$.
Recall that the group $G(M)$, with the topology induced by $\tT(M)$, is said to be definably topologically amenable if there is a (left) $G(M)$-invariant, Borel probability measure on $S_{\mu_M \backslash G}(M)$  (equivalently, on $S_G^{\mu_M}(M)$). 
A natural question arises, whether the definable topological amenability of $(G(M),\tT(M))$ is independent of the choice of $M$. The positive answer follows from Corollary \ref{corollary: absoluteness of the existence of measure} applied to the family $\eE:=\{G(x) \wedge G(y) \wedge \varphi_i(yx^{-1}, \bar z_i): i \in I, \bar z_i\}$; similarly, applying Corollary \ref{corollary: absoluteness of the existence of measure} to the family $\eE:=\{G(x) \wedge G(y) \wedge \varphi_i(x^{-1}y, \bar z_i): i \in I, \bar z_i\}$, we get item (2) of the following corollary.

\begin{cor}\label{corollary: absoluteness of definable topological amenability}
Let $\tT$ be a $\bigvee$-definable group topology on $G$. Then:
\begin{enumerate}
\item the definable topological amenability of $(G(M),\tT(M))$ does not depend on the choice of $M$,
\item the existence of a $G(M)$-invariant, Borel probability measure on $S_{G/\mu_M}(M)$ does not depend on the choice of $M$.
\end{enumerate}
\end{cor}

But the following question remains open.

\begin{ques}
1) Let $\tT$ be a $\bigvee$-definable group topology on $G$. Does amenability of $(G(M),\tT(M))$ as a topological group  depend on the choice of $M$?\\
2) Let $G$ be an arbitrary topological group. Let $\tT$ be the smallest topology among topologies $\bigvee$-definable in expansions of $(G,\cdot)$ which specialize to the given topology on $G$ (see Remark \ref{remark: smallest extension of topology}). Does amenability of $G$ (as a topological group) imply amenability of $(G(N),\tT(N))$ for $N \succ M$ (where $\tT$ is defined in $\Th(M)$).
\end{ques}

 It is clear that a negative answer to (1) implies a positive answer to (2). Question (2) is interesting, as it asks whether there is any chance to transfer (topological) amenability to elementary extensions. (Note that whenever we have two group topologies $\tT_1 \subseteq \tT_2$ on a given group $G$, then amenability of $(G,\tT_2)$ implies amenability of $(G,\tT_1)$.)

As before, $\tT=\{ \varphi(x,\bar z_i): i \in I, \bar z_i\}$ is a $\bigvee$-definable group topology on $G$.
%

\begin{dfn} \label{definition: premeanT}  1) A {\em right pre-mean} for $\tT_M$  is a $G$-pre-mean for $\eE_M$ (in the sense of Definition \ref{definition: G-premean for vee-definable equivalence relations}) with $\eE:=\{G(x) \wedge G(y) \wedge \varphi_i(x^{-1}y, \bar z_i): i \in I, \bar z_i\}$. Explicitly, it
is a  monotone  function $m$ on definable subsets of $G$  
into $[0,1]$,  with $m(\emptyset)=0$, $m(G)=1$,  
and  $m (Y \cup Y') \leq m (Y)+m (Y') $, such that equality holds whenever $YD \cap Y' = \emptyset$ for some $D \in \tT_M$.\\
2) A {\em left $G$-pre-mean} for $\tT_M$  is a $G$-pre-mean for $\eE_M$ (in the sense of Definition \ref{definition: G-premean for vee-definable equivalence relations}) with $\eE:=\{G(x) \wedge G(y) \wedge \varphi_i(yx^{-1}, \bar z_i): i \in I, \bar z_i\}$. Explicitly, it
is a  monotone  function $m$ on definable subsets of $G$  
into $[0,1]$,  with $m(\emptyset)=0$, $m(G)=1$,  
and  $m (Y \cup Y') \leq m (Y)+m (Y') $, such that equality holds whenever $(D^{g_1} \cap \dots \cap D^{g_n})Y \cap Y' = \emptyset$ for some $D \in \tT_M$ and $g_1,\dots,g_n \in G$.
\end{dfn}

Then Proposition \ref{proposition: main proposition} specializes to the following two statements.

\begin{cor}\label{corollary: added from the new subsection, no 1} 
The following conditions are equivalent.
\begin{enumerate}
\item $S_{G/\mu_M}(M)$ carries a $G(M)$-invariant, Borel probability measure.
\item There is a $G(M)$-invariant [$G$-invariant] right pre-mean for $\tT_M$.
\item The lattice $\dD_{\tT}^{\mu_M}$ carries a $G(M)$-invariant [$G$-invariant], normalized mean.
\end{enumerate}
\end{cor}

\begin{cor}\label{corollary: added from the new subsection, no 2}
The following conditions are equivalent.
\begin{enumerate}
\item $(G(M),\tT(M))$ is definably topologically amenable (i.e. $S_{\mu_M \backslash G}(M)$ carries a $G(M)$-invariant, Borel probability measure).
\item There is a $G(M)$-invariant [$G$-invariant] left $G$-pre-mean for $\tT_M$.
\item The lattice ${\dD_{\tT}^{\mu_M}}'$ carries a $G(M)$-invariant [$G$-invariant], normalized mean.
\end{enumerate}
\end{cor}

By Corollary \ref{corollary: absoluteness of the existence of mean}, we get that the existence of a left-invariant mean on $\dD^{\mu_M}_{\tT}$ [or on ${\dD^{\mu_M}_{\tT}}'$, respectively] is independent of the choice of both $M$ and $M^*$. Similarly, the existence of an invariant mean on $\dD_{\mathcal{C}}$ is independent of the choice of $M^*$. A question is whether the existence of an invariant mean on $\dD_\tT$ depends on the choice of $M^*$.

Along with Remark \ref{remark: closed}, Corollaries \ref{corollary: added from the new subsection, no 1} and \ref{corollary: added from the new subsection, no 2}  seem to suggest that one can get (from amenability) a $G$-invariant, normalized mean on the lattice $\dD_{\tT}$ of closed, type-definable sets; but we do not quite see this. It is certainly not true about $\dD_\mathcal{C}$. To see this, it is enough to take an amenable (as a topological group) but not definably amenable group $G(M)$ such that $G$ is strongly Hausdorff (as then $\dD_{\mathcal{C}}$ consists of all type-definable subsets of $G=G(M^*)$, so the restriction of an invariant mean defined on $\dD_{\mathcal{C}}$ to the algebra of all definable subsets would be a left-invariant Keisler measure, contradicting the failure of definable amenability). 
As a concrete example with these properties one can take the group $S_\infty=\Sym(\mathbb{N})$ with the usual topology, considered as a group definable in a standard model $(M, \in)$ of a sufficient fragment of set theory.

The following is a corollary of the proofs of Theorems \ref{theorem: c4} and \ref{theorem: indeed c4} applied for $H := \mu_M$; the set $\hat{D}$ from the conclusion below will be $p^4:=ppp^{-1}p^{-1}$ for a type $p \in S_G(M)$ which is wide in the sense that $m(D\mu_M)>0$ [resp. $m(\mu_MD)>0$] for every $D \in p$.
\begin{cor}    
Let $\tT$ be a $\bigvee$-definable group topology.
Assume $\dD^{\mu_M}_\tT$ [or ${\dD^{\mu_M}_\tT}'$, respectively] carries a $G$-invariant mean $m$.  Then  $G^{00}_M \leq G^{000}_M \langle \mu_M^G\rangle$.  
 More precisely, there exists an $M$-type-definable set $\widehat{D} \subseteq G^{000}_M$, with $\widehat{D}\langle \mu_M^G \rangle \supseteq G^{00}_M$.  In fact, for any wide, $M$-type-definable set $P=\mu_M P \mu_M$ we have $P^4:=PP P^{-1}P^{-1} \supseteq G^{00}_M$.
\end{cor}

The main result of this subsection is the the following

\begin{prop}\label{proposition: normal2}
Let $\tT$ be a $\bigvee$-definable group topology such that for all $n \in \mathbb{N}$
the projections (to all subproducts) of every type-definable, closed set in $G^n$ are closed. 
Assume $\dD_{\tT}$ carries a $G$-invariant mean $m$.     Then    $\cl(G^{00}_M )  = \cl(G^{000}_M)$.  
 More precisely, there exists an $M$-type-definable set $\widehat{D} \subseteq G^{000}_M$, with $\cl(\widehat{D}) = \cl(G^{00}_M ) $.   In fact, for any closed, wide (i.e. of positive mean), $M$-type-definable set $P$ we have $P^4:=PP P^{-1} P^{-1} \supseteq G^{00}_M$.
\end{prop}



\begin{proof}

We start from\\[-2mm] 

\begin{clm}\label{claim: basic properties of cl}
i)  The product of any two closed, type-definable sets is always closed (and clearly type-definable).\\
ii) For all type-definable sets $P$ and $Q$, $\cl(\cl(P) \cdot \cl(Q))=\cl(P \cdot Q)$.\\
iii) For all type-definable sets $P$ and $Q$, $\cl(P)\cdot \cl(Q) = \cl(P\cdot Q)$.\\
iv) For every type-definable set $P = \bigcap P_i$ (where $P_j \subseteq P_i$ whenever $i<j$), $\cl(P) = \bigcap_i \cl(P_i)$.
\end{clm}

\begin{clmproof}
i) This would follow immediately from the assumption that projections of closed, type-definable sets are closed if the topology induced by $\tT$ on $G=G(M^*)$ was Hausdorff. But if it is not Hausdorff, we can always pass to the Hausdorff quotient $G/\mathcal{C}$, where $\mathcal{C} =\cl(1)$. Working with $G/\mathcal{C}$ in place of $G$, we still have that projections of closed, type-definable sets are closed, so the product of any two closed, type-definable subsets of $G/\mathcal{C}$ is closed. Now, take any two closed, type-definable subsets $P$ and $Q$ of $G$. Then $P=P\mathcal{C}$ and $Q=Q\mathcal{C}$. So $P/\mathcal{C}$ and $Q/\mathcal{C}$ are closed, type-definable subsets of $G/\mathcal{C}$, and so $PQ/\mathcal{C} = P/\mathcal{C} \cdot Q/\mathcal{C}$ is closed in $G/\mathcal{C}$, hence $PQ$ is closed in $G$.\\[1mm]
ii) is a general property of all topological groups.\\[1mm]
iii) Using (ii), we immediately see that (iii) is equivalent to (i).\\[1mm]
iv) follows from Remark \ref{remark: definition of cl}.
\end{clmproof}

\begin{clm}  
For any closed, $M$-type-definable set $P$ with $m(P)>0$, there exists a generic, closed set $Q$ type-definable over some parameters and such that
$Q^8 \subseteq P^4$.
\end{clm}

\begin{clmproof}
We use Proposition \ref{proposition: variant of Massicot-Wagner}, with $G=X$ the present $G/\mathcal{C}$ (where $\mathcal{C} = \cl(1)$),  $A=B=P/\mathcal{C}$, $N=8$, $\dD$ being the lattice of closed, type-definable subsets of $G/\mathcal{C}$, and $m=m'$ being the pushforward of the mean $m$ from the statement of Proposition \ref{proposition: normal2}. (Item (i) of the first claim is used to ensure that the assumptions of  Proposition \ref{proposition: variant of Massicot-Wagner} hold.)  So there exists a generic, symmetric $\bar Q  \subseteq G/\mathcal{C}$ positively definable in $(G/\mathcal{}{C},\cdot, P/\mathcal{C})$, and with $\bar Q^8 \subseteq (P/\mathcal{C})^4$.  By the assumption that projections of closed, type-definable sets are closed (and the fact that $G/\mathcal{C}$ is Hausdorff), it follows that $\bar Q$ is closed and type-definable in the original structure $M^*$. So the pullback $Q$ of $\bar Q$ by the quotient map $G \to G/\mathcal{C}$ is also generic, closed, type-definable, and $Q^8 \subseteq P^4$.
\end{clmproof}

Since we are going to deal with $G^{00}_M$, we need to be more careful about parameters, and force $Q$ to be defined over $M$.

First, we will prove the last statement of Proposition \ref{proposition: normal2}, and then we will quickly explain how to deduce the previous one. 

So take any closed, wide, $M$-type-definable set $P$ (where wide means  that $m(P) >0$). Consider any $M$-definable set $P'$ containing $P^4$. 

By the first and last item of the first claim, we can find an $M$-definable set $P''$ such that $P^4 \subseteq P'' \subseteq \cl(P'') \subseteq P'$. Let $Q$ be a set provided by the second claim. We can find an $M$-definable, generic set $Q_0$ such that $Q_0^8 \subseteq P''$, and so, by item (iii) of the first claim, $\cl(Q_0)^8 =\cl(Q_0^8)\subseteq \cl(P'') \subseteq P'$. By the last item of the first claim, we can find an $M$-definable set $Q_1$ such that $\cl(Q_0) \subseteq Q_1$ and $\cl(Q_1)^8 \subseteq P'$. Put
$$C_1:= \cl(Q_1)^4.$$

Now, apply the above argument to $\cl(Q_0)$ (which is $M$-type-definable by Remark \ref{remark: definition of cl}) in place of $P$, and $Q_1^4$ in place of $P'$. As a result, we obtain $M$-definable, generic sets $R_0$ and $R_1$ such that $\cl(R_0) \subseteq R_1$ and $\cl(R_1)^8 \subseteq Q_1^4$. Put
$$C_2:= \cl(R_1)^4.$$

Continuing in this way, we obtain a sequence $C_1,C_2,\dots$ of $M$-type-definable, generic and symmetric subsets of $P'$ such that $C_{i+1}^2 \subseteq C_i$ for all $i$. Then $\bigcap_i C_i$ is a bounded index, $M$-type-definable subgroup contained in $P'$. Therefore, $G^{00}_M \subseteq P'$. Since $P'$ was an arbitrary $M$-definable set containing $P^4$, we conclude that $G^{00}_M \subseteq P^4$, which is the desired conclusion.

Let us prove now the existence of $\widehat{D}$. Let $p$ be a wide type of $G$ over $M$, in the sense that $m(\cl(D)) >0$ for any $D \in p$.  For every $D \in p$, by what we have just proved applied to $P:=\cl(D)$, we have $G^{00}_M \subseteq \cl(D)^4$. Hence, by the last item of the first claim, we get $G^{00}_M\subseteq \cl(p)^4$.  Put $\widehat{D}:=p^4$. It is clearly contained in $G^{000}_M$. On the other hand, by item (iii) of the first claim, $\cl(\widehat{D}) =\cl(p^4) = \cl(p)^4 \supseteq G^{00}_M$.
\end{proof}



 \begin{rem}\label{remark: last remark in 2.5}
The assumption in Proposition \ref{proposition: normal2} 
that the projections of closed, type-definable sets are closed may seem a bit artificial, perhaps it can be changed.  At any rate, it holds in each of the following two situations.
\begin{enumerate}
\item The situation from the last paragraph of Example \ref{Example: Vee-topology given by a type-definable subgroup}, namely: $H=\bigcap_{i \in I} X_i$ is a normal, type-definable subgroup of $G$ (and without loss $X_j \subseteq X_i$ when $i<j$), and $\tT := \{ X_i : i \in I\}$.
\item $\tT$ is a definable family and $(G(M),\tT(M))$ is compact (Hausdorff) for some model $M$.
\end{enumerate}
 \end{rem}

\begin{proof}
(1) follows from the observation that $F \subseteq G^n$ is closed if and only if $F=F \cdot \mathcal{C}^n$, where $\mathcal{C}=\cl(1)$.\\
(2) By the compactness and Hausdorffness of $G(M)$, the projections of any closed subset of $G(M)^n$ are closed. Thus, since $\tT=\{\varphi(x,\bar y): \bar y\}$ is a definable family, we easily get that the projections of any closed and definable subset $F$ of $G^n$ are closed. On the other hand, for any type-definable, closed set $F =\bigcap F_i \subseteq G^n$ (where $F_j \subseteq F_i$ whenever $i<j$), using the last item of the first claim of the proof of Proposition \ref{proposition: normal2}, we get that $F=\bigcap_i \cl(F_i)$ and each $\cl(F_i)$ is definable (by the definability of $\tT$), and, by compactness, any projection of $F$ is the intersection of the projections of the $\cl(F_i)$'s. So the conclusion follows.
\end{proof}

By virtue of Remark \ref{remark: last remark in 2.5}(1), the following obvious corollary of Theorem \ref{theorem: c4} also follows from Proposition \ref{proposition: normal2}.

\begin{cor}
 Let $N$ be any normal, $\emptyset$-type-definable subgroup of $G$.  
 Assume the lattice $\dD_N$ (of type-definable subsets $Y$ of $G$ such that $YN=N$) carries a $G$-invariant mean.  
 Then $G^{00}_M  \leq N G^{000}_M$.  
\end{cor}

\section{Definable actions, weakly almost periodic actions, and stability}\label{Section 3}

One aim of this section is to give a negative answer to Conjecture \ref{con: the most general} about {\em definable actions} of definable groups on compact spaces: see Corollary \ref{corollary: conjecture 0.3 from [KrPi] is false} below.
But we go rather beyond this, discussing the relationships between the notions in the title of the section. Weakly almost periodic actions (or flows) of a (topological) group $G$ on a compact space $X$ are important  in topological dynamics. Weak almost periodicity (for functions on a topological group) was introduced in \cite{Eberlein}, and discussed later in  \cite{Grothendieck}. We will be referring to  \cite{Ellis-Nerurkar} where weak almost periodicity of $G$-flows is defined and studied.  The connection of weak almost periodicity with stability  is by now fairly well-known, although much of  what is in print or published,  such as \cite{BY-T}  and \cite{Ibarlucia}, deals with the case where   the relevant group $G$ is the (topological) automorphism group of  a countable  $\omega$-categorical structure. In contrast, we  are here concerned with an action of a group $G(M)$ definable in a structure $M$ on a compact space $X$ where $G(M)$ is viewed as a discrete group, but where the action on $X$ is assumed to factor through the action of $G(M)$ on its space $S_{G}(M)$ of types over $M$.

We will give some background below on both continuous logic (in an appropriate form) and weak almost periodicity. The connection between stability in continuous logic and weak almost periodicity goes through results of Grothendieck \cite{Grothendieck} in functional analysis, which have been commented on in several expository papers such as \cite{BY} and later \cite{Pillay}.   However, it is relative stability, namely stability of a formula {\em in a model $M$} which is relevant, and  only  equivalent to stability when the model is saturated enough. 

One of our main  structural results  is Theorem \ref{theorem: equivalence of wap, definability, and stability} below  characterizing when the action of $G(M)$ on $X$ is weakly almost periodic in terms of stable in $M$ formulas. When $M$ is $\omega_{1}$-saturated, another equivalent condition is that the action of $G(M)$ on $X$ is {\em definable}, which will yield  the desired conclusions (Theorem \ref{theorem: each saturated group is weakly definably amenable} and Corollary \ref{corollary: conjecture 0.3 from [KrPi] is false}).   

Although this is a model theory paper, it is convenient for us to quote heavily from the topological dynamics literature, especially for results which have not yet been developed in the parallel model-theoretic environment. 

We will generally be assuming any ambient theory $T$ to be countable.

\vspace{5mm}
\noindent

The notion of a definable  action of a definable group  on a compact space  was given in \cite{GPP} and explored in some degree of generality in \cite{Kr}.  We repeat the definition below. As was said in the introduction, it would be more appropriate to call it a ``separately definable action'', but for simplicity we are saying ``definable action''. 

\begin{dfn}\label{definition: definable function and action} (i) Let $X$ be a  set definable over $M$. A function $f$ from $X(M)$ to a compact space $C$ is said to be {\em definable} if
 if for every pair $C_{1}$, $C_{2}$ of closed disjoint subsets of $C$, there is a definable (in $M$) set $Z\subseteq X(M)$ such that $f^{-1}(C_{1})\subseteq Z$, and $f^{-1}(C_{2})\subseteq G(M)\setminus Z$. 
\newline
(ii)  Suppose $G$ is a group definable over $M$. A group action by $G(M)$ on  a compact space $X$ by homeomorphisms is said to be {\em definable} if for every $x\in X$ the map from $G$ to $X$ taking $g$ to $g\cdot x$ is definable.
\end{dfn}

When all types over $M$ are definable, then the natural action of $G(M)$ on $S_{G}(M)$ is a definable action and is moreover the universal definable $G(M)$-ambit (see \cite{GPP}).
This is interesting for structures $M$ such as the reals or $p$-adics where all types over $M$ are definable, although the complete theories are unstable. 
However, in general, definability of an action of $G(M)$ on a compact space $X$  is a rather restrictive condition.  In \cite{Kr}, it was shown that there is always a universal definable 
$G(M)$-ambit (which will of course factor through $S_{G}(M)$).  
Recall from Definition \ref{def: definition of weak definable amenability} that $G(M)$  is said to be {\em weakly definably amenable} if whenever $G(M)$ acts definably on a compact space $X$, 
then $X$ supports a $G(M)$-invariant, Borel probability measure, equivalently the universal definable $G(M)$-ambit supports a $G(M)$-invariant, Borel probability measure. 
A special case of Conjecture \ref{con: the most general} says that if  $G(M)$ is weakly definably amenable, then $G^{00}_{M} = G^{000}_{M}$.  
At the end of Subsection \ref{subsec: Subsection 3.2}, we will show that this fails drastically, by proving that when $M$ is sufficiently saturated, then $G(M)$ is {\em always} weakly definably amenable.

\begin{thm}\label{theorem: each saturated group is weakly definably amenable}  Suppose $M$ is $\omega_{1}$-saturated. Then $G(M)$ is weakly definably amenable: for any definable action of $G(M)$ on a compact space $X$, $X$ supports a $G(M)$-invariant, Borel probability measure.
\end{thm}

We deduce a negative answer to Conjecture \ref{con: the most general} (i.e. Conjecture 0.3 of \cite{KrPi}):
\begin{cor}\label{corollary: conjecture 0.3 from [KrPi] is false} There is a model $M$, and a group $G(M)$ definable in $M$ such that $G(M)$ is weakly definably amenable, but $G^{00}_{M} \neq G^{000}_{M}$.
\end{cor}
\begin{proof} In fact, whenever $G$ is a group definable in a NIP theory $T$ and $G^{00} \neq G^{000}$, then  choosing an $\omega_{1}$-saturated model $M$ of $T$, we see from  Theorem \ref{theorem: each saturated group is weakly definably amenable} that $G(M)$ is weakly definably amenable. Moreover $G^{00}_{M} = G^{00} \neq G^{000} = G^{000}_{M}$.  There are many such examples, such as from \cite{CP}:  $T$ is the theory of the $2$-sorted structure $M$ with sorts $(\R,+,\times)$ and $({\mathbb Z},+)$ and no additional structure.  As pointed out there, the universal cover of $\Sl(2,\R)$ is naturally definable in  $M$. $T$ is NIP, and if $G$ is the interpretation of this group in a saturated model, then $G^{00} \neq G^{000}$.
\end{proof}


\subsection{Continuous logic}

Continuous logic is about real-valued relations and formulas, or, more generally, formulas with values in compact spaces, and, as such,  is present in a lot of recent work which does not explicitly mention continuous logic (even in Definition \ref{definition: definable function and action} above).

There have been various  approaches  to continuous logic, starting with \cite{CK}.   An  attractive formalism was developed in  \cite{B-BY-H-U} and  \cite{BY-U}, and our set up will be  a special case. 
Here, we will give relatively self-contained proofs, for  reasons explained below.  

$T$ will be a complete  first order theory in the usual (non-continuous) sense, which is countable (for convenience)  and we work as earlier in a big saturated (or monster)  model $\C$.  We fix a sort $X$ (which will be a definable group $G$ in the applications).  As usual, $M, N, \dots$ denote small elementary submodels of $\C$, and $A, B, \dots$   small subsets of this monster model. There is no harm assuming that $T = T^{eq}$. 

\begin{dfn}\label{definition: stable formulas} (i) By a {\em continuous logic (CL) formula on $X$ over $A$}, we mean a continuous function $\phi \colon S_{X}(A)\to \R$.  
\newline
(ii) If $\phi$ is such a CL-formula, then for any $b\in X$ (in the monster model) by $\phi(b)$ we mean $\phi(\tp(b/A))$.  Hence, we have a map 
$\phi\colon X(N) \to \R$ for all models $N$, in particular a map $\phi\colon  X = X(\C) \to \R$.  As the notation suggests, we are identifying a CL-formula on $X$ over $A$ with the latter map, and so may write it as $\phi(x)$ where $x$ is a variable of sort $X$. 
\newline
(iii) We consider two such CL-formulas on $X$, $\phi$, $\psi$,   over sets $A, B$, respectively to be {\em equivalent} if they agree in the sense of (ii), namely if for all $a\in X$, $\phi(a) = \psi(a)$.
\end{dfn}

\begin{rem}\label{remark: basics on CL-formulas} (i) The range of any CL-formula is a compact subset of $\R$.
\newline
(ii) A CL-formula $\phi$ (over some $A$) is equivalent to a CL-formula over $B$ if $\phi$ is invariant under automorphisms of the monster model which fix $B$ pointwise. 
\newline
(iii) The maps $\phi\colon X(M)\to \R$ given by CL-formulas $\phi$ over $M$ are precisely the {\em definable} maps from $X(M)$ to $\R$ in the sense of Definition \ref{definition: definable function and action}.
\newline
(iv) Any CL-formula over a set $A$  is (equivalent to a CL-formula) over a countable subset of $A$.
\end{rem}

\begin{dfn} (i) Let $M$ be a model, and $\phi(x,y)$ a CL-formula over $M$, where $x,y$ are variables of sorts $X,Y$, respectively. Let $a\in X$. 
Then $\tp_{\phi}(a/M)$ is the function taking $b\in Y(M)$ to $\phi(a,b)$, and is called a {\em complete $\phi(x,y)$-type over $M$}. 
\newline
(ii) In the context of (i), $\tp_{\phi}(a/M)$ is {\em definable} if it is definable in the sense of Definition \ref{definition: definable function and action}, equivalently, by Remark \ref{remark: basics on CL-formulas}(iii), given by or rather induced by a CL-formula on $Y$ over $M$.
\end{dfn}

\begin{rem} Suppose $M$ is $\omega_{1}$-saturated, $\phi(x,y)$ is a CL-formula over $M$, and $a$ is in the big model. Then $\tp_{\phi}(a/M)$ is definable if and only if for each closed subset $C$ of $\R$, $\{b\in Y(M):\phi(a,b)\in C\}$ is type-definable over some countable subset of $M$. 
\end{rem}

\begin{proof} 
This follows from Remark \ref{remark: basics on CL-formulas}(iv). 
\end{proof}

\begin{dfn} Let $\phi(x,y)$ be a CL-formula over $M$.\\ 
(i) We say that $\phi(x,y)$ is {\em stable} (for the theory $T$)  if for all $\epsilon > 0$ there do not exist $a_{i}, b_{i}$ for $i<\omega$ (in the monster model) such that for all $i<j$, $|\phi(a_{i},b_{j}) - \phi(a_{j},b_{i})| \geq \epsilon$.
\newline
(ii) We say that $\phi(x,y)$ is {\em stable in $M$} if for all $\epsilon> 0$ there do not exist $a_{i}, b_{i}$ for $i<\omega$ {\em in M} such that for all $i<j$, $|\phi(a_{i},b_{j}) - \phi(a_{j},b_{i})| \geq \epsilon$
\end{dfn}

\begin{rem}\label{remark: equivalent definitions of stability}
(i)  Routine methods show that $\phi(x,y)$ is stable (for $T$)  iff  whenever $(a_{i},b_{i})_{i<\omega}$ is indiscernible (over $M$), then $\phi(a_{i},b_{j}) = \phi(a_{j},b_{i})$ for $i< j$. 
\newline
(ii) On the other hand, stability of $\phi(x,y)$ in $M$ is easily seen to be equivalent to Grothendieck's double limit condition: given $a_{i}, b_{i}$ in $M$ for $i<
\omega$ we have that  $\lim_{i}\lim_{j}\phi(a_{i},b_{j}) = \lim_{j}\lim_{i}\phi(a_{i},b_{j})$ if both double limits exist.
\newline
(iii) A CL-formula $\phi(x,y)$ is stable [in $M$] iff $\phi^{op}(x,y):=\phi(y,x)$ is stable [in $M$]. 

\end{rem}

The following is due to Grothendieck (modulo a routine translation), and we give an explanation below.

\begin{prop}\label{proposition: stability in a model} Let $\phi(x,y)$ be a CL-formula over $M$. Then the following are equivalent.
\begin{enumerate}
\item[(i)]  $\phi(x,y)$ is stable in $M$.
\item[(ii)] Whenever $M \prec M^{*}$, and $\tp(a/M^{*})\in S_{x}(M^{*})$ is finitely satisfiable in $M$, then $\tp_{\phi}(a/M^{*})$ is definable over $M$,  namely the function taking $b\in M^{*}$ to $\phi(a,b)$ is given by a $CL$-formula $\psi(y)$ over $M$. 
\end{enumerate}
\end{prop}

\begin{proof}  
Consider the (compact) space $Z = S_{y}(M)$ of complete types  over $M$ in variable $y$, and let $C(Z)$ be the real Banach space of continuous real valued functions on $Z$ (equipped with the supremum norm).  Let $A$ denote the subset of $C(Z)$ consisting of the functions  $\phi(a,y)$ for $a\in M$.  Note that $A$ is bounded.  Let $Z_{0}$ be the set  of realized types, namely those 
$\tp(b/M)$ for $b\in M$, a dense subset of $Z$. With this notation,  Grothendieck's Theorem 6 in \cite{Grothendieck} says that the following are equivalent.
\begin{enumerate}
\item[(i)']  If $f_{i}\in A$ and $q_{i}\in Z_{0}$ for $i<\omega$, then $\lim_{i}\lim_{j}f_{i}(q_{j}) = \lim_{j}\lim_{i}f_{i}(q_{j})$ if both double limits exist.
\item[(ii)'] The closure of $A$ in the pointwise convergence topology on $C(Z)$ is compact.
\end{enumerate}

Now, if $f_{i}$ is $\phi(a_{i},y)$ and $q_{i} = \tp(b_{i}/M)$, then (i)'  says precisely that  $\lim_{i}\lim_{j}\phi(a_{i},b_{j}) = \lim_{j}\lim_{i}\phi(a_{i}.b_{j})$ for all sequences $a_{i},b_{i}\in M$ with $i<\omega$ for which both double limits exist, which by Remark \ref{remark: equivalent definitions of stability}(ii) says that $\phi(x,y)$ is stable in $M$, namely condition (i) in the proposition. 

On the other hand (ii)' implies that the closure of $A$ in $C(Z)$ (in the pointwise topology) is a compact, so closed,  subset of the space $\R^{Z}$ of {\em all}  functions from $Z$ to $\R$ (equipped with the pointwise, equivalently Tychonoff topology). So every function in the closure of $A$ in $\R^{Z}$ is already in $C(Z)$, so is continuous.  So it is clear that (ii)' is equivalent to 

\begin{enumerate}
\item[(ii)''] whenever $f\in \R^{Z}$ is in the closure of $A$, then $f$ is continuous. 
\end{enumerate}

It is now easy to see that if $f\in \R^{Z}$ is in the closure of $\{\phi(a,y):a\in M\}$, then $f$ is of the form $\phi(a^{*},y)$,  where $M^{*}$ is a saturated model containing $M$, and $\tp(a^{*}/M^{*})$ is finitely satisfiable in $M$. So
for $q\in Z = S_{y}(M)$, $f(q) = \phi(a^{*},b)$ for some (any) realization $b$ of $q$ in $M^{*}$.  The continuity of $f$ means that it is given by a CL-formula $\psi(y)$ over $M$, which precisely means that $\psi(y)$ is a definition over $M$ of $\tp_{\phi}(a^{*}/M^{*})$.  So we get that (ii) implies (ii)'', and it is again easy to see that they are equivalent. 
\end{proof}

\begin{rem} (a) Actually the original statement of (ii)' in \cite{Grothendieck} is that the closure of $A$ in the {\em weak} topology on $C(Z)$ is compact. The weak topology on $C(Z)$ is the one whose basic open neighbourhoods of a point $f_{0}$ are of the form $\{f\in C(Z): |g_{1}(f-f_{0})| < \epsilon, \dots, |g_{r}(f-f_{0})| < \epsilon\}$, where $g_{1},\dots,g_{r}$ are in $L(C(Z),\R)$ --- the space of bounded linear functions on $C(Z)$.  This weak topology is stronger than the pointwise convergence topology on $C(Z)$ whose basic open neighbourhoods of a point $f_{0}$ are as above but where $g_{i}$ is evaluation at some point $x_{i}\in Z$.  It is pointed out in \cite{Grothendieck} that relative compactness of a bounded subset $A$ of $C(Z)$ in the weak topology is equivalent to relative compactness of $A$ in the pointwise convergence topology,  yielding the statement (ii)' in the proof of Proposition \ref{proposition: stability in a model}. 
\newline
(b)  In \cite{BY}, which seems to be the first model theory paper to recognize Grothendieck's contribution, only the implication ``$\phi(x,y)$ stable in $M$ implies that all $\phi$-types over $M$ are definable'' is deduced from Grothendieck's theorem, rather than the stronger equivalence in Proposition \ref{proposition: stability in a model}. 
\newline
(c)  Grothendieck's proof in \cite{Grothendieck} is basically a model theory proof.  See \cite{Pillay} for the case of classical ($\{0,1\}$-valued) formulas.

\end{rem}

\begin{prop}\label{prop: stability iff definability over all models}
The CL-formula $\phi(x,y)$ is stable (for $T$)  if and only if every complete $\phi(x,y)$-type over any model over which $\phi$ is defined is definable.
\end{prop}
\begin{proof}   In the more general metric structures formalism, this appears in \cite{BY-U} (Proposition 7.7 there) and  adapts to our context. However, we give a relatively self contained  account. Left implies right is given by Proposition \ref{proposition: stability in a model}.  The other direction is the easy one and can be seen as follows.  Assume $\phi(x,y)$ to be unstable (for a contradiction).  By (or as in) Remark \ref{remark: equivalent definitions of stability}, we can find  $a_{i},b_{i}\in \C$ for $i\in \mathbb{Q}$, and real numbers $r < s$ that  $\phi(a_{i},b_{j}) \leq r$ for $i<j$ and $\phi(a_{i},b_{j}) \geq s$ for $i>j$.  
Let $M$ be a countable model containing the $b_{i}$ for $i\in\mathbb{Q}$ over which $\phi$ is defined.
By compactness, for each cut $C$ in $\mathbb{Q}$ there is some $a_{C}\in \C$ such that $\phi(a_{C},b_{j})\geq s$ for $j<C$ and  $\phi(a_{C},b_{j})\leq r$ for $j>C$. Now, by assumption, each $\tp_{\phi}(a_{C}/M)$ is definable, so for each $C$ there is some (ordinary) formula $\psi_{C}(y)$ over $M$ such that for any $b\in M$, $\phi(a_{C},b)\leq r$ implies $\psi_{C}(b)$, and $\phi(a_{C},b)\geq s$ implies $\neg\psi_{C}(b)$. This is a contradiction, as there are continuum many distinct $C$'s but only countably many (ordinary) formulas over the countable model $M$. 
\end{proof}

\begin{prop}\label{proposition: saturation and definability implies stability} Suppose $M$ is $\omega_{1}$-saturated, $\phi(x,y)$ is a CL-formula over $M$, and every complete $\phi(x,y)$-type over $M$ is definable.
Then every complete $\phi(x,y)$-type over any model $N$ (over which $\phi$ is defined) is definable, and hence, by Proposition \ref{prop: stability iff definability over all models}, $\phi(x,y)$ is stable. 
\end{prop}

\begin{proof}  Let $A\subset M$ be countable such that $\phi(x,y)$ is over $A$. 
By Proposition \ref{prop: stability iff definability over all models} and by the proof of the right to left implication in Proposition \ref{prop: stability iff definability over all models}, it suffices to prove that every complete $\phi$-type over a countable model containing $A$ is definable. As $M$ is $\omega_{1}$-saturated, it is enough to prove that every complete $\phi$-type over any countable submodel $M_{0}$ of $M$ which contains $A$ is definable.
So let $p(x)$ and $M_{0}$ be such. Let $p'$ be a coheir of $p$ over $M$, namely $p' = \tp_{\phi}(a/M)$, $p = p'|M_{0}$, and $\tp(a/M)$ is finitely satisfiable in $M_{0}$.  By our assumptions, $p'$ is definable.  So to prove that $p$ is definable it suffices to prove:\\[-2mm]

\begin{clm}
$p'$ is definable over $M_{0}$.
\end{clm}

\begin{clmproof}
Let $C$ be a compact subset of $\R$, and let $\Psi(y,b)$ be a partial type over a countable sequence $b$ from $M$ such that for all $c\in M$, $\phi(a,c)\in C$ iff
$M\models \Psi(c,b)$.  We will show that in fact $\Psi(y,b)$ is equivalent to a partial type over $M_{0}$.  For this it is enough to show that if 
$b'$ realizes $\tp(b/M_{0})$ in $M$, then 
$\Psi(y,b')$ is equivalent to $\Psi(y,b)$. 

Suppose $c'\in M$ and suppose $M\models \Psi(c',b')$. Let $c\in M$ be such that $\tp(c,b/M_{0}) = \tp(c',b'/M_{0})$. As $\tp(a/M)$ is finitely satisfiable in $M_{0}$, $\phi(a,c) = \phi(a,c')$. As $M\models \Psi(c,b)$, we have that $\phi(a,c)\in C$. Hence, $\phi(a,c')\in C$, whereby $M\models \Psi(c',b)$.  Hence, $\Psi(y,b')$ is equivalent to $\Psi(y,b)$, as required. This finishes the proof of the claim.  
\end{clmproof}
Hence, the proof of the proposition is also finished. \end{proof}

\subsection{Weakly almost periodic actions}\label{subsec: Subsection 3.2}

The context here is a $G$-flow $(X,G)$, where $X$ is a compact space and $G$ a topological group.  For $f$ a continuous function from $X$ to $\R$ and $g\in G$, $gf$ denotes the (continuous) function taking $x\in X$ to $f(gx)$. 
We will take our definition of a weakly almost periodic $G$-flow  from Theorem II.1 of \cite{Ellis-Nerurkar}.
\begin{dfn} (i) A continuous function $f\colon X\to \R$ is {\em weakly almost periodic} (or wap) if whenever $h\in \R^{X}$ is in the closure of $\{gf:g\in G\}$ (in the pointwise convergence topology) then $h$ is continuous.
\newline
(ii) The $G$-flow $(X,G)$ is {\em weakly almost periodic} (or wap),  if every continuous function $f\colon X\to\R$ is weakly almost periodic. 
\end{dfn}

\begin{fct}\label{fact: wap systems carry measures} Suppose that $(X,G)$ is wap. Then there is a $G$-invariant, Borel probability measure on $X$.
\end{fct}
\begin{proof}  This is well-known within topological dynamics, but we nevertheless give an account with some references.  We may assume that $(X,G)$ is minimal  (by passing to a minimal subflow).  By Proposition II.8 of \cite{Ellis-Nerurkar}, the flow 
$(X,G)$ is almost periodic (also known as equicontinuous).  The minimal equicontinuous flows have been classified in \cite{Auslander} for example (see \cite[Chapter 3, Theorem 6]{Auslander}), as homogeneous spaces for compact groups (on which $G$ acts as subgroups of the compact groups in question), whereby the Haar measure induces the required $G$-invariant measure on $X$.
\end{proof}

We now pass to the model-theoretic context, which here means that we consider actions of a definable group $G(M)$ on a compact space $X$ which factor through $S_{G}(M)$.

\begin{thm}\label{theorem: equivalence of wap, definability, and stability}
 Let $M$ be a structure, $G(M)$ a group definable in $M$, and let a $G(M)$-flow $(X,G(M))$ be given, which factors through the action of $G(M)$ on $S_{G}(M)$ via a continuous surjective ($G(M)$-equivariant) map $\pi\colon S_{G}(M)\to X$. Consider the following three conditions:
\begin{enumerate}
\item[(i)] $(X,G(M))$ is wap,
\item[(ii)]  for each continuous function $F\colon  S_{G}(M) \to \R$ of the form $f\circ \pi$ for $f\colon X\to \R$ continuous, the CL-formula $F(yx)$  is stable in $M$,
\item[(iii)] the action of $G(M)$ on $X$ is definable.
\end{enumerate}
Then: 
\begin{enumerate}
\item[(a)] (i) and (ii) are equivalent, and imply (iii),
\item[(b)]  if $M$ is $\omega_{1}$-saturated, then  (i), (ii), (iii) are equivalent, and, moreover, in (ii) we have that $F(yx)$ is stable for $T$ (not just in $M$). 
\end{enumerate}
\end{thm}

\begin{proof} (a)  Suppose $(X,G(M))$ is wap.  Let $F = f\circ\pi$ for some $f\in C(X)$.  Let  $h\colon S_{G}(M)\to \R$ be in the pointwise closure of $\{gF: g\in G(M)\}$.  Then clearly for any $p\in S_{G}(M)$, $h(p)$ depends only on $\pi(p)$, so $h = h_{1}\circ \pi$ for a unique $h_{1}\colon X\to \R$. But $h_{1}$ is in the pointwise closure of $\{gf:g\in G(M)\}$, so, by assumption, $h_{1}$ is continuous. Hence, $h$ is continuous.  
By the proof of Proposition \ref{proposition: stability in a model}, or, more precisely, by the equivalence of (i) and (ii)'' in there, the CL-formula $F(xy)$ is stable in $M$, and so is $F(yx)$ by Remark \ref{remark: equivalent definitions of stability}(iii).

The converse goes the same  way: Let $f\in C(X)$, and  $h\colon X\to \R$ be in the closure, again in the pointwise topology,  of $\{gf:g\in G(M)\}$.  Let $F  = f\circ\pi  \in C(S_{G}(M))$. Let $h_{1} = h\circ \pi$. Then clearly $h_{1}$ is in the closure of $\{gF: g\in G(M)\}$.  
As $F(yx)$ is assumed to be stable in $M$, by  Remark \ref{remark: equivalent definitions of stability}(iii) and the equivalence of (i) and (ii)'' in the proof of Proposition \ref{proposition: stability in a model}, $h_{1}$ is continuous, and so $h$ is continuous. 

So far we have shown (i) if and only if (ii).  We now show that either of these equivalent conditions imply that the action of $G(M)$ on $X$ is definable.
Let $x_{0}\in X$.\\[-2mm]

\begin{clm}
For any continuous function $f\colon X\to \R$, the function from $G(M)\to \R$ taking $g$ to $f(gx_{0})$ is definable  (over $M$). 
\end{clm}

\begin{clmproof} Let $p\in S_{G}(M)$ be such that $\pi(p) = x_{0}$.  Consider the lift $F$ of $f$ to $S_{G}(M)$ via $\pi$.  We use $x,y$ to denote variables of sort $G$.  By (ii), the formula $F(yx)$ (in variables $x,y$) is stable in $M$, so, by Proposition  \ref{proposition: stability in a model}, the function taking $g\in G(M)$ to $F(gp)$ is definable over $M$, namely induced by a CL-formula $\psi(y)$ over $M$.  But $F(gp) = f(gx_{0})$. Hence, the claim is proved. 
\end{clmproof}

Definability of the action of $G(M)$ on $X$ now follows from the claim and Urysohn's lemma:    Let $X_{0}$, $X_{1}$ be disjoint closed subsets of $X$. By Urysohn, there is a continuous 
function $f\in C(X)$ such that $f$ is $0$ on $X_{0}$ and $1$ on $X_{1}$.  By the claim, there is some definable (in $M$) subset $Z$ of $G(M)$, such that  for all $g\in G(M)$, if 
$f(gx_{0}) = 0$ then $g\in Z$, and if $f(gx_{0}) = 1$ then $g\notin Z$. But this implies that if $gx_{0}\in X_{0}$ then $g\in Z$, and if $gx_{0}\in X_{1}$ then $g\notin Z$. As 
$x_{0}\in X$ was arbitrary, this shows that the action of $G(M)$ on $X$ is definable.

\vspace{5mm}
\noindent
(b) We assume now that $M$ is $\omega_{1}$-saturated. All we have to do is to prove that (iii) implies the stronger version of (ii) (with stability for $T$).  Now, exactly as in the previous paragraph, definability of the action of $G(M)$ on $X$ means precisely that whenever $F\colon S_{G}(M) \to \R$ lifts some continuous function $f$ on $X$, then every complete $F(yx)$-type over $M$ is definable. By Proposition \ref{proposition: saturation and definability implies stability}, each such $F(yx)$ is stable (for $T$). 
\end{proof}

\begin{proof}[Proof of Theorem \ref{theorem: each saturated group is weakly definably amenable}] We may assume that  $X$ is a (definable) $G(M)$-ambit, in which case, by \cite{GPP} or \cite[Remark 3.2]{Kr}, the action factors through the action of $G(M)$ on $S_{G}(M)$. By Theorem \ref{theorem: equivalence of wap, definability, and stability}(b), and $\omega_{1}$-saturation of $M$,  the action of $G$ on $X$ is wap, so, by Fact \ref{fact: wap systems carry measures}, $X$ has a $G(M)$-invariant,  Borel probability measure. 
\end{proof}

\subsection{On universal ambits and minimal flows}

We give a description of the universal definable wap ambit and universal minimal definable wap flow for a group $G(M)$ definable in a structure $M$.  As seen by the material above, this is closely connected to stable group theory in the continuous logic sense, but unless $M$ is saturated enough, it will be stability in $M$. Actually, even in  the classical case,  stable group theory relative to a model $M$ (i.e. where relevant formulas $\phi(x,y)$ are stable in $M$) has not been written down, so it is not surprising if we happen to rely on the topological dynamical literature.   By $G$ we mean $G(M^{*})$ for a suitably saturated elementary extension $M^{*}$ of $M$. 

$M$ will be an arbitrary structure and $G(M)$ a group definable in $M$. Following on from notation in the previous section, if $F(x)$ is a CL-formula on $G$ (i.e. where the variable $x$ ranges over $G$) over $M$, then we will call $F$ {\em stable in $M$}  if the CL-formula $F(yx)$ (in variables $x,y$)  is stable in $M$. Let $\mathcal A$  be the collection (in fact algebra) of such stable in $M$, CL-formulas $F(x)$ on $G$.
Let ${\mathcal S}$ be the quotient of $S_{G}(M)$ by the closed equivalence relation $\sim_{{\mathcal A}}$ given by
$$p \sim_{{\mathcal A}} q \iff (\forall F \in {\mathcal A}) (F(p)=F(q)).$$ 
${\mathcal S}$ is naturally a compact space which we call the type space over $M$ of the stable in $M$, CL-formulas over $M$.
Let $\pi_{0}\colon S_{G}(M)\to\mathcal{S}$ be the canonical surjective continuous map. Note that $G(M)$ acts on $\mathcal S$, and that $\pi_{0}$ is a map of $G(M)$-flows (in fact ambits, where $\pi_{0}(e)$ is taken as the distinguished point of $\mathcal S$).  With the above notation we have:

\begin{prop}\label{proposition: universal definable wap ambit}
(i) $({\mathcal S}, G(M))$  is the (unique) universal definable wap ambit of $G(M)$.
\newline
(ii) 
 $G/G^{00}_{M}$ is the unique universal minimal definable wap flow of $G(M)$.
\end{prop}
\begin{proof} Let  us first note:\\[-2mm]

\begin{clm}
With above notation, a continuous function $F\colon S_{G}(M)\to \R$ is stable in $M$  if and only if it is induced, via $\pi_{0}$,  by a continuous function from $\mathcal S$ to $\R$.  
\end{clm}

\begin{clmproof} 
%
This follows from the Stone-Weierstrass theorem and the easy fact that ${\mathcal A}$ is a closed subalgebra of the Banach algebra $C(S_G(M))$ of all real valued continuous functions on $S_G(M)$ (where $C(S_G(M))$ is equipped with the uniform convergence topology).
\end{clmproof}
\noindent
(i) follows easily from the claim and previous results.  First, by Theorem \ref{theorem: equivalence of wap, definability, and stability} and the claim, $({\mathcal S}, G(M))$ (with distinguished point $s_{0} = \pi_{0}(e)$)  is definable and wap.  Secondly, suppose $(X,G(M))$ is a definable wap ambit with distinguished point $x_{0}$, and corresponding $\pi \colon S_{G}(M)\to X$ (taking $e$ to $x_{0}$).  By Theorem \ref{theorem: equivalence of wap, definability, and stability} again, for every continuous function $f$ on $X$, $F= f\circ \pi$ is stable, hence is in the algebra $\mathcal A$.  This easily induces a surjective, continuous, $G(M)$-equivariant map from $\mathcal S$ to $X$ taking $s_{0}$ to $x_{0}$.
\newline
(ii)  The action of $G(M)$ on $G/G^{00}_{M}$ is induced by multiplication on the left. Clearly every orbit is dense (i.e. the $G(M)$-flow $G/G^{00}_M$ is minimal), in particular the image of $G(M)$ in $G/G^{00}_{M}$ under the canonical homomorphism $\iota$
  taking $g$ to $g/G^{00}_{M}$ is dense.  The action factors through the type space. Why is it wap?  Let $f$ be a continuous function from $G/G^{00}_{M}$ to $\R$, and $F\colon G\to \R$ be  $f\circ\pi$, where $\pi\colon G\to G/G^{00}_{M}$ is the canonical homomorphism.  
So $F$ is a CL-formula on  $G$ over $M$. 
We claim that the CL-formula $F(yx)$ in variables $x,y$  (so on $G\times G$) is stable for the theory, in particular $F$ is stable in $M$. 
If not, we can  find a large indiscernible over $M$ sequence  $((g_{i},h_{i}): i\in I)$ such that for $i<j$, $F(g_{i}h_{j}) \neq  F(g_{j}h_{i})$; but this is 
impossible, as $\tp(g_i/M)=\tp(g_j/M)$ and  $\tp(h_i/M)=\tp(h_j/M)$ implies $\pi(g_i)=\pi(g_j)$ and   $\pi(h_i)=\pi(h_j)$, and so $\pi(g_{i}h_{j}) =  \pi(g_{j}h_{i})$.
Thus, using Theorem \ref{theorem: equivalence of wap, definability, and stability}, we have shown that the action of $G(M)$ on $G/G^{00}_{M}$ is a minimal  wap flow which factors through $S_{G}(M)$, so is also definable by Theorem \ref{theorem: equivalence of wap, definability, and stability}.

To see that it is universal such, we will appeal  again to the topological dynamics literature.  So let $(X,G(M))$ be a minimal definable wap flow. As already remarked, we deduce from II.8 of \cite{Ellis-Nerurkar} that the flow  $(X,G(M))$ is equicontinuous. By Theorem 3.3 from Chapter I of \cite{Glasner}, the Ellis semigroup  $E(X)$ is a compact topological group acting by homeomorphisms
on $X$, and, moreover, $(X,G(M))$ is isomorphic to $E(X)/H$ for a suitable closed subgroup $H$ of  $E(X)$ (with the action of $G(M)$ on $E(X)/H$ given by $g(\eta H)=(g\eta)H$).  So it  remains to show that the natural homomorphism $h\colon G(M) \to E(X)$ is definable, because in that case  $E(X)$ will be a definable group compactification  of $G(M)$  (in the sense of  \cite{GPP}) and we know from Proposition 3.4 of \cite{GPP}  that $G/G^{00}_{M}$ is the universal such definable group compactification. 

The fact that $h\colon G(M)\to E(X)$ is definable follows from the fact that $E(X)$ is a subflow of the product
$G(M)$-flow $X^X$ which is definable (because a product of definable flows is always definable \cite[Remark 1.12]{KrPi0}).

Finally, the uniqueness of a universal minimal definable wap $G(M)$-flow follows from the observation every endomorphism of the $G(M)$-flow $G/G^{00}_{M}$ is an automorphism, namely it is right translation by an element of $G/G^{00}_{M}$.
%
%
%
\end{proof}

The above proposition together with Theorem \ref{theorem: equivalence of wap, definability, and stability} yields

\begin{cor}\label{corollary: universal definable ambit in the saturated context} When $M$ is $\omega_{1}$-saturated, the universal definable wap ambit coincides with the universal definable ambit and can be described as the type space of the collection (algebra) of CL-formulas $F$ on $G$ over $M$ which are stable (for $T$). 
\end{cor}

We can also give a description of the universal definable $G(M)$-ambit for an arbitrary (not necessarily $\omega_1$-saturated) $M$. For this recall that in the proof of Theorem \ref{theorem: equivalence of wap, definability, and stability} (Claim 1 and the paragraph afterwards; see also the proof of (b)) we showed that definability of the action of $G(M)$ on $X$ means precisely that whenever $F\colon S_{G}(M) \to \R$ lifts some continuous function $f$ on $X$, then every complete $F(yx)$-type over $M$ is definable. Thus, applying Stone-Weierstrass as in the proof of Claim 1 in the proof of Proposition \ref{proposition: universal definable wap ambit} and following the lines of the easy proof of item (i) of this proposition, we get

\begin{cor}
Let ${\mathcal D}$ be the quotient of $S_G(M)$ corresponding to the algebra ${\mathcal B}$ of all CL-formulas $F(x)$ on $G$ over $M$ for which every complete $F(yx)$-type over $M$ is definable. Then $G(M)$ acts naturally on ${\mathcal D}$, and $({\mathcal D},G(M))$ is the universal definable ambit.
\end{cor}

Finally, as promised in the introduction, we give a negative answer to Problem 4.11 (1) from \cite{Kr}, concerning whether the natural map from $S_{G}(M)$ to $G/G^{000}_{M}$, given by  $\tp(g/M) \mapsto g/G^{000}_{M}$,  factors through the universal definable ambit. When $M$ is sufficiently saturated, Corollary \ref{corollary: universal definable ambit in the saturated context} says that the universal definable ambit is precisely the universal definable wap ambit. So we consider, as in the proof of Corollary \ref{corollary: conjecture 0.3 from [KrPi] is false}, a group $G$ definable in a countable NIP theory $T$ such that $G^{000}\neq G^{00}$.  Let $M$ be an $\omega_{1}$-saturated model over which $G$ is defined.  
Then $\mathcal S$, as defined above, is, by Proposition \ref{proposition: universal definable wap ambit},  the universal definable wap ambit of $G(M)$, and likewise, the universal minimal definable wap flow of $G(M)$ is $G/G^{00}_{M}$. 
The natural map $f\colon S_{G}(M) \to G/G^{000}_{M}$ referred to above takes $\tp(g/M)$ to $g/G^{000}_{M}$. 
If $I$ is a minimal subflow of $S_{G}(M)$, then $f[I]=G/G^{000}_{M}$. Indeed, take any $p =\tp(a/M) \in I$. Then $I$ is the closure of $\{ \tp(ga/M): g \in G(M)\}$.  Consider any $b \in G$, and choose $b' \equiv_M b$ with $\tp(b'/Ma)$ finitely satisfiable in $M$. It is easy to see that $\tp(b'a/M) \in I$. Hence, $ba/G^{000}_M = b'a/G^{000}_M \in f[I]$, which is enough.
 
Following earlier notation, let $\pi_{0}$ be the canonical map from $S_{G}(M)$ to $\mathcal S$.  
Then $\pi_{0}[I]$ is a minimal subflow of $\mathcal S$. 

Now, suppose for a contradiction that $f \colon S_{G}(M) \to G/G^{000}_{M}$ factors through $\pi_{0}$, i.e. there is a unique map $f_0 \colon {\mathcal S} \to G/G^{000}_M$ with $f=f_0\circ \pi_0$. Let $f_1 \colon G/G^{000}_M \to G/G^{00}_M$ be the obvious map sending $g/G^{000}_{M}$ to $g/G^{00}_{M}$, and put $f_2:= f_1 \circ f_0 \colon {\mathcal S} \to G/G^{00}_{M}$. All these maps are clearly $G(M)$-equivariant, and $f_2$ is a flow epimorphism. Since $\pi_0[I]$ is a minimal definable wap $G(M)$-flow, and $G/G^{00}_M$ is universal such (and with the property that each flow endomorphism is an automorphism), we get that $f_2|\pi_0[I]$ is an isomorphism. On the other hand, since $f|I$ is surjective, so is $f_0 |\pi_0[I]$. Therefore, $f_1$ is injective, hence $G^{00}_{M} = G^{000}_{M}$, a contradiction.


\section{Approximate subgroups and connected components}\label{Section 4}

As discussed in the introduction, the main goal here is to refute 
Wagner's conjecture on the existence of $H^{00}_\emptyset$ for $H:=\langle X \rangle $, where $X$ is an approximate subgroup\footnote{The definitions of an approximate subgroup 
and of ``existence of $H^{00}_\emptyset$'' are recalled in the introduction, p. \pageref{page: approximate subgroup}.} (in a monster model).
 We also clarify connections between approximate subgroups and the equality $H^{00}_\emptyset = H^{000}_\emptyset$.

\subsection{Connected components and thick sets}\label{section: thick}

We work in a monster model $\C$ of a first order theory $T$,

By a {\em $\bigvee$-definable group} we mean a group $(H,\cdot)$ of the form $\bigcup H_n$, where $(H_n)_{n < \omega}$ is an increasing union of definable sets, and $\cdot \colon H_n \times H_n \to H_{n+1}$ and  $^{-1} \colon H_n \to H_n$ are definable for every $n$. We will be interested in the special case where $G$ is a $\emptyset$-definable group, $X$ a $\emptyset$-definable approximate subgroup, and $H:= \langle X \rangle$; so here we can take $H_n: =X^{2^n}$. 
Note that, by compactness, a definable subset of $H$ is always contained in some power of $X$.

Recall again that if $H$ is definable, then a definable, symmetric subset $D$ of $H$ is called {\em thick} if there is $m$ such that for every $h_0,\dots,h_{m-1}$ there are $i<j<m$ with $h_i^{-1}h_j \in D$. 

If $H$ is a $\emptyset$-definable group, then $H^{000}_\emptyset$ exists and is generated by the intersection of all $\emptyset$-definable thick sets \cite[Lemmas 2.2 and 3.3]{Gis}. Now, we want to generalize it to $\bigvee$-definable groups.


From now on, let $X$ be a $\emptyset$-definable approximate subgroup, and $H:= \langle X \rangle$. In this $\bigvee$-definable context, we modify the definition of thick sets as follows.

\begin{dfn}\label{definition: thick}
A definable, symmetric subset $D$ of $H$ is {\em thick} if for every unbounded sequence $(h_i)_{i<\lambda}$ of elements of $H$ there are $i<j<\lambda$ with $h_i^{-1}h_j \in D$.
\end{dfn}

Using compactness, one can easily show 

\begin{rem}\label{remark: thick} A definable, symmetric subset $D$ of $H$ is thick if and only if  for every $n \geq 1$ there exists $m \in \N$ 
such that for every $h_0,\dots,h_{m-1} \in X^n$ there are $i<j<m$ with $h_i^{-1}h_j \in D$. 
\end{rem}

Using this remark together with finite Ramsey theorem (exactly as in the proof of \cite[Lemma 1.2]{Gis2}), we get that the class of thick subsets of $H$ is closed under finite intersections. Remark \ref{remark: thick} also implies that in Definition \ref{definition: thick}(2) the adjective ``unbounded'' can be replaced by ``uncountable''. 

Recall that a definable subset $D$ of $H$ is {\em generic} if finitely many left translates of $D$ by elements of $H$ cover $X$. 
As $X$ is an approximate subgroup, $D$ being generic is equivalent to any of the conditions:
\begin{enumerate}
\item for every $n$ finitely many left translates of $D$ cover $X^n$,
\item every definable subset of $H$ is covered by finitely many left translates of $D$,
\item  countably many left translates of $D$ cover $H$.
\end{enumerate}
Note that if $D \subseteq H$ is generic, then $D^{-1}D$ is thick. Conversely, if $D\subseteq H$ is thick, then $D$ is generic.

For $a,b \in \C$ we write $a\Theta b$ if the pair $(a,b)$ can be extended to an infinite indiscernible sequence. It is well-known and easy to check that
the transitive closure of $\Theta$ is the finest bounded, invariant equivalence relation on $\C$. This relation is said to be the relation of having the same {\em Lascar strong type} and is denoted by $E_L$. It is also clear that for any invariant subset $Y$ of $\C$, the relation $E_L$ restricted to $Y$ is the finest bounded, invariant equivalence relation on $Y$, and that it is the transitive closure of the relation $\Theta$ restricted to $Y$.

Now, in the context of our $H = \langle X \rangle$, which is clearly invariant, we easily get that $H^{000}_\emptyset$ exists and is exactly the group generated by the set 
\[P:=\{a^{-1}b: a,b \in H \textrm{ and } a\Theta b\}.\]
Namely, we have


\begin{prop}\label{proposition: H000=<P>}
 $H^{000}_\emptyset = \langle P \rangle$
\end{prop}

\begin{proof}
 The group $\langle P \rangle$ is clearly invariant and of bounded index, because the induced relation of lying in the same left coset is coarser than $E_L$ (restricted to $H$). And, on the other hand, for every bounded index, invariant subgroup $K$ of $H$ the relation of lying in the same left coset of $K$ is bounded and invariant so coarser than $E_L$, and hence $a^{-1}b \in K$ for every $a,b \in H$ with $a \Theta b$. Thus, $\langle P \rangle$ is indeed the smallest invariant subgroup of $K$ of bounded index, which is $H^{000}_\emptyset$ by definition.
\end{proof}


\begin{lem}\label{lemma: P=P_X}
$P=   \{a^{-1}b: a,b \in X \textrm{ and } a\Theta b\}$. In particular, $P$ is $\emptyset$-type-definable and symmetric.
\end{lem}

\begin{proof}
Let $P_X:= \{a^{-1}b: a,b \in X \textrm{ and } a\Theta b\}$. The inclusion $P \supseteq P_X$ is obvious. 

For the opposite inclusion, consider any $a,b \in H$ with $a \Theta b$. We need to show that $a^{-1}b \in P_X$. Since $X$ is a generic subset of $H$, we can find a countable $C \subseteq H$ with $CX =H$. Let $(a_i)_{i<\omega}$ be an indiscernible sequence such that $a_0=a$ and $a_1=b$. By a standard application of Ramsey theorem and compactness, we can find a sequence $(a_i')_{i<\omega}$ which is indiscernible over $C$ and has the same type over $\emptyset$ as $(a_i)_{i<\omega}$. 

By indiscernibility over $C$ and the choice of $C$, there are $c \in C$ and an indiscernible sequence $(a_i'')_{i<\omega}$ of elements of $X$ such that $a_i' = c a_i''$ for all $i< \omega$. Then $a_0'^{-1}a_1' =  a_0''^{-1}a_1'' \in P_X$. Since $P_X$ is clearly invariant and $a_0'^{-1}a_1' \equiv a_0^{-1}a_1$, we get that $a_0^{-1}a_1 \in P_X$, i.e. $a^{-1}b \in P_X$.

The ``in particular'' part is now clear, as $P_X$ is easily seen to be $\emptyset$-type-definable and, of course, $P$ is symmetric.
\end{proof}

\begin{prop}\label{proposition: P intersection of thick sets}
$P$ coincides with the intersection of all $\emptyset$-definable, thick subsets of $H$.
\end{prop}

\begin{proof}
$(\subseteq)$ Take $\alpha \in P$, i.e. $\alpha :=a^{-1}b$ for some $a,b \in H$ starting an infinite indiscernible sequence $a_0=a, a_1=b, a_2,\dots$. Consider any  $\emptyset$-definable, thick subset $D$ of $H$. By compactness, we can extend $(a_i)_{i<\omega}$ to an unbounded indiscernible sequence $(a_i)_{i<\kappa}$. Then there are $i<j$ such that $a_i^{-1}a_j \in D$, and so $\alpha= a_0^{-1}a_1 \in D$ by invariance of $D$.

($\supseteq$) By Lemma \ref{lemma: P=P_X}, $P$ is a $\emptyset$-type-definable, symmetric subset of $H$, so it can be written as the intersection of a family $\{D_k\}_{k}$ of $\emptyset$-definable, symmetric subsets of $H$.

Observe that there is no unbounded sequence $(a_i)_{i<\lambda}$ of elements of $H$ such that $a_i^{-1}a_j \notin P$ for all $i<j<\lambda$.  Otherwise, by extracting indiscernibles, there is an unbounded indiscernible sequence $(a_i')_{i<\kappa}$ with $(a_0',a_1') \equiv (a_i,a_j)$ for some $i<j$. Then all $a_i'$ are in $H$ and $a_0'^{-1}a_1' \notin P$, a contradiction with the definition of $P$.  Thus, for every $k$, the set $D_{k}$ is thick. So the intersection of all $\emptyset$-definable, thick subsets of $H$ is contained in $P$.
\end{proof}

For completeness note that whenever $H^{00}_\emptyset$ exists, then $H^{000}_\emptyset \leq H^{00}_\emptyset$.

\subsection{Equivalent conditions}

Let $X$ be a $\emptyset$-definable approximate subgroup (in a $\emptyset$-definable group $G$), everything in a monster model. As before, $H := \langle X \rangle $. Let $m \geq 1$.
We are interested in the following conditions:

\[ (\diamond) \;\;\;\;\; H^{000}_\emptyset = H^{00}_\emptyset,\]
\[ (\diamond \diamond)_m \;\;\;\;\; H^{00}_\emptyset \subseteq Y^m \textrm{ for every $\emptyset$-definable, generic, symmetric subset $Y$ of $H$},\]
\[  (\diamond \diamond \diamond)_m \;\;\;\;\; H^{000}_\emptyset \subseteq Y^m \textrm{ for every $\emptyset$-definable, generic, symmetric subset $Y$ of $H$}.\]

The property $(\exists m) (\diamond \diamond)_m$ was crucial in \cite{MaWa} to find a ``locally compact model'' for $X$ under the definable amenability assumption, and in \cite{KrPi} as well as in Section \ref{Section 2} of this paper to prove the appropriate variants of the equality $H^{000}_\emptyset =H^{00}_\emptyset$ in the case when $H$ is definable and satisfies various kinds of amenability assumptions.

Here, we prove

\begin{prop}\label{proposition: three equivalent conditions}
The following conditions are equivalent:
\begin{enumerate}
\item  $H^{000}_\emptyset = H^{00}_\emptyset$,
\item $(\exists m) (\diamond \diamond)_m$,
\item  $(\exists m) (\diamond \diamond \diamond)_m$.
\end{enumerate}
\end{prop}

\begin{proof}
(1) $\rightarrow$ (2). Assume  that $H^{000}_\emptyset = H^{00}_\emptyset$, so $H^{000}_\emptyset$ is type-definable. On the other hand, by Propositions \ref{proposition: H000=<P>} and \ref{proposition: P intersection of thick sets}, we know that $H^{000}_\emptyset = \langle P \rangle$ and $P$ is type-definable and symmetric. Therefore, by \cite[Theorem 3.1]{Ne}, we get that $H^{000}_\emptyset =P^k$ for some $k \geq 1$. We will show that $(\diamond \diamond)_{2k}$ holds.

So take any $\emptyset$-definable, generic, symmetric subset $Y$ of $H$. Then $Y^2$ is thick. Hence, $P \subseteq Y^2$ by Proposition \ref{proposition: P intersection of thick sets}. Thus, $P^k \subseteq Y^{2k}$, i.e. $H^{00}_\emptyset \subseteq Y^{2k}$ by the previous paragraph. 

(2) $\rightarrow$ (3) is trivial.

(3) $\rightarrow$ (1). 
Let $m \geq 1$ be such that $(\diamond \diamond \diamond)_m$ holds. Then, by Proposition \ref{proposition: P intersection of thick sets}, compactness, and the fact that the class of thick subsets of $H$ is closed under finite intersections, we get  $H^{000}_\emptyset \subseteq P^m$. On the other hand, by Proposition \ref{proposition: H000=<P>},  $H^{000}_\emptyset$ is the union of the $P^n$'s. Thus, we conclude that $H^{000}_\emptyset =P^m$, so $H^{000}_\emptyset$ is type-definable and hence equals $H^{00}_\emptyset$.
\end{proof}

So the conditions $(\exists m)(\diamond \diamond)_m$ and $(\exists m)(\diamond \diamond \diamond)_m$ fail in any example where $H^{000}_\emptyset \ne H^{00}_\emptyset$. As mentioned in the introduction, \cite{Gismatullin-Krupinski} yields many such examples with definable $H$, e.g. $X=H$ being a sufficiently saturated extension of the free group $\F_2$ considered with the full structure (i.e. with predicates for all subsets of all finite Cartesian powers).  This already shows that there are $\emptyset$-definable approximate subgroups $Y_m$, $m \in \N$, of the monster model $\F_2^*$ (which are generic subsets of $\F_2^*$ so generate $\emptyset$-definable, finite index subgroups in finitely many steps) such that $(H_m)^{00}_\emptyset=(\F_2^*)^{00}_\emptyset \nsubseteq Y_m^m$, where $H_m: = \langle Y_m \rangle$. On the other hand, under the definable amenability assumption, the above equivalent conditions hold (as $(\diamond \diamond)_8$ holds by \cite{MaWa} and \cite[Theorem 5.2]{Mas}). 

The above equivalent conditions imply in particular that $H^{00}_\emptyset$ exists.  As mentioned in the introduction, if $H^{00}_\emptyset$ exists, then $H^{00}_\emptyset \subseteq X^m$ (so also  $H^{000}_\emptyset \subseteq X^m$) for some $m \geq 1$.

\begin{rem}
If $H^{00}_\emptyset$ exists, then for every $\emptyset$-definable, generic, symmetric subset $Y$ of $H$ and $H':= \langle Y \rangle$, the component $H'^{00}_\emptyset$ also exists and coincides with $H^{00}_\emptyset$.
\end{rem}

\begin{proof}
$H'$ is the increasing union $\bigcup_{n \geq 1} Y^n$ and it has bounded index in $H$. Hence, $H^{00}_\emptyset \cap H'$ is the  increasing union $\bigcup_{n \geq 1} H^{00}_\emptyset \cap Y^n$ and the index $[H^{00}_\emptyset : H^{00}_\emptyset \cap H']$ is bounded.
Therefore, both $H^{00}_\emptyset \cap H'$ and $H^{00}_\emptyset \setminus H'$ are $\bigvee$-relatively definable subsets of $H^{00}_\emptyset$, and so, by compactness, they are relatively definable in $H^{00}_\emptyset$. In particular, $H^{00}_\emptyset \cap H'$ is $\emptyset$-type-definable and of bounded index in $H'$ so also in $H$. Hence, $H^{00}_\emptyset \cap H' = H^{00}_{\emptyset}$ does not have proper $\emptyset$-type-definable subgroups of bounded index. Thus,  $H^{00}_\emptyset = H'^{00}_\emptyset$.
\end{proof}

\begin{cor}\label{corollary: H00 exists implies its is contained in Ym}
If $H^{00}_\emptyset$ exists, then for every for every $\emptyset$-definable, generic, symmetric subset $Y$ of $H$ there exists $m \geq 1$ such that $H^{00}_\emptyset \subseteq Y^m$ (so also  $H^{000}_\emptyset \subseteq Y^m$).
\end{cor}

We leave as a question whether the converse holds, i.e.

\begin{ques}
 Assume that for every $\emptyset$-definable, generic, symmetric subset $Y$ of $H$ there exists $m \geq 1$ such that $H^{000}_\emptyset \subseteq Y^m$. Does it imply that $H^{00}_\emptyset$ exists?
\end{ques}

\subsection{A counter-example}\label{section: counter-example}

In this subsection, we we give an example of a $\emptyset$-definable approximate subgroup $X$ for which $H^{00}_\emptyset$ does not exist (where $H = \langle X \rangle$).

A function $f \colon G \to A$ from a group $G$ to an abelian group $A$ is called a {\em quasi-homomorphism} if the image of the function $f'(x,y): = f(x) + f(y) -f(xy)$ is finite. Whenever $f$ is a quasi-homomorphism satisfying $f(g^{-1})=-f(g)$ for all $g \in G$, then the graph of $f$ is an approximate subgroup in $G \times A$, because it is symmetric and $\graph(f)^2 \subseteq (\{e\} \times \im(f')) \cdot \graph(f)$.

Let $\F_2$ be the free group with the free generators $a$ and $b$.  Let $\alpha, \beta \in l^{\infty}_{\odd}(\Z)$, where $l^{\infty}_{\odd}(\Z)$ is the set of bounded, odd, integer valued functions on $\Z$. Let $f \colon \F_2 \to \Z$ be given by 
\[ f(a^{n_1}b^{m_1} \dots a^{n_k}b^{m_k}):= \sum_{i=1}^k \alpha(n_i) + \beta(m_i),\]
where $n_i \ne 0$ for $1<i \leq k$ and $m_i \ne 0$ for $1 \leq i <k$. By \cite[Proposition 4.1]{Ro}, $f$ is a quasi-homomorphism. Let us choose $\alpha$ which is non-zero for almost all arguments $n \in \Z$. For example, one can take $\alpha =\beta := \sgn$, where $\sgn(n)$ is $1$ if $n$ is postive, $-1$ if $n$ is negative, and $0$ if $n=0$; then the fact that the resulting $f$ is a quasi-homomorphism follows from \cite[Proposition 2.1]{Ro}. Let $F$ be the graph of $f$. It is an approximate subgroup in $\F_2 \times \Z$.

Let $M$ be the two-sorted structure with the sorts being the pure groups $(\F_2, \cdot)$ and $(\Z,+)$ together with the above function $f: \F_2 \to \Z$ (or any expansion of this structure). Then $F$ is $\emptyset$-definable in $M$. We will be working in a monster model $M^* \succ M$. Let $F^*:=F(M^*)$ (i.e. the interpretation of $F$ in $M^*$), and  similarly we have $\F_2^*$, $\Z^*$, and $f^*$. Then $f^*\colon \F_2^* \to \Z^*$ is clearly a quasi-homomorphism for which ${f^*}'$ has the same finite image as $f'$, and $F^*$ is the graph of $f^*$ and so a $\emptyset$-definable approximate subgroup in $\F_2^* \times \Z^*$. To be consistent with the notation from previous subsections, let us denote $F^*$ by $X$, and put $H:=\langle X \rangle$. 
We will be using the set $P$ defined and considered in Subsection \ref{section: thick}.

\begin{prop}\label{proposition: H000 not contained in Xm}
For every $m \geq 1$, $H^{000}_\emptyset$ is not contained in $X^m$.
\end{prop}

\begin{proof}
Suppose for a contradiction that $H^{000}_\emptyset \subseteq X^m$ for some $m \geq 1$. 

For $B \subseteq \Z$ and $n \in \N$, let $B^{+n}:= B + \dots +B$ be the $n$-fold sum. Note that
\[ (*)\;\;\:\;\; X^m \subseteq \left(\{e\} \times \im(f')^{+(m-1)}\right)X.\]

Consider any $\emptyset$-definable, thick subset $D$ of $H$. 
Let $\pi_1: \F_2^* \times \Z^* \to \F_2^*$ be the projection. 
By thickness of $D$ and Remark \ref{remark: thick}, there are infinitely many $k \in \N$ with $a^k \in \pi_1[D]$. 
So, by the definition of $f$ and the assumption on $\alpha$, we can find $k \geq 1$ such that
\[ a^k \in \pi_1[D] \; \textrm{ and } \; f(a^k) \ne 0 \; \textrm{ and } \; f((a^k)^n) \in \im(\alpha) \textrm{ for all } n \in \Z.\]
Similarly, we can find $l \ne 0$ such that
\[ b^l \in \pi_1[D] \; \textrm{ and }\;  f((b^l)^n) \in \im(\beta) \textrm{ for all } n \in \Z.\]
Since $f(a^kb^l) \ne 0$ or $f(a^kb^{-l}) \ne 0$, replacing $l$ by $-l$ if necessary, we can assume that
\[ f(a^kb^l) \ne 0.\] 
By the definition of $f$ and the fact that $k,l \ne 0$, we also have
\[ f((a^kb^l)^n) = nf(a^kb^l) \textrm{ for all } n \in \Z.\]

Using the last four displayed conditions, the assumption that $\im(\alpha)$ and $\im(\beta)$ are finite, Proposition \ref{proposition: P intersection of thick sets}, the fact that the class of thick subsets of $H$ is closed under finite intersections, and compactness, we can find $s,t \in \F_2^*$ and $\epsilon, \delta \in \Z^*$ such that:
\[ 
(**) \;\;\;\;\;  \left\{ \begin{array}{ll}
(s,\epsilon) \in P \; \textrm{ and } \; f^*(s) \ne 0 \; \textrm{ and } \; f^*(s^n) \in \im(\alpha) \textrm{ for all } n \in \Z,\\
(t,\delta) \in P \; \textrm{ and } \; f^*(t^n) \in \im(\beta) \textrm{ for all } n \in \Z,\\
f^*(st) \ne 0 \; \textrm{ and } \; f^*((st)^n) = nf^*(st)  \textrm{ for all } n \in \Z.
\end{array} \right.
\]

Since by Proposition \ref{proposition: H000=<P>} we have $H^{000}_\emptyset = \langle P \rangle$, we get that $(s,\epsilon)$ and $(t,\delta)$ are in $H^{000}_\emptyset$, and so are all their integer powers. Thus, by the first sentence of the proof,  we have $(s,\epsilon)^n, (t,\delta)^n \in X^m$ for all $n \in \Z$. As $(s,\epsilon)^n = (s^n,n\epsilon)$ and  $(t,\delta)^n =(t^n,n\delta)$, using $(*)$ and $(**)$, we conclude that $n\epsilon \in \im(f')^{+(m-1)} +f^*(s^n) \subseteq  \im(f')^{+(m-1)} + \im(\alpha)$ and $n\delta \in \im(f')^{+(m-1)} +f^*(t^n) \subseteq  \im(f')^{+(m-1)} + \im(\beta)$. Since this holds for all $n \in \Z$, and the right hand sides of both inclusions are finite, we conclude that $\epsilon = \delta =0$.

Thus, $(st,0) =(s,0)(t,0) \in H^{000}_\emptyset$, and so $((st)^n,0) = (st,0)^n \in H^{000}_\emptyset \subseteq X^m$ for all $n \in \Z$. Hence, using $(*)$ and $(**)$, we get that $0 \in \im(f')^{+(m-1)} + f^*((st)^n) = \im(f')^{+(m-1)} + nf^*(st)$, and so  $nf^*(st) \in - \im(f')^{+(m-1)}$, for all $n \in \Z$. But from $(**)$ we also know that $f^*(st) \ne 0$. Since  $\im(f')^{+(m-1)}$ is finite, we get a contradiction.
\end{proof}

Proposition \ref{proposition: H000 not contained in Xm} implies the main result of Section \ref{Section 4}, which refutes Wagner's conjecture.

\begin{cor}
$H^{00}_\emptyset$ does not exist.
\end{cor}

As mentioned in the introduction, by \cite[Corollary 5.14]{Mas}, this implies that $H^{00}_C$ does not exist for any small set of parameters $C$ and even working with any expansion of the structure $M$. But in our particular example this also follows from our proof above, where the assumption that $C=\emptyset$ and the choice of the structure expanding $M$ are irrelevant.

\section{Final remarks}\label{section: final remarks}

\subsection{Connected components and approximate subgroups}

Let us work here over a small model $M$ (as a set of parameters), i.e. definability is over $M$. Let $G$ be a definable group, and $X$ a definable approximate subgroup; put $H:= \langle X \rangle$. In this context, and under various auxiliary amenability-type hypotheses, one proves the ``stabilizer theorem''

\[ (\dagger) \\\\\\\  H^{00}_M\subseteq X^{4}.\]

This leads to a connection with locally compact groups $L$, and through them Lie groups. (See \cite{Hrushovski-approximate}, \cite{Montenegro-Onshuus-Simon}, \cite{Sanders}, and most relevant to us \cite{MaWa}.)  In Section \ref{Section 4}, we gave an example showing that in general $(\dagger)$ may drastically fail, namely $H^{00}_M$ need not exist, which refutes Wagner's conjecture (i.e. \cite[Conjecture 0.15]{Mas}). In \cite{MaWa}, Massicot and Wagner conjectured that ``even without the definable amenability assumption a suitable Lie model exists''. Our counter-example shows that in order to find such a model, one should use different methods (i.e. not involving $H^{00}_M$).

In Sections \ref{Section 2} and \ref{Section 3} of this paper, we have restricted to the case where the $\bigvee$-definable group $H$ is actually definable.
In this case, the locally compact group $L$ is compact. This case is not ruled out as trivial, and indeed is of considerable interest; for instance some of the first theorems in this line, by Gowers and Helfgott, asserted in effect that generic definable subsets of certain pseudofinite groups generated the group in boundedly many steps (3 or 4), and were in turn important in further developments by Bourgain-Gamburd and many others. 

Let us briefly discuss a connection between $(\dagger)$ (or rather $(\exists m )(\diamond \diamond)_m$ from Section \ref{Section 4}) with the main results of this paper. In the proof of Theorems \ref{theorem: c4}, \ref{theorem: indeed c4}, and \ref{proposition: normal2}, we showed that under an appropriate amenability assumption and for suitable sets $X$ we have $G^{00}_M \subseteq X^4$, from which we easily deduced the appropriate variants of the equality $G^{00}_M = G^{000}_M$. This general idea was first used in \cite{KrPi} under the full definable amenability assumption of $G$ (where $G^{00}_M \subseteq X^4$ for any definable, generic and symmetric $X$ was easily deduced from \cite{MaWa}). Also in \cite{KrPi}, this idea was extended to some situations when an invariant probability measure is available only on a suitable subalgebra of the Boolean algebra of all definable subsets of $G$. A similar observation (i.e. a version of the ``stabilizer theorem'' assuming a measure on a suitable subalgebra) appeared also  in \cite{Mas}. In order to prove the above main results of this paper, we essentially extended this idea. This required a more precise version of the ``stabilizer theorem'', using $\bigvee$-positively definable sets. The key application of positive definability appears in the last paragraph of the proof of Theorem \ref{theorem: c4}, where we used positive $M$-definability in $(G, \cdot , E)$ of a certain set $Y_\nu$ produced by our version of the ``stabilizer theorem'', where $E$ is $M$-type-definable in the original structure, to deduce that $Y_\nu$ is $M$-type-definable in the original structure. Alternatively, to prove the implication (4) $\rightarrow$ (5) in Theorem \ref{theorem: c4}, one could use the version of the ``stabilizer theorem'' established in \cite[Corollary 4.14]{Mas}, and then, instead of positive definability, use \cite[Theorem 5.2]{Mas} whose proof involves a variant of Beth's theorem and a new version of Schlichting's theorem. 
The two approaches (i.e. Massicot's  based on Schlichting's theorem and ours using positive definability) are different and were developed independently.\\

\subsection{Connected components and complexity}

Let us consider these notions from the point of view of descriptive set theory (see for example \cite{Moschovakis} for the terms below.) Fix a countable language $L$ with distinguished sort $G$ (with a binary operation), and consider the space of complete theories $T$ (with $G$ a group).
For now, $G^{000}$ etc. will mean $G^{000}_{\emptyset}$ etc.   

The condition $G = G^{000}$ is at the finite level of the Borel hierarchy (``arithmetic''), and is  in fact a countable union of closed sets. 
This can be seen as follows. 
First, it is known that $G = G^{000}$ is equivalent to  $P^{n} = G$ for some $n$ (see \cite[Theorem 3.1]{Ne}), where $P:=\{a^{-1}b: a,b \in G \textrm{ and } a\Theta b\}$, and $a \Theta b$ means that $a,b$ start an infinite indiscernible sequence (as in Subsection \ref{section: thick}).  We can now unwind the statement $P^{n} = G$, using compactness: for any approximation $I_{k}$ to indiscernibility, 
$$T\models (\forall x)(\exists y_{i1}, y_{i2})_{i\leq n} \left( x = y_{11}y_{12}^{-1} \dots y_{n1}y_{n2}^{-1}\wedge \bigwedge_{i}I_{k}(y_{i1},y_{i2})\right).$$


The condition  $G^{00}=G^{000}$ is also Borel. Namely, $G^{00}=G^{000}$ is equivalent to saying that  for some $n$, $P^{n+1} = P^n$ (by \cite[Theorem 3.1]{Ne} and the fact that the sequence $(P^n)_n$ is ascending). The last equality can be expressed by: for every $k$ there is $l$ such that 
\begin{align*}
T\models \left( (\exists z_{i1}, z_{i2})_{i\leq n+1} \left( x = z_{11}z_{12}^{-1} \dots z_{(n+1)1}z_{(n+1)2}^{-1}\wedge \bigwedge_{i}I_{l}(z_{i1},z_{i2})\right)\right) \longrightarrow\\
 \left( (\exists y_{i1}, y_{i2})_{i\leq n} \left( x = y_{11}y_{12}^{-1} \dots y_{n1}y_{n2}^{-1}\wedge \bigwedge_{i}I_{k}(y_{i1},y_{i2})\right)\right).
\end{align*}

The equality $G=G^0$  requires only one integer quantifier:  for all formulas $\phi(x)$, and each $n>1$,
the sentence ``$\phi$ defines a subgroup of index $n$'' is false in $T$.

The condition  $G=G^{00}$ is more mysterious.  It is clearly at worst $\Pi^1_1$, as it can be expressed as the non-existence of a chain of proper, generic, definable subsets $D_i \subseteq G$ with $D_{i+1}^{-1}D_{i+1} \subseteq D_i$.   
 This together with the previous paragraphs suffices to  see that the various connectedness properties (equalities among the various connected components) are all  properties of $T$ itself and do not depend on the ambient model of set theory.   But it remains quite interesting to know if such a chain, when it exists, 
  can be constructed in some explicit way.  In particular, is the condition $G=G^{00}$ Borel?   

If $G$ is definably amenable and we compute everything over parameters from a given model, then, by \cite[Theorem 12]{MaWa} and the short argument in Case 2 on page 1282 in \cite{KrPi}, one easily gets that the condition $G=G^{00}$ is equivalent to the (Borel) statement that for any generic, symmetric, definable set $X$, $X^4=G$ (the method from \cite{MaWa} concretely constructs a sequence of generic, symmetric, definable subsets $D_i$ with $D_{i+1}D_{i+1} \subseteq D_i \subseteq X^4$). However, in general, the latter condition may be strictly stronger, which is always the case when $G=G^{00} \ne G^{000}$ (as $G \ne G^{000}$ implies $G \ne P^n$ for every $n$). 
So the question whether $G^{00}=G$ is Borel remains open and quite interesting. 
 

\section*{Acknowledgments}
	
We would like to thank the referee for careful reading and all the comments and suggestions.


\begin{thebibliography}{99}

\bibitem{Auslander} J. Auslander, {\em Minimal flows and their extensions}, North Holland, 1988.

\bibitem{B-BY-H-U} A. Berenstein, I. Ben Yaacov, C. W. Henson, A. Usvyatsov, {\em Model theory for metric structures}, In: {\em Model theory and applications to analysis and algebra, vol II}, CUP, 2008. 

\bibitem{BrGrTa} E. Breuillard, B. Green, T. Tao, {\em The structure of approximate subgroups}, Publ. Math. Inst. Hautes \'{E}tudes Sci. 116 (2012), 115-221.

\bibitem{BY-U} I. Ben Yaacov and  A. Usvyatsov, {\em Continuous first order logic and local stability},  Trans. Amer. Math. Soc. 362 (2010), 5213-5269.

\bibitem{BY-T} I Ben Yaacov and T. Tsankov, {\em Weakly almost periodic functions, model-theoretic stability, and minimality of topological groups},  Trans. Amer. Math. Soc. 368 (2016), 8267-8294.

\bibitem{BY} I. Ben Yaacov, {\em Model theoretic stability and definability of types, after A. Grothendieck}, Bull. Symbolic Logic 20 (2014), 491-496. 


\bibitem{CK} C. C. Chang  and H. J. Keisler,  {\em Continuous model theory}, Princeton University Press, 1966. 

\bibitem{CP} A. Conversano and A. Pillay, {\em Connected components of definable groups and o-minimality I}, Adv. Math. 231 (2012), 605-623.

\bibitem{Eberlein} F. Eberlein, {\em Abstract ergodic theorems and weak almost periodic functions}, Trans. Amer. Math. Soc. 67 (1949), 217-240.

\bibitem{Ellis-Nerurkar} R. Ellis and M. Nerurkar, {\em Weakly almost periodic flows}, Trans. Amer. Math. Soc. 313 (1989), 103-119.

\bibitem{Gis} J. Gismatullin, {\em Model theoretic connected components of groups}, Israel J. Math. 184 (2011), 251-274.

\bibitem{Gis2} J. Gismatullin, {\em Absolute connectedness and classical groups}, preprint (2010). arXiv:1002.1516

\bibitem{Gismatullin-Krupinski} J. Gismatullin and K. Krupi\'{n}ski, {\em On model-theoretic connected components in some group extensions}, J. Math. Log. 15 (2015), 1550009 (51 pages).

\bibitem{GPP} J. Gismatullin, D. Penazzi, A. Pillay, {\em On compactifications and the topological dynamics of definable groups}, Ann. Pure Appl. Logic 165 (2014), 552-562.

\bibitem{Glasner} E. Glasner, {\em Proximal flows}, Lecture Notes in Mathematics 517, Springer Verlag, 1976. 

\bibitem{Grothendieck} A. Grothendieck, {\em Crit\`{e}re de compacit\'{e} dans les espaces fonctionneles generaux}, Amer. J. Math.  74 (1952), 168-186.

\bibitem{Hrushovski-approximate} E. Hrushovski, {\em Stable group theory and approximate subgroups}, J. Amer. Math. Soc. 25 (2012), 189-243.

\bibitem{HKP2} E. Hrushovski, K. Krupi\'{n}ski, A. Pillay, {\em On first order amenability}, preprint.

\bibitem{HrPePi} E. Hrushovski, Y. Peterzil, A. Pillay,  {\em Groups, measures and the NIP}, J. Amer. Math. Soc. 21 (2008), 563-596.


\bibitem{Ibarlucia} T. Ibarlucia, {\em The dynamical hierarchy for Roelcke precompact Polish groups}, Israel J. Math. 215 (2016), 965-1009.





\bibitem{KaMiSi} I. Kaplan, B. Miller, P. Simon. {\em The Borel cardinality of Lascar strong types}, J. London Math. Soc. 90 (2014), 609-630.

\bibitem{Kr} K. Krupi\'{n}ski, {\em Definable topological dynamics},  J. Symb. Logic 82 (2017), 1080-1105.



\bibitem{KrPi0} K. Krupi\'{n}ski, A. Pillay, {\em Generalized Bohr compactification and model-theoretic connected components}, Math. Proc. Cambridge Philos. Soc. 163 (2017), 219-249. 

\bibitem{KrPi}  K. Krupi\'{n}ski, A. Pillay, {\em Amenability, definable groups, and automorphism groups}, Adv. Math. 345 (2019), 1253-1299.

\bibitem{KrPiRz} K. Krupi\'{n}ski, A. Pillay, T. Rzepecki, {\em Topological dynamics and the complexity of strong types}, Isr. J. Math. 228 (2018), 863-932.

\bibitem{KrPiSo} K. Krupi\'{n}ski, A. Pillay, S. Solecki, {\em Borel equivalence relations and Lascar strong types}, J. Math. Logic 13 (2013), p. 1350008.

\bibitem{KrRz} K. Krupi\'{n}ski, T. Rzepecki, {\em Galois groups as quotients of Polish groups}, J. Math. Logic 20 (2020), p. 2050018.

\bibitem{Mas} J.-C. Massicot, {\em Approximate subgroups and Model Theory}. Ph.D. thesis, Universit{\'e} de Lyon, 2018.

\bibitem{MaWa} J.-C. Massicot, F. O. Wagner {\em Approximate subgroups}, Journal de l'\'{E}cole Polytechnique 2 (2015), 55-63.

\bibitem{Montenegro-Onshuus-Simon} S. Montenegro, A. Onshuus, and P. Simon, {\em Stabilizers, $NTP_{2}$ groups with $f$-generic, and PRC fields}, J. Inst. Math. Jussieu 19 (2020), 821 - 853. 

\bibitem{Moschovakis} Y. Moschovakis, {\em Descriptive Set Theory}, 2nd edition, AMS, 2009. 

\bibitem{Ne} L. Newelski, {\em The diameter of a Lascar strong type}, Fund. Math. 176 (2003), 157-170.

\bibitem{Pillay} A. Pillay, {\em Generic stability and Grothendieck}, South American Journal of Logic 2 (2016), 437-442.

\bibitem{Ro} P. Rolli, {\em Quasi-morphisms on free groups}, preprint. arXiv:0911.4234

\bibitem{Sanders} T. Sanders, {\em On a non-abelian Balog-Szemeredi-type lemma}, J. Aust. Math. Soc. 89 (2010), 127-132. 


\end{thebibliography}
\end{document}